\documentclass[11pt]{amsart}

\usepackage{amssymb}
\usepackage{enumitem}
\usepackage{mathrsfs}
\usepackage{hyperref}
\usepackage{graphicx}

\renewcommand{\setminus}{\smallsetminus}
\renewcommand{\subset}{\subseteq}

\DeclareMathOperator{\R}{R}
\DeclareMathOperator{\grad}{grad}
\DeclareMathOperator{\Hess}{Hess}
\DeclareMathOperator{\Div}{div}
\DeclareMathOperator{\tr}{tr}

\DeclareMathOperator{\vol}{vol}

\newcommand{\Defn}[1]{{\boldmath\it\bfseries #1}}

\newtheorem{thm}{Theorem}[section]
\newtheorem{prop}[thm]{Proposition}
\newtheorem{cor}[thm]{Corollary}
\newtheorem{lemma}[thm]{Lemma}

\renewcommand{\epsilon}{\varepsilon}

\numberwithin{equation}{section} 

\title[Asymptotic gluing of shear-free data]{Asymptotic gluing of shear-free hyperboloidal initial data sets}
\author[Allen {\it et al.}]{Paul T.~Allen,\\ James Isenberg, \\John M.~Lee,\\ Iva Stavrov Allen}

\thanks{This work is partially supported by National Science Foundation grants DMS 1263431 and PHY 1707427 at the University of Oregon.}

\email{ptallen@lclark.edu}

\email{johnmlee@uw.edu}

\email{isenberg@uoregon.edu}

\email{istavrov@lclark.edu}

\address{Department of Mathematical Sciences, Lewis \& Clark College}

\address{Department of Mathematics, University of Washington}

\address{Department of Mathematics, University of Oregon}

\address{Department of Mathematical Sciences, Lewis \& Clark College}


\begin{document}

\maketitle

\begin{abstract}
We present a procedure for asymptotic gluing of hyperboloidal initial data sets for the Einstein field equations that preserves the shear-free condition.
Our construction is modeled on the gluing construction in \cite{ILS-Gluing}, but with significant modifications that incorporate the shear-free condition.
We rely on the special H\"older spaces, and the corresponding theory for elliptic operators on weakly asymptotically hyperbolic manifolds, introduced by the authors in \cite{WAH} and applied to the Einstein constraint equations in \cite{AHEM}.

\end{abstract}

%\tableofcontents

\section{Introduction}
\label{introduction}

One of the most useful ways to mathematically define asymptotically flat spacetimes -- solutions to the Einstein field equations that model isolated gravitational systems -- is to require that they admit a conformal compactification; see \cite{Penrose-AsymptoticBehavior}, \cite{HawkingEllis}.
Such spacetimes can be foliated by spacelike leaves intersecting the conformal boundary along future null infinity.
The intrinsic and extrinsic geometry induced on such a leaf comprises a solution to the Einstein constraint equations, commonly referred to as {\it hyperboloidal} in the literature.

In \cite{AHEM} the authors have constructed constant-mean-curvature hyperboloidal solutions to the Einstein constraint equations satisfying a boundary condition, known as the {\it shear free condition}, along the conformal boundary.
Being shear-free is a necessary condition on the initial data set for
any spacetime development of that data to admit a regular conformal structure at future null infinity; see \cite{AnderssonChrusciel-Obstructions}.

In this paper we present an asymptotic gluing procedure for vacuum constant-mean-curvature shear-free hyperboloidal initial data as constructed in \cite{AHEM}. 
Previous gluing constructions for the solutions to the Einstein constraint equations with asymptotically hyperbolic geometry \cite{ChruscielDelay-Exotic},\cite{ILS-Gluing} have not accounted for the shear-free condition.

Topologically the gluing construction produces a connected sum of the conformal boundary.
While our construction is independent of the topological type of the boundary, we note the important special case that the original data consists of two connected components, each having a spherical conformal boundary; see Figure \ref{fig:bdy-glue}.
In this case, our construction yields ``two-body'' initial data sets having a spherical conformal boundary.

\begin{figure}[h]
   \centering
   \includegraphics[width=0.9\textwidth]{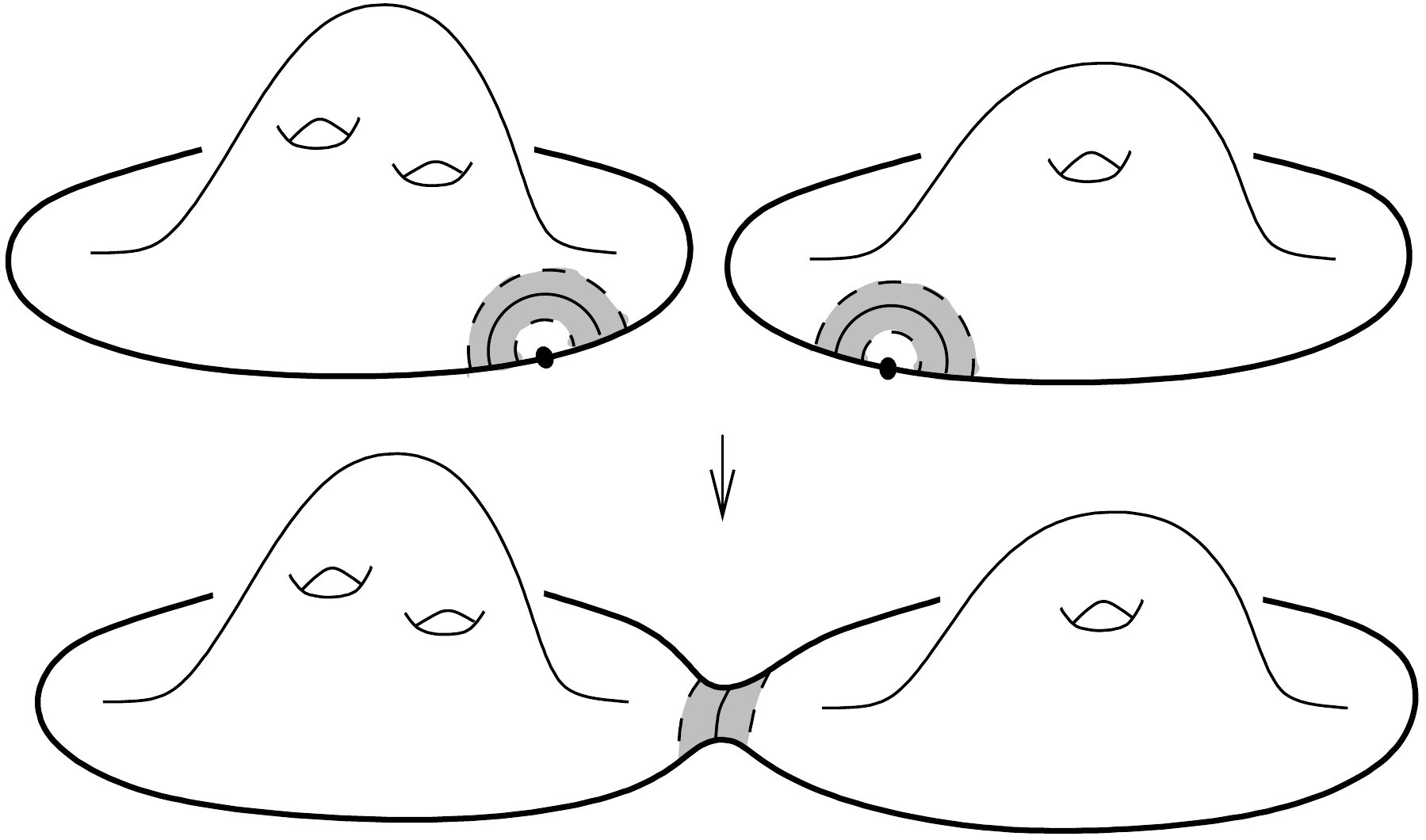} 
   \caption{This diagram, adapted from \cite{ILS-Gluing}, shows the boundary gluing construction in the case of two connected components.}
   \label{fig:bdy-glue}
\end{figure}

Our construction produces a one-parameter family of shear-free hyperboloidal initial data sets.
We are able to show, in the limit as the parameter tends to zero, that the geometry converges to that of the original data set.
Furthermore, the geometry in the center of the gluing region converges to a portion of the hyperboloid inside the Minkowski spacetime.

\subsection{Constant-mean-curvature hyperboloidal data}
We now give a definition of the hyperboloidal data to which our result applies.
This type of initial data has been discussed extensively in \cite{AHEM}, which in turn relies heavily on \cite{WAH}. 

First, we briefly review the definitions of the function spaces we work with; see \cite{WAH}, \cite{Lee-FredholmOperators}, and \S\ref{sec:function-spaces} below for more details.
We assume that $M$ is the interior of a compact 3-dimensional manifold $\overline M$ having boundary $\partial M$ and let $\rho$ be a smooth defining function on $\overline M$ (meaning $\rho$ vanishes to first order on $\partial M$ and is positive in $M$).
Let $C^{k,\alpha}(M)$ be the intrinsic H\"older space of tensor fields on $M$ and for $\delta\in \mathbb R$ let $C^{k,\alpha}_\delta(M) = \rho^\delta C^{k,\alpha}(M)$.
A covariant $2$-tensor field $u$ is defined to be of class $\mathscr C^{k,\alpha;m}(M)$ if
\begin{equation*}
\mathscr L_{X_1}\dots \mathscr L_{X_j}u \in C^{k-j,\alpha}_2(M)
\end{equation*}
for all $0\leq j \leq m$ and for all smooth vector fields $X_1,\dots, X_j$ on $\overline M$; here $\mathscr L$ denotes the Lie derivative.
For example, if $u\in C^{m,\alpha}(\overline M)$ then $u\in \mathscr C^{m,\alpha;m}(M)$.

A complete metric $g$ and a symmetric covariant 2-tensor $K$ (representing the second fundamental form) form a \Defn{constant-mean-curvature shear-free (CMCSF) hyperboloidal data set of class $\mathscr C^{k,\alpha;2}$} on $M$ if
\begin{enumerate}
\item
\label{define-wah} $g = \rho^{-2}\overline g$ for some $\overline g  \in \mathscr C^{k,\alpha;2}(M)$ that extends to a metric on $\overline M$ and is such that $|d\rho|_{\overline g} = 1+O(\rho)$;

\item $K = \Sigma -g$ for some traceless tensor $\Sigma = \rho^{-1}\overline\Sigma$ with $\overline\Sigma  \in \mathscr C^{k-1,\alpha;1}(M)$;

\item the shear-free condition holds, meaning that
\begin{equation}
\label{define-SF}
\overline\Sigma\Big|_{\rho=0} = \left[\Hess_{\overline g}\rho - \frac13(\Delta_{\overline g}\rho)\overline g \right]_{\rho=0};
\end{equation}
and

\item the vacuum constraint equations hold, meaning that
\begin{equation}
\label{CMC-constraints}
\R[g] + 6 - |\Sigma|^2_g  =0
\quad\text{ and }\quad
\Div_g\Sigma =0,
\end{equation}
where $\R[g]$ is the scalar curvature of $g$.
\end{enumerate}
A metric $g$ satisfying condition \eqref{define-wah} is said to be \Defn{weakly asymptotically hyperbolic of class $\mathscr C^{k,\alpha;2}$}.

An important example of CMCSF hyperboloidal data is the data induced on the unit hyperboloid in the Minkowski spacetime, given in the usual Cartesian coordinates by $\{ (x^0)^2 = (x^1)^2 + (x^2)^2 + (x^3)^2 + 1, x^0>0\}$.
The induced metric for this example is the hyperbolic metric $\breve g$, while the second fundamental form is given by $K=-\breve g$; thus $\Sigma =0$ for this data.

\subsection{Statement of the main result}
The asymptotic gluing procedure of \cite{ILS-Gluing} produces data which generically fails to be shear-free (cf.~Proposition 3.2 of \cite{AnderssonChrusciel-Obstructions}). 
Here we present a modification of the gluing method of \cite{ILS-Gluing} within the category of shear-free initial data. 
We now give a precise statement of our result.

\begin{thm}
\label{main}
Suppose that a metric $g$ and a tensor field $\Sigma$ give rise to a CMCSF hyperboloidal data set of class $\mathscr C^{k,\alpha;2}$ on $M$ for some $k\geq 3$ and $\alpha \in (0,1)$.
Fixing $p_1, p_2\in \partial M$, we define for each sufficiently small $\epsilon>0$ a manifold $ M_\epsilon$, which is the interior of a compact manifold $\overline M_\epsilon$ whose boundary is obtained from a  connected sum joining neighborhoods of $p_1, p_2$.

For each sufficiently small $\epsilon>0$ there exist a metric $g_\epsilon$ and a tensor field $\Sigma_\epsilon$ that give rise to a CMCSF hyperboloidal data set of class $\mathscr C^{k,\alpha;2}$ on $M_\epsilon$.
As $\epsilon\to 0$, the tensor fields $(g_\epsilon, \Sigma_\epsilon)$ converge to $(g,\Sigma)$ in the following sense:

\begin{description}

\item[Convergence in the exterior region]
For each sufficiently small $c>0$ we define an open set $E_c\subset M$, whose closure in $\overline M$ is disjoint from $p_1,p_2$.
The sets $E_c$ exhaust $M$ in the sense that
$
\bigcup_{c>0}  E_c = M.
$

For each $\epsilon\ll c$ there exists an embedding $\iota_\epsilon\colon  E_c\to M_\epsilon$.
Our convergence result in the exterior region is that for fixed $c$ we have
\begin{equation}
\label{ext-convergence}
(\rho^2\iota_\epsilon^* g_\epsilon, \rho \iota_\epsilon^*\Sigma_\epsilon) \to (\rho^2 g,\rho\Sigma )
\end{equation}
in the $\mathscr C^{k,\alpha;2}\times \mathscr C^{k-1,\alpha;1}$ topology on $E_c$.

\item[Convergence in the neck]
For each sufficiently small $c>0$ we define a subset $A_c$ of hyperbolic space $\mathbb H$ that in the half-space model corresponds to a semi-annular region; see \eqref{sets-in-hyperbolic-space}.
The sets $A_c$ exhaust hyperbolic space in the sense that
$
\bigcup_{c>0}  A_c = \mathbb H.
$

For each $\epsilon\ll c$ there exists an embedding $\Psi_\epsilon\colon A_c \to M_\epsilon$ such that 
$\Psi_\epsilon(A_c)\cap \iota_\epsilon(E_c) = \emptyset$.
Our convergence result in the neck is that the data converges to the unit hyperboloid of Minkowski space, in the sense that  for fixed $c$ we have
\begin{equation}
\label{neck-convergence}
(\breve\rho^2\Psi_\epsilon^* g_\epsilon, \breve\rho \Psi_\epsilon^*\Sigma_\epsilon)
 \to 
 (\breve\rho^2 \breve g, 0)
\end{equation}
in the $\mathscr C^{k,\alpha;2}\times \mathscr C^{k-1,\alpha;1}$ topology on $ A_c$.
Here $\breve\rho$ is a fixed defining function for hyperbolic space; see \eqref{breve-rho}.

\end{description}
\end{thm}

We emphasize that the above topology is the ``right space'' for convergence with regard to the shear free condition in view of the results of \cite{AHEM} that show that the shear-free condition is continuous in this topology, and the results of \cite{AllenStavrov-Density}, which shows that shear-free data sets are dense with respect to the weaker $C^{k,\alpha}$ topology.

We note that the class of initial data sets considered here includes those with polyhomogeneous regularity along the conformal boundary; see \cite{AnderssonChrusciel-Dissertationes},\cite{AHEM}.
The observant reader will note that each step in our construction preserves polyhomogeneity, and thus the application of Theorem \ref{main} to initial data that is both polyhomogeneous and shear-free yields polyhomogeneous data on $M_\epsilon$.
We refer the reader to \cite{AHEM}, and the references therein, for additional details concerning polyhomogeneous data.

\subsection{Overview of the construction}
\label{subsec:overview}
We begin our construction in the same manner as in \cite{ILS-Gluing}.
First, given $(g,\Sigma)$ on $M$, and given the two gluing points $p_1,p_2\in \partial M$, we use inversion with respect to half-spheres to construct a manifold $M_\epsilon$, along with a defining function $\rho_\epsilon$. We then use cutoff functions to construct a spliced metric $\lambda_\epsilon$ and spliced tensor field $\mu_\epsilon$ on $M_\epsilon$.
Second, we apply the conformal method of \cite{AHEM} to $(\lambda_\epsilon, \mu_\epsilon)$ in order to obtain $(g_\epsilon, \Sigma_\epsilon)$ satisfying the constraint equations \eqref{CMC-constraints}.

The spliced metric $\lambda_\epsilon$ is obtained from $g$ using a cutoff function.
To construct $\mu_\epsilon$ we follow the approach of \cite{AHEM}, and express the shear-free condition \eqref{define-SF} using a tensor $\mathcal H_{\overline g}(\rho)$ that, for the metrics appearing here, agrees with the traceless Hessian of $\rho$ along $\partial M$.
The definition and properties of $\mathcal H_{\overline g}(\rho)$ are detailed in \S\ref{sec:H-tensor}.
In order to splice the second fundamental forms, we write $\Sigma = \rho^{-1}\mathcal H_{\overline g}(\rho) +  \nu$ and then use a cutoff function to construct a tensor $\nu_\epsilon^\text{ext}$ that agrees with $\nu$ in the exterior region and vanishes in the neck.

We require that the metric $\lambda_\epsilon$, together with the tensor $\mu_\epsilon$, form a good approximate solution to the constraint equations.
In the middle of the neck we expect the solution to be very close to data corresponding to a hyperboloid in Minkowski space; for such data, $g = \breve g$ and $\Sigma =0$.
However, while $\nu_\epsilon^\text{ext}=0$ in the neck, the tensor $\rho_\epsilon^{-1}\mathcal H_{\overline\lambda_\epsilon}(\rho_\epsilon)$ is not small there.
Thus we must correct our approximate data by constructing a tensor $\nu_\epsilon^\text{neck}$ that counteracts the large terms in $\rho_\epsilon^{-1} \mathcal H_{\overline\lambda_\epsilon}(\rho_\epsilon)$.
The result is a family of  {\it spliced data sets}, each consisting of the metric $\lambda_\epsilon$ together with the tensor 
$$\mu_\epsilon = \rho_\epsilon^{-1}\mathcal H_{\overline\lambda_\epsilon}(\rho_\epsilon) 
+ \nu_\epsilon^\text{neck}+ \nu_\epsilon^\text{ext},$$
which approximately solve the constraint equations \eqref{CMC-constraints}.

In order to obtain an exact solution to the constraint equations from each  spliced data set $(\lambda_\epsilon, \mu_\epsilon)$, we make use of the conformal method; see \cite{AHEM} for a detailed description of the conformal method in this setting.
The first step of this method is to prove the existence of a vector field $W_\epsilon$ such that 
\begin{equation}
\label{vl-eqn}
L_{\lambda_\epsilon} W_\epsilon 
= (\Div_{\lambda_\epsilon}\mu_\epsilon)^\sharp.
\end{equation}
Here $L_{\lambda_\epsilon} = \mathcal D_{\lambda_\epsilon}^* \circ \mathcal D_{\lambda_\epsilon}$ is the \Defn{vector Laplace operator}, defined in terms of the \Defn{conformal Killing operator} $\mathcal D_{\lambda_\epsilon}$, which acts on vector fields by
\begin{equation}
\label{define-conformal-killing-operator}
\mathcal D_{\lambda_\epsilon}\colon W_\epsilon\mapsto \frac12\mathscr L_{W_\epsilon}{\lambda_\epsilon} - \frac13(\Div_{\lambda_\epsilon} W_\epsilon){\lambda_\epsilon}.
\end{equation}
Note that the adjoint $\mathcal D_{\lambda_\epsilon}^*$ acts on symmetric traceless covariant 2-tensors by 
\begin{equation*}
\mathcal D_{\lambda_\epsilon}^*\colon \mu_\epsilon \mapsto -(\Div_{\lambda_\epsilon} \mu_\epsilon)^\sharp
\end{equation*}
and thus $L_{\lambda_\epsilon} = -\Div_{\lambda_\epsilon}\circ \mathcal D_{\lambda_\epsilon}(\cdot) ^\sharp$.
That $W_\epsilon$ satisfies \eqref{vl-eqn} ensures that the tensor
\begin{equation}
\label{sigma-intro}
\begin{aligned}
\sigma_\epsilon 
&= \rho_\epsilon^{-1}\mathcal H_{\overline\lambda_\epsilon}(\rho_\epsilon) + \nu_\epsilon^\text{neck} + \nu_\epsilon^\text{ext} + \mathcal D_{\lambda_\epsilon} W_\epsilon
\\
&= \mu_\epsilon+ \mathcal D_{\lambda_\epsilon} W_\epsilon
\end{aligned}
\end{equation}
is divergence-free with respect to $\lambda_\epsilon$.
Furthermore, we solve for $W_\epsilon$ in a weighted function space that implies that the tensor $\overline\sigma_\epsilon = \rho_\epsilon\sigma_\epsilon$ satisfies 
\begin{equation*}
\overline\sigma_\epsilon\big|_{\rho_\epsilon=0} 
= \mathcal H_{\overline\lambda_\epsilon}(\rho_\epsilon)\big|_{\rho_\epsilon=0}.
\end{equation*}
This ensures that the resulting data set satisfies the shear-free condition.

Subsequently, we show the existence of a positive function $\phi_\epsilon$ satisfying the Lichnerowicz equation
\begin{equation}
\label{lichnero}
\Delta_{\lambda_\epsilon}\phi_\epsilon -\frac18\R[\lambda_\epsilon]\phi_\epsilon + \frac18|\sigma_\epsilon|^2_{\lambda_\epsilon}\phi_\epsilon^{-7} - \frac34 \phi_\epsilon^5=0
\end{equation}
and such that $\phi_\epsilon \to 1$ as $\rho_\epsilon \to 0$.
Direct computation shows that the metric $g_\epsilon = \phi_\epsilon^4\lambda_\epsilon$ and tensor $\Sigma_\epsilon = \phi_\epsilon^{-2}\sigma_\epsilon$ satisfy the constraint equations \eqref{CMC-constraints}, while a more delicate argument shows that in fact $g_\epsilon$ and $\Sigma_\epsilon$ have the necessary regularity to give rise to a CMCSF hyperboloidal data set as defined above.

In order to control the properties of $g_\epsilon$ and $\Sigma_\epsilon$, and thus establish the main theorem, the above process must be carried out in such a way that we obtain uniform control (in $\epsilon$) for each step of the process.
Quantifying this uniform control is a somewhat delicate matter, and we make use of specially weighted function spaces in order to accomplish the task.
Among other things, this requires uniform estimates on the mapping properties of various elliptic operators arising from $\lambda_\epsilon$.

Our work is organized as follows.
In \S\ref{sec:function-spaces} we define the regularity classes used, and recall from \cite{WAH,AHEM,Lee-FredholmOperators} their basic properties.
Subsequently in \S\ref{sec:H-tensor} we recall from \cite{WAH} the tensor $\mathcal H$, which is used in \cite{AHEM} to characterize the shear-free condition in a manner compatible with the conformal method.
The proof of Theorem \ref{main} begins in \S\ref{sec:splicing} with the construction of  the spliced manifolds $M_\epsilon$.
The spliced metrics $\lambda_\epsilon$ are defined in \S\ref{metric:details}, where their properties are established.
In \S\ref{sec:2ff} we construct the spliced tensors $\mu_\epsilon$ that give rise to the tensors $\sigma_\epsilon$, for which we obtain a number of crucial estimates.
We analyze the Lichnerowicz equation in \S\ref{sec:lich} before assembling the final bits of the proof in \S\ref{sec:proof}.
The uniform estimates for mapping properties of various elliptic operators arising from $\lambda_\epsilon$ involve a framework more general than our construction requires, and are placed in the appendix.

\section{Function spaces}
\label{sec:function-spaces}

Since the gluing construction uses the fact that the asymptotic geometry of $(M,g)$ is locally close to that of hyperbolic space, we first fix some notation involving hyperbolic space.
Using this, we briefly recall from \cite{WAH} the construction of various function spaces on $M$.

\subsection{Hyperbolic space}
\label{sec:hyperbolic-space}
Let $(\mathbb H, \breve g)$ denote the upper half space model of 3-dimensional hyperbolic space; in coordinates $X = (x,y)\in \mathbb R^2\times (0,\infty)$ we have 
\begin{equation*}
\breve g = \frac{(dx^1)^2 +  (dx^{2})^2 + dy^2}{y^2}.
\end{equation*}

For $r>0$ we define the following subsets of $\mathbb H$:
\begin{equation}
\label{sets-in-hyperbolic-space}
\begin{gathered}
\breve B_r = \{(x,y) : d_{\breve g}((x,y), (0,1))<r\},
\\
Y_r = \{(x,y) : |x|<r, y<r\},
\\
A_r = \{(x,y) : r^2 < |x|^2 + y^2 < 1/r^2\};
\end{gathered}
\end{equation}
here $|x|^2 = (x^1)^2 +  (x^{2})^2$.
We note for later use that, since $e^2<8$ and since $\breve B_r$ is determined by the hyperbolic metric $\breve g$, we have
\begin{equation}
\label{factor-of-8}
\breve B_2 \subset A_{1/8}.
\end{equation}

We make use of the fact that the inversion map
\begin{equation}
\label{define-inversion}
\mathcal I\colon (x,y) \mapsto \left( \frac{x}{|x|^2 + y^2}, \frac{y}{|x|^2 + y^2}\right)
\end{equation}
and the scaling maps
\begin{equation}
\label{define-scaling}
\mathcal S_\epsilon \colon (x,y)\mapsto (\epsilon x, \epsilon y),\qquad \epsilon >0,
\end{equation}
are isometries of $\mathbb H$. Note that $\mathcal I$ restricts to a map $A_r \to A_r$.

We identify $\overline{\mathbb H}$ with the half space $\{y\geq 0\}\subset \mathbb R^3$ and denote
\begin{equation}
\label{define-B-bar}
\begin{gathered}
 B_r = \{(x,y): |x|^2 + y^2 < r^2\}\subset \overline{\mathbb H},
\\
\overline B_r = \{(x,y): |x|^2 + y^2 \leq r^2\}\subset \overline{\mathbb H}.
\end{gathered}
\end{equation}

We make use of the defining function
\begin{equation}
\label{breve-rho}
\breve\rho = \frac{2y}{|x|^2 + (1+y)^2},
\end{equation}
which is the pullback to the half-space model of the standard defining function $\frac12(1-|u|^2)$ for the ball model.
On any fixed $\overline B_{r}$ we have
\begin{equation*}
\frac{1}{C_r}y \leq \breve\rho \leq C_r y,
\end{equation*}
for some constant $C_r$ depending only on $r$.

It is convenient to construct an inversion-invariant defining function on the annular region $A_c$.
To accomplish this, we first recall the following.
\begin{lemma}[Lemma 5 in \cite{ILS-Gluing}]
\label{cutoff1}
There exists a nonnegative and nondecreasing smooth cutoff
function $\chi\colon \mathbb{R}\to\mathbb{R}$ that is identically $1$ on $[2,\infty)$,
is supported in $(\frac{1}{2},\infty)$, and
satisfies the condition
\begin{equation}\label{phi-eqn}
\chi(r)+\chi \left(1/r\right)\equiv 1.
\end{equation}
\end{lemma}
We now define the function $F\colon(0,\infty)\to (0,\infty)$ by
$$
F(r)=\chi(r)+\frac{1}{r^2}\chi(1/r).
$$
The following is immediate from this definition.
\begin{prop}\label{lemma:F}
The function $F\colon(0,\infty)\to (0,\infty)$ satisfies
$$
\begin{aligned}
&F\left(1/r\right)=r^2F(r), && r\in \mathbb R,\\
&F(r)=1, &&r\ge 2,\\
&F(r)=\frac{1}{r^2}, &&r\le 1/2.
\end{aligned}
$$
\end{prop}

The functions $\chi$ and $F$ give rise to functions on $\mathbb H$, which we denote by the same symbols, by taking $r$ to be given by $r^2 = |x|^2 + y^2$.
Using this, we see that $yF$ is inversion-invariant; i.e.,~
$
\mathcal I^*(yF) = yF;
$
and that on each $A_c$ we have 
\begin{equation}
\label{annular-defining-functions}
\frac{1}{C} y \leq \breve \rho \leq {C} y 
\quad\text{ and }\quad
\frac{1}{C^\prime} yF\leq \breve \rho \leq {C^\prime} yF,
\end{equation}
for constants $C$ and $C^\prime$ depending only on $c$.

\subsection{Background coordinates and M\"obius parametrizations}
\label{subsec:background-coord}
The construction of function spaces on $M$ given in \cite{WAH} (see also \cite{Lee-FredholmOperators}) 
relies on identifying coordinate neighborhoods of $M$ with neighborhoods in hyperbolic space.
Here we slightly modify that construction to be compatible with our gluing construction.

By choosing a  collar neighborhood of $\partial M \subset \overline M$ and rescaling $\rho$ by a constant if necessary, we hereafter identify a neighborhood of $\partial M$ in $\overline M$ with $\partial M\times [0,1)$, and identify $\rho$ with the coordinate on $[0,1)$.

For each $\hat p\in \partial M$ we choose smooth coordinates $\theta$ taking a neighborhood $U(\hat p)\subset \partial M$ to a ball of radius $1$ in $\mathbb R^2$.
We extend these coordinates to smooth coordinates 
\begin{equation}
\label{background-coord}
\Theta = (\theta,\rho)\colon Z(\hat p) \to  Y_1,
\end{equation}
where $Z(\hat p) = U(\hat p) \times[0,1)$ and $Y_1$ is given by \eqref{sets-in-hyperbolic-space}.
We may assume that $\Theta$ extends smoothly to $\overline{Z(\hat p)}$.
We fix a finite collection of points $\hat p$ such that the corresponding sets $\Theta^{-1}(Y_{1/8})$ cover $\partial M\times[0,1/8)$.
The finite collection of coordinates $\Theta$ we refer to as \Defn{background coordinates}.

We may assume that the finite collection of points $\hat p$ contains the points $p_i$, $i=1,2$, of the main theorem.
About these two points we may further assume that we have \Defn{preferred background coordinates} 
$\Theta_i = (\theta_i,\rho)$ centered at $p_i$ satisfying the following conditions:
\begin{itemize}
\item $\Theta_i(p_i) = (0,0)$,
\item the coordinates $\Theta_i$ are defined on the set $Z(p_i)$ with
\begin{equation*}
\Theta_i(Z(p_i) )= \{(\theta,\rho): (\theta_i^1)^2 + (\theta_i^2)^2\leq 1\text{ and }  \rho^2 \leq 1\},
\end{equation*}
 and
\item in coordinates $\Theta_i$ the metric $\overline g = \rho^2g$ takes the form $\overline g\big|_{p_i} = \delta_{ab}d\Theta_i^a d\Theta_i^b$.
\end{itemize}
Note that we can arrange the preferred background coordinates so that $Z(p_1)\cap Z(p_2) = \emptyset$.
We also define, for use below, the coordinate half balls
\begin{equation}
\label{preferred-half-ball}
\begin{gathered}
 U_{i,r} = \{q\in Z(p_i) \colon  (\theta_i^1(q))^2 + (\theta_i^2(q))^2 + \rho(q)^2 < r^2 \},
\\
\overline U_{i,r} = \{q\in Z(p_i) \colon  (\theta_i^1(q))^2 + (\theta_i^2(q))^2 + \rho(q)^2 \leq r^2 \}.
\end{gathered}
\end{equation}

We fix a smooth \Defn{preferred background metric} $\overline h$ on $\overline M$ that satisfies $\overline h = \delta_{ab}d\Theta_i^ad\Theta_i^b$ in each of the two preferred background coordinate charts.
Let $\overline\nabla$ denote the Levi-Civita connection of $\overline h$ on $\overline M$.

For each $p\in M$ with $\rho(p) <1/8$ we define a \Defn{M\"obius parametrization} $\Phi_p\colon \breve B_2 \to M$ as follows.
Let $\Theta$ be a background coordinate chart with $p\in \Theta^{-1}(Y_{1/8})$; denote $\Theta(p)$ by $(\theta_p, \rho_p)$. Define $\Phi_p$ by $(\Theta\circ \Phi_p)(x,y) = (\theta_p + \rho_px, \rho_py)$.
The inclusion \eqref{factor-of-8} ensures that $\Phi_p$ is well-defined.
To this collection we append an additional finite number of smooth parametrizations $\Phi\colon \breve B_2\to M\setminus \{\rho<1/16\}$ such that the sets $\Phi(\breve B_1)$ cover $M\setminus \{\rho<1/8\}$ and such that $\Phi$ extends smoothly to the closure of $\breve B_2$.
We denote this extended collection by $\{\Phi\}$ and refer to them as M\"obius parametrizations.

\subsection{H\"older spaces on $M$}
H\"older spaces of tensor fields on $M$ are defined using the norms
\begin{equation*}
\|u\|_{C^{k,\alpha}(M)} = \sup_{\Phi} \| \Phi^*u\|_{C^{k,\alpha}(\breve B_2)},
\end{equation*}
where the $C^{k,\alpha}(\breve B_2)$ norm is computed using the Euclidean metric.
Uniformly equivalent norms are produced by replacing $\breve B_2$ by $\breve B_r$ for any $1\leq r\leq 2$.
If $U\subset M$ is open, the $C^{k,\alpha}(U)$ norms are defined by appropriately restricting the domains of the M\"obius parametrizations.

In \cite{WAH}, weighted function spaces are defined using $\rho$ as a weight function.
Here we generalize that construction; see \cite{ChruscielDelay-mapping-properties}.
We say that a smooth function $w\colon M\to (0,\infty)$ satisfies the \Defn{scaling hypotheses} if there exist constants $c_0$ and $c_k$ such that for every M\"obius parametrization $\Phi_p$ we have
\begin{equation}
\label{scaling-hypotheses}
\begin{gathered}
c_0^{-1} w(p) \leq \Phi_p^*w \leq c_0 w(p) \quad\text{ and }
\\
\|\Phi_p^*w\|_{C^k(\breve B_2)} \leq c_k w(p),\quad k\geq 1.
\end{gathered}
\end{equation}
It is straightforward to see that if functions $w_1$ and $w_2$ each satisfy the scaling hypotheses, then so do the functions $w_1w_2$ and $w_1/w_2$.

Let $U\subset M$ be open.
Two functions $w_1$ and $w_2$ satisfying the scaling hypotheses are said to be \Defn{equivalent weight functions on $U$} if there exists a constant $C_U$ such that
\begin{equation}
\label{weight-function-equivalence}
\frac{1}{C_U} w_1(p)\leq w_2(p) \leq C_U w_1(p),
\qquad p\in U.
\end{equation}

For any function $w$ satisfying the scaling hypotheses and for any $\delta\in \mathbb R$, we endow the weighted H\"older spaces $C^{k,\alpha}_\delta (M;w) = w^\delta C^{k,\alpha}(M)$ with the norm
\begin{equation*}
\|u\|_{C^{k,\alpha}_\delta(M;w)} = \|w^{-\delta}u\|_{C^{k,\alpha}_\delta(M)}.
\end{equation*}

\begin{lemma}
\label{lemma:basic-weight-function-equivalence}
Let $U\subset M$ be open.
\begin{enumerate}
\item
\label{MobiusNormEquivalence}
 Suppose $w$ satisfies the scaling hypotheses.
 For each M\"obius parametrization $\Phi$ set $w_\Phi = w(\Phi(0,1))$.
For any $1\leq r\leq 2$ we have the norm equivalence
\begin{equation*}
\frac{1}{C} \|u\|_{C^{k,\alpha}_\delta(U;w)} 
\leq 
\sup_\Phi w_\Phi^{-\delta} \| \Phi^*u\|_{C^{k,\alpha}(\breve B_r)}
\leq
{C} \|u\|_{C^{k,\alpha}_\delta(U;w)} ,
\end{equation*}
where the constant $C$ depends only on $k$, $\delta$, and where the constants $c_0,\dots,c_{k+1}$  appearing in the scaling hypotheses for $w$.

\item If $w_1$ and $w_2$ are equivalent on $U$ then we have the norm equivalence
\begin{equation*}
\frac{1}{ C^\prime} \|u\|_{C^{k,\alpha}_\delta(U;w_1)} 
\leq
\|u\|_{C^{k,\alpha}_\delta(U;w_2)} 
\leq 
{ C^\prime} \|u\|_{C^{k,\alpha}_\delta(U;w_1)} ,
\end{equation*}
where the constant $ C^\prime$ depends only the constant $C$ from part \eqref{MobiusNormEquivalence} and the constant $C_U$ in \eqref{weight-function-equivalence}.
\end{enumerate}

\end{lemma}
\begin{proof}
The first claim is a straightforward generalization of \cite[Lemma 3.5]{Lee-FredholmOperators}, while the second follows immediately from the first.
\end{proof}

It follows from \cite[Lemma 3.3]{Lee-FredholmOperators} that the defining function $\rho$ satisfies the scaling hypotheses.
We suppress explicit reference to the weight function if $w=\rho$, so that $C^{k,\alpha}_\delta(M) = C^{k,\alpha}_\delta(M;\rho)$.

It is straightforward to verify that for any fixed $0<c<1$ the functions $\breve\rho$, $y$, and $yF$ satisfy the scaling hypotheses on $A_c\subset \mathbb H$.
Furthermore, from \eqref{annular-defining-functions} we see that these functions are equivalent weight functions on $A_c$.
Unless otherwise specified, we use the weight function $\breve\rho$.
Thus the norms
\begin{equation*}
\|u\|_{C^{k,\alpha}_\delta(A_c)},
\qquad
\|u\|_{C^{k,\alpha}_\delta(A_c;y)},
\qquad
\|u\|_{C^{k,\alpha}_\delta(A_c;yF)}
\end{equation*}
are all equivalent.

The following lemma relates H\"older spaces on $M$ and  $\overline M$.
We define the \Defn{weight} of a tensor field to be the covariant rank less the contravariant rank; thus a Riemannian metric has weight $2$.
\begin{lemma}
\label{lemma:barM-M}
\cite[Lemma 3.7]{Lee-FredholmOperators}
Suppose $u$ is a tensor field of weight $r$.
\begin{enumerate}
\item If $u\in C^{k,\alpha}(\overline M)$ and $|u|_{\overline h} = O(\rho^s)$ then $u\in C^{k,\alpha}_{r+s}(M)$ with 
$$\|u\|_{C^{k,\alpha}_{r+s}(M)} \leq C \|u\|_{C^{k,\alpha}(\overline M)}$$
for some constant $C$.
\item If $u\in C^{k,\alpha}_{k+\alpha + r}(M)$ then $u\in C^{k,\alpha}(\overline M)$ with 
$$\|u\|_{C^{k,\alpha}(\overline M)}\leq C \|u\|_{C^{k,\alpha}_{k+\alpha+r}(M)}$$
for some constant $C$.
\end{enumerate}
\end{lemma}

We now introduce the spaces $\mathscr C^{k,\alpha;m}(M)$, intermediate between $C^{k,\alpha}(M)$ and $C^{k,\alpha}(\overline M)$, first defined in \cite{WAH}.
For $0\leq m\leq k$ and $\alpha\in (0,1)$ we say that a tensor field $u$ having weight $r$ is in $\mathscr C^{k,\alpha;m}(M)$ if
\begin{equation*}
\mathscr L_{ X_1}\dots \mathscr L_{ X_j}u \in C^{k-j,\alpha}_r(M)
\end{equation*}
for all $0\leq j \leq m$ and for all smooth vector fields $X_1, \dots, X_j$ on $\overline M$.
By \cite[Lemma 2.2]{WAH} this is equivalent to requiring that the norm
\begin{equation}
\label{define-curly-norm}
\|u\|_{\mathscr C^{k,\alpha;m}(M)} = \sum_{j=0}^m \|
\overline\nabla{}^ju\|_{C^{k-j,\alpha}_{r+j}(M)}
\end{equation}
be finite;
recall that $\overline\nabla$ is the connection associated to the preferred background metric $\bar h$.
We also have occasion to use norms such as $\| \cdot \|_{\mathscr C^{k,\alpha;m}(M;w)}$, defined by replacing $C^{k-j,\alpha}_{r+j}(M)$ by $C^{k-j,\alpha}_{r+j}(M;w)$ in \eqref{define-curly-norm}.

We also define similar norms on $\mathbb H$. The 
H\"older norms are defined as above using the half-space model and the M\"obius parametrizations $\breve\Phi\colon \breve B_2\to \mathbb H$ of the form
\begin{equation}
\label{hyp-mobius}
\breve\Phi\colon (x,y) \mapsto (x_* + y_*x, y_*y).
\end{equation}
On $B_{1/c}$ for any $c<1/4$, we define the $\mathscr C^{k,\alpha;m}$ norms using $^\text{E}\nabla$, the connection associated to the Euclidean metric $g_\text{E}$, and the hyperbolic defining function $\breve\rho$.

The following proposition records several important properties of the function spaces described above.
\begin{prop}[{\cite[Lemma 2.3]{WAH}}]
\label{prop:basic-inclusions}
\hfill
\begin{enumerate}

\item The space of tensor fields on $M$ of a specific type and of class $\mathscr C^{k,\alpha;m}$ forms a Banach space with the norm \eqref{define-curly-norm}.
The space of all tensor fields of class $\mathscr C^{k,\alpha;m}$ forms a Banach algebra under the tensor product, and is invariant under contraction.

\item
\label{curly-to-weighted}
If $u\in C^{k,\alpha}_{r+m}(M)$ is a tensor field of weight $r$ then $u\in \mathscr C^{k,\alpha;m}(M)$ with
\begin{equation*}
\| u \|_{\mathscr C^{k,\alpha;m}(M)}
\leq C\| u\|_{ C^{k,\alpha}_{r+m}(M)}.
\end{equation*}

\item 
\label{IvasMagicLemma}
If $u\in \mathscr C^{k,\alpha;m}(M)$ is a tensor field of weight $r$ and
\begin{equation*}
|\overline\nabla{}^ju|_{\overline h} \to 0 \text{ and }\rho\to 0,\quad 0\leq j\leq m-1
\end{equation*}
then $u\in C^{k,\alpha}_{r+m}(M)$ with $\|u\|_{C^{k,\alpha}_{r+m}(M)} \leq C \| u \|_{\mathscr C^{k,\alpha;m}(M)}$.

\end{enumerate}
\end{prop}

We record a regularization result that follows from \cite[Theorem 2.6]{WAH}.
\begin{prop}
\label{prop:regularization}
Suppose $\tau\in \mathscr C^{k,\alpha;m}(M)$ is a tensor field of weight $r$. Then there exists a tensor field $\tilde\tau$ such that $\tilde\tau\in \mathscr C^{l,\beta;m}(M)$ for all $l$ and $\beta$, and such that $\tilde\tau - \tau \in C^{k,\alpha}_{r+m}(M)$.

Furthermore, for each $l$ and $\beta$ there exists a constant $C$ such that $\| \tilde \tau \|_{\mathscr C^{l,\beta;m}(M)} \leq C \| \tau \|_{\mathscr C^{k,\alpha;m}(M)}$.
 
Finally, in the case that $M = \mathbb H$ and $\tau$ is supported in $A_r$ then for any $0<\tilde r< r$ it can be arranged that $\tilde \tau$ is supported in $A_{\tilde r}$.
\end{prop}

We recall also the following version of Taylor's theorem.
\begin{prop}[Lemma 3.2 of \cite{WAH}]
\label{prop:bdy-taylor}
Suppose $g$ is weakly asymptotically hyperbolic of class $\mathscr C^{k,\alpha;2}$.
Then for any function $u\in \mathscr C^{k,\alpha;2}(M)\cap C^{k,\alpha}_1(M)$ we have $u - \rho \langle d\rho, du\rangle_{\overline g} \in C^{k-1,\alpha}_2(M)$ with
\begin{equation*}
\| u - \rho \langle d\rho, du\rangle_{\overline g}\|_{ C^{k-1,\alpha}_2(M)}
\leq C \| u \|_{\mathscr C^{k,\alpha;2}(M)},
\end{equation*}
where the constant $C$ depends only on $\| \overline g\|_{\mathscr C^{k,\alpha;2}(M)}$.
\end{prop}

We conclude this section by noting the effects of the scaling maps \eqref{define-scaling} on weighted H\"older norms in hyperbolic space.
Direct computation using the M\"obius parametrizations \eqref{hyp-mobius} shows that for a tensor field $u$ of weight $r$ we have
\begin{equation}
\label{scaling-of-norms}
\begin{gathered}
\|\mathcal S_\epsilon^*u \|_{C^{k,\alpha}_\delta(B_{1/c};y)}
=
\epsilon^\delta \|u\|_{C^{k,\alpha}_\delta(B_{\epsilon/c};y)}
\\
\| ({}^\text{E}\nabla)^j (\mathcal S_\epsilon^*u) \|_{C^{k-j,\alpha}_{r+j}(B_{1/c};y)}
=
\epsilon^{r+j} \|({}^\text{E}\nabla)^ju\|_{C^{k-j,\alpha}_{r+j}(B_{\epsilon/c};y)}.
\end{gathered}
\end{equation}

\section{The shear-free condition and the tensor $\mathcal H$}
\label{sec:H-tensor}

We recall here the tensor $\mathcal H_{\overline g}(\omega)$ introduced in \cite{WAH} and used in \cite{AHEM} to incorporate the shear-free condition into the conformal method.
For any metric $\overline g$ and for any function $\omega$ we define the tensor $\mathcal H_{\overline g}(\omega)$ by
\begin{equation}
\label{H-formula}
\mathcal H_{\overline g}(\omega)=
|d\omega|_{\overline  g}^6\,\mathcal D_{\overline  g}\left(|d\omega|^{-2}_{\overline g} \grad_{\overline  g}\omega\right)
+ A_{\overline  g}(\omega) \left( d\omega \otimes d\omega - \frac{1}{3}|d\omega|^2_{\overline  g} \overline  g \right),
\end{equation}
where
\begin{equation*}
A_{\overline  g}(\omega) 
= \frac{1}{2} |d\omega|_{\overline g} \Div_{\overline g}\left( |d\omega|_{\overline  g}\grad_{\overline  g}\omega\right)
\end{equation*}
and where $\mathcal D_{\overline g}$ is the conformal-Killing operator defined in \eqref{define-conformal-killing-operator}.
The following properties of this tensor are established in \S4 of \cite{WAH}.
\begin{prop}
\label{H-properties}
For any $C^1$ metric $\overline g$ and $C^2$ function $\omega$ we have the following.
\begin{enumerate}
\item\label{B-symmetric-tf}
 $\mathcal H_{\overline g}(\omega)$ is symmetric and trace-free.

\item\label{B-TransverseProperty} $\mathcal H_{\overline g}(\omega)(\grad_{\overline g}\omega, \cdot)=0$.

\item\label{B-Scaling} $\mathcal H_{\overline g}(c \omega)=c^5\mathcal H_{\overline g}(\omega)$ for all constants $c$.

\item\label{B-ConformalScaling} 
For any strictly positive function $\theta$, we have $\mathcal H_{\theta\overline g}(\omega)=\theta^{-2}\mathcal H_{\overline g}(\omega)$.

\end{enumerate}

If $\rho$ is a smooth defining function, we furthermore have the following.

\begin{enumerate}[resume]

\item \label{H-regularity}
If $\overline g\in \mathscr C^{k,\alpha;2}(M)$, then $\mathcal H_{\overline g}(\rho)\in \mathscr C^{k-1,\alpha;1}(M)$ and $\Div_{\overline g}\mathcal H_{\overline g}(\rho)\in C^{k-2,\alpha}_1(M)$.

\item 
\label{H-Hessian}
If $\overline g\in \mathscr C^{k,\alpha;2}(M)$ and $g = \rho^{-2}\overline g$ satisfies $\R[g]+6 = C^{k-2,\alpha}_2(M)$, then 
\begin{equation}
\label{BisHessian}
\mathcal H_{\overline g}(\rho) - \left(\Hess_{\overline g}\rho - \frac{1}{3}(\Delta_{\overline g}\rho)\overline g\right)\in C^{k-1,\alpha}_3(M).
\end{equation}
\end{enumerate}
\end{prop}

Due to the property \eqref{BisHessian}, the shear-free condition \eqref{define-SF} is equivalent to requiring
$\overline\Sigma = \mathcal H_{\overline g}(\rho)$ along $\partial M.$

\section{The spliced manifold $M_\epsilon$} 
\label{sec:splicing}
We now begin the proof of the main theorem. We consider a CMCSF hyperboloidal data set $(g, K)$ of class $\mathscr C^{k,\alpha;2}$ on $M$ for fixed $k\geq 3$ and $\alpha\in (0,1)$.
As outlined above, the first step of the proof is to construct the spliced manifold $M_\epsilon$ and the spliced defining function $\rho_\epsilon$, as well as various function spaces on $M_\epsilon$.

\subsection{The splicing construction}
Recall the definition of $\overline B_r$ in \eqref{define-B-bar} and let $\epsilon>0$ be a small parameter.
For each of the gluing points $p_i$, $i=1,2$, let the mapping $\alpha_{\epsilon,i}\colon \overline{B}_{1/\epsilon} \to \overline M$ be given in preferred background coordinates by 
$$\Theta_i = (\theta_i,\rho)=\alpha_{\epsilon, i}(x,y)=(\epsilon x, \epsilon y).$$
The mappings $\alpha_{\epsilon,i}$ give us scaled parametrizations of  neighborhoods of $p_i\in \overline M$. 
For future use we note that because $\alpha_{\epsilon,i}= \Theta_i^{-1}\circ \mathcal S_\epsilon$ and $\Theta_i$ is an isometry from $\overline h$ to $g_E$, 
it follows from \eqref{scaling-of-norms} that
for any tensor $u\in \mathscr C^{k,\alpha;m}(M)$ of weight $r$ 
and any $0\leq j\leq m$, we have
\begin{equation}
\label{norm-scaling}
\begin{aligned}
\|({}^E\nabla)^j \alpha^*_{\epsilon,i}u\|_{C^{k-j}_{r+j}(B_{1/c};y)}
&\leq C \epsilon^{r+j}\| \overline\nabla^j u\|_{C^{k-j,\alpha}_{r+j}(M)}
\\
&\leq C \epsilon^{r+j} \|u\|_{\mathscr C^{k,\alpha;m}(M)},
\end{aligned}
\end{equation}
where the constant $C$ depends on $c$, but is independent of $\epsilon$.

Recall the sets $\overline U_{i,r}$ defined in \eqref{preferred-half-ball}.
Consider the equivalence relation $\sim$ on 
$$\overline M\setminus \left(\alpha_{\epsilon, 1}(\overline B_\epsilon)\cup \alpha_{\epsilon, 2}(\overline B_\epsilon)\right)
=
\overline M \setminus \left(\overline U_{1,\epsilon^2} \cup \overline U_{2,\epsilon^2} \right)$$ generated by 
$$
\alpha_{\epsilon,1}(x,y)\sim \left(\alpha_{\epsilon, 2}\circ \mathcal{I}\right) (x,y),
$$ 
where $\mathcal{I}$ is the inversion map defined in \eqref{define-inversion}. 
Define the \Defn{spliced manifold} $\overline M_\epsilon$ as the quotient manifold whose points are the equivalence classes of $\sim$. 
It is clear from the construction that $\overline M_\epsilon$ is a family of smooth manifolds with boundary. 
In addition, define $M_\epsilon$ as the subset of $\overline M_\epsilon$ consisting of points whose representatives are elements of $M$; thus $M_\epsilon$ is the interior of $\overline M_\epsilon$. 
Denote the underlying quotient map by $\pi_\epsilon$. The map 
\begin{equation*}
\Psi_\epsilon= \pi_\epsilon \circ \alpha_{\epsilon,1} =  \pi_\epsilon \circ \alpha_{\epsilon,2}\circ \mathcal I
\colon A_\epsilon \to M_\epsilon
\end{equation*}
is used throughout this paper to parametrize a region of $M_\epsilon$ that we refer to as the \Defn{neck}.

Recall the definition of $\overline U_{i,r}$ in \eqref{preferred-half-ball}.
For each sufficiently small $c>0$ we define the \Defn{exterior region} $E_c\subset M$ by 
$$
E_c = M \setminus \left( \overline U_{1,c}\cup  \overline U_{2,c}\right).
$$
Note that $\overline U_{i,c} = \alpha_{\epsilon, i}(\overline B_{c/\epsilon})$ and thus
$$E_c = M\setminus \left(\alpha_{\epsilon, 1}(\overline B_{c/\epsilon})\cup \alpha_{\epsilon, 2}(\overline B_{c/\epsilon})\right).$$
Clearly $\bigcup_{c>0}E_c = M$.

For the rest of this paper we assume 
\begin{equation}
\label{epsilon-estimate}
0<\epsilon < {c^2} < \frac{1}{64}.
\end{equation}
In particular, this implies that the map $\iota_\epsilon = \pi_\epsilon\big|_{E_c}\colon E_c\to M_\epsilon$ is an embedding.
Note that $\Psi_\epsilon^{-1}(E_c) = A_\epsilon \setminus \overline A_{\epsilon/c}$.

In establishing the main theorem, it is important to obtain estimates that are uniform in $\epsilon$; thus we adopt the following notational convention.
For quantities $X$ and $Z$, both depending on $c$ and $\epsilon$, we write $X\lesssim Z$ to mean that $X\leq CZ$ for some constant $C$ that may depend on $c$, but is independent of $\epsilon$ satisfying \eqref{epsilon-estimate}.
We write $X\approx Z$ when both $X\lesssim Z$ and $Z\lesssim X$.

\subsection{The defining function $\rho_\epsilon$}

We now introduce a family of defining functions $\rho_\epsilon$ on $M_\epsilon$. These functions agree with the original defining function $\rho$ away from the neck $\Psi_\epsilon(A_\epsilon)$, while on $\Psi_\epsilon(A_\epsilon)$ they are determined by 
\begin{equation}
\label{defn-of-rho-epsilon}
\Psi_\epsilon^*\rho_\epsilon=\epsilon yF.
\end{equation} 
Thus
\begin{equation}
\label{rho-transition}
\Psi_\epsilon^*\rho_\epsilon
=\begin{cases}
\alpha_{\epsilon, 1}^*\rho & \text{ where }r>2,
\\
(\alpha_{\epsilon,2}\circ\mathcal I)^*\rho & \text{ where } r< 1/2.
\end{cases}
\end{equation}
Furthermore, since $yF$ satisfies the scaling hypotheses on $A_\epsilon$, the functions $\rho_\epsilon$ satisfy the scaling hypotheses on $M_\epsilon$.

\begin{lemma}
Provided \eqref{epsilon-estimate} holds we have the following.
\begin{enumerate}
\item
\label{claim:exterior-rho}
 On $E_c$ we have $\iota_\epsilon^*\rho_\epsilon = \rho$.
\item
\label{claim:interior-rho}
 On $A_c$ the weight functions 
\begin{equation*}
\Psi_\epsilon^*\rho_\epsilon=\epsilon yF,\quad \epsilon y,\quad \epsilon \breve\rho,
\end{equation*}
are all equivalent, uniformly in $\epsilon$.
\end{enumerate}
\end{lemma}
\begin{proof}
Part \eqref{claim:exterior-rho} is due to \eqref{rho-transition}, while
part \eqref{claim:interior-rho} is a consequence of \eqref{annular-defining-functions}.
\end{proof}

\subsection{Function spaces on $M_\epsilon$}
\label{subsec:functions-on-Me}
In order to define function spaces on $M_\epsilon$, we first construct a  collection of M\"obius parametrizations for $M_\epsilon$ of two types:
\begin{itemize}
\item parametrizations of the form $\pi_\epsilon \circ \Phi \colon \breve B_2 \to M_\epsilon$, where $\Phi$ is a M\"obius parametrization for $M$ such that 
$$\Phi(\breve B_2)\subset  M\setminus \left(\overline U_{1,\epsilon^2} \cup \overline U_{2, \epsilon^2}\right),$$
 and
\item parametrizations of the form $\Psi_\epsilon \circ \breve\Phi \colon \breve B_2\to M_\epsilon$, where $\Breve\Phi\colon \breve B_2 \to \mathbb H$ is a M\"obius parametrization of $\mathbb H$ such that $\breve\Phi(\breve B_2)\subset A_{\epsilon}$.
\end{itemize}
The second type of parametrizations allow us to compare the geometry of the neck with that of hyperbolic space.
It is easy to see that these parametrizations cover all of $M_\epsilon$, and that this remains true if restricted to $\breve B_r$ for any $1\leq r\leq 2$.

This collection of M\"obius parametrizations is used to define the intrinsic H\"older spaces $C^{k,\alpha}(M_\epsilon)$ with norms 
\begin{equation*}
\|u\|_{C^{k,\alpha}(M_\epsilon)} = \sup \|\Phi_\epsilon^*u\|_{C^{k,\alpha} (\breve B_2)};
\end{equation*}
as before, we obtain alternate norms, uniformly equivalent in $\epsilon$, by replacing $\breve B_2$ with $\breve B_r$  for any $1\leq r\leq 2$.

Suppose that $\Phi_p\colon \breve B_2 \to M$ is a M\"obius parametrization arising from one of the preferred background coordinate charts $\Theta_i$ such that $\Phi_p(\breve B_2)\cap B_{\epsilon^2}(p_i) = \emptyset$.
Then the corresponding M\"obius parametrization $\pi_\epsilon \circ \Phi_p\colon \breve B_2\to M_\epsilon$ coincides with the parametrization $\Psi_\epsilon \circ \breve\Phi \colon \breve B_2\to M_\epsilon$ centered at $X_*$ with $\alpha_{\epsilon,i}(X_*) =p$.
Such parametrizations, which can be viewed as arising either from hyperbolic space or from the manifold $M$, cover the the transitional region of $M_\epsilon$ between the exterior region $\iota_\epsilon(E_c)$ and the neck region $\Psi_\epsilon(A_c)$.

The following is immediate.
\begin{lemma}
\label{lemma:pullback-norms}
For any $l$ and $\beta$ we have
\begin{enumerate}
\item $\|u\|_{C^{l,\beta}(\iota_\epsilon(E_c))} 
= \|\iota_\epsilon^*u\|_{C^{l,\beta}(E_c)}$,
\item $\|u\|_{C^{l,\beta}(\Psi_\epsilon(A_c))} 
= \|\Psi_\epsilon^*u\|_{C^{l,\beta}(A_c)}$.
\end{enumerate}
\end{lemma}

Notice that $\Psi_\epsilon^*\rho_\epsilon \approx \epsilon y$ becomes degenerate as $\epsilon \to 0$.
In order to avoid difficulties associated with this degeneracy, we define weighted H\"older spaces and intermediate spaces on $M_\epsilon$ using an alternate defining function $\tilde\rho_\epsilon$ that scales better in the neck as $\epsilon\to 0$.
In order to accomplish this, let
$\psi\colon (0,\infty) \to (0,1]$ be a smooth, nondecreasing function such that
\begin{equation*}
\psi(x) = \begin{cases}
2x & \text{ if } 0< x \leq \frac14
\\
1 & \text{ if } x \geq \frac34.
\end{cases}
\end{equation*}
We subsequently define a smooth function $\omega_\epsilon$ on $\overline M_\epsilon$ by setting $\omega_\epsilon = 1$ outside $\Psi_\epsilon(A_\epsilon)$ and requiring that
\begin{equation}
\label{define-w-epsilon}
\Psi_\epsilon^*\omega_\epsilon = \psi(\epsilon r + \epsilon/r)
\end{equation}
on $A_\epsilon$.
Note that
\begin{equation}
\label{omega-epsilon-cases}
\Psi_\epsilon^*\omega_\epsilon
=\begin{cases}
2\left(\epsilon r + \frac{\epsilon}{r}\right) & \text{ on }A_{8\epsilon},
\\
1 & \text{ on } A_\epsilon \setminus A_{4\epsilon/3}.
\end{cases}
\end{equation}
Note also that in preferred background coordinates
$\Theta_i = (\theta, \rho)$ we have $\pi_\epsilon^*\omega_\epsilon = \psi(|(\theta, \rho)| + \epsilon^2 |(\theta,\rho)|^{-1})$, while outside the domain of those coordinates we have $\pi_\epsilon^*\omega_\epsilon =1$.
Thus on $E_c$ we have $\iota_\epsilon^*\omega_\epsilon = \widetilde \omega_\epsilon \approx 1$.

With the function $\omega_\epsilon$ in hand we define 
\begin{equation}
\label{define-tilde-rho}
\tilde\rho_\epsilon = \rho_\epsilon / \omega_\epsilon.
\end{equation}
Direct computation shows that both $\rho_\epsilon$ and $\omega_\epsilon$, and hence $\tilde\rho_\epsilon$, satisfy the scaling hypotheses \eqref{scaling-hypotheses}.
Furthermore, for each fixed $c$, on $A_c$ we have
\begin{equation}
\label{rho-tilde-equivalence}
\Psi_\epsilon^*\tilde\rho_\epsilon = \frac{yF}{2\left(r + \frac{1}{r}\right)} \approx  yF \approx  y \approx  \breve\rho,
\end{equation}
together with analogous uniform estimates for all derivatives of $\Psi_\epsilon^*\tilde\rho_\epsilon$, while on $E_c$ we have $\iota_\epsilon^*\tilde\rho_\epsilon \approx \rho$.
Combining this with Lemma \ref{lemma:pullback-norms} yields the following.
\begin{lemma}
\label{lemma:pullback-weighted-norms}
For any $l$, $\beta$, and $\delta$ we have
\begin{enumerate}
\item $\|u\|_{C^{l,\beta}_\delta (\iota_\epsilon(E_c);\tilde\rho_\epsilon)} 
\approx \|\iota_\epsilon^*u\|_{C^{l,\beta}_\delta(E_c)}$,

\item $\|u\|_{C^{l,\beta}_\delta(\Psi_\epsilon(A_c);\tilde\rho_\epsilon)} 
\approx \|\Psi_\epsilon^*u\|_{C^{l,\beta}_\delta(A_c)}$.
\end{enumerate}
\end{lemma}
For any region $U\subset M_\epsilon$ we note that $C^{l,\beta}_\delta(U;\rho_\epsilon)$ and $C^{l,\beta}_\delta(U;\tilde\rho_\epsilon)$ coincide as sets; thus we only indicate the weight function $\tilde\rho_\epsilon$ if we are referring to the corresponding norms.

In order to define the intermediate spaces $\mathscr C^{k,\alpha;m}$ on $M_\epsilon$, we construct a family of smooth background metrics $\overline h_\epsilon$.
Recall the preferred background metric $\overline h$ defined in \S\ref{subsec:background-coord} and note that $h = \rho^{-2}\overline h$ satisfies 
$$
\alpha_{\epsilon, 1}^*h = \breve g = (\alpha_{\epsilon, 2}\circ\mathcal I)^*h.
$$
Thus $h$ descends to a metric $h_\epsilon$ on $M_\epsilon$ under the quotient map $\pi_\epsilon$.
We set $\overline h_\epsilon = \tilde\rho_\epsilon^2h_\epsilon$.
Note that on $A_c$ we have
\begin{equation*}
\Psi_\epsilon^*\overline h_\epsilon  = F^2 g_\text{E} \approx g_\text{E}.
\end{equation*}

Using $\overline h_\epsilon$ and $\tilde\rho_\epsilon$, we define the intermediate norm $\mathscr C^{k,\alpha;m}$ of a tensor field $u$ having weight $r$ as follows:
\begin{equation*}
\|u\|_{\mathscr C^{k,\alpha;m}(M_\epsilon;\tilde\rho_\epsilon)} = \sum_{j=0}^m \|^\epsilon\overline\nabla{}^j u\|_{C^{k-j,\alpha}_{r+j}(M_\epsilon,\tilde\rho_\epsilon)},
\end{equation*}
where ${}^\epsilon\overline\nabla$ is the connection associated to $\overline h_\epsilon$.

The following lemma follows directly from the various definitions involved.
\begin{lemma}
\label{lemma:pullback-curly-norms}
For any $l$, $\beta$, and $m$ we have
\begin{enumerate}
\item $\|u\|_{\mathscr C^{l,\beta;m} (\iota_\epsilon(E_c);\tilde\rho_\epsilon)} 
\approx \|\iota_\epsilon^*u\|_{\mathscr C^{l,\beta;m}(E_c)}$,

\item $\|u\|_{\mathscr C^{l,\beta;m}(\Psi_\epsilon(A_c);\tilde\rho_\epsilon)} 
\approx \|\Psi_\epsilon^*u\|_{\mathscr C^{l,\beta;m}(A_c)}$.
\end{enumerate}
\end{lemma}

\section{The spliced metrics}
\label{metric:details}
For each $0<\epsilon< {1}/{64}$, we  define the \Defn{spliced metric} $\lambda_\epsilon$ to be the metric on $M_\epsilon$ that agrees with $(\pi_\epsilon)_*g$ away from the neck $\Psi_\epsilon(A_\epsilon)$, while on $\Psi_\epsilon(A_\epsilon)$ it satisfies
\begin{equation}\label{def-g}
\Psi_\epsilon^*\lambda_\epsilon^{-1} =\chi(r) [(\alpha_{\epsilon,1})^*g]^{-1} 
+
\chi\left(1/r\right) [(\alpha_{\epsilon,2}\circ \mathcal{I})^* g]^{-1};
\end{equation}
here $\chi$ is the cutoff in \eqref{phi-eqn}.
The reason for splicing cometrics rather than metrics is that it is easier to verify that the asymptotic hyperbolicity property holds if we work with cometrics; see Proposition \ref{prop:WAH} below.
We set $\overline\lambda_\epsilon = \rho_\epsilon^2 \lambda_\epsilon$.
Note that
\begin{equation}
\label{exterior-metric}
\Psi_\epsilon^*\lambda_\epsilon
=
\begin{cases} 
\alpha_{\epsilon,1}^*g & \text{ where }r>2,
\\
(\alpha_{\epsilon,2}\circ\mathcal I)^*g & \text{ where } r<1/2.
\end{cases}
\end{equation}

In order to establish estimates for the spliced metric $\lambda_\epsilon$, we first analyze the pullback metrics $\alpha_{\epsilon,i}^*g$.
Following \cite{ILS-Gluing}, we write
\begin{equation*}
\alpha_{\epsilon,i}^*g = y^{-2}(g_\text{E} + m_{\epsilon, i})
\end{equation*}
for tensors $m_{\epsilon, i}$; here $g_\text{E}$ represents the Euclidean metric on the half space.
We furthermore define the contravariant $2$-tensor fields $j_{\epsilon,i}$ by
\begin{equation*}
(g_\text{E} + m_{\epsilon,i})^{-1} = g_\text{E}^{-1} + j_{\epsilon,i}.
\end{equation*}

Recall that throughout we let $c$ be a fixed constant less than $1/8$ and assume that $\epsilon$ satisfies \eqref{epsilon-estimate}.
\begin{prop}
\label{prop:estimate-m}
The tensors $m_{\epsilon, i}$  and $j_{\epsilon, i}$ are in $\mathscr C^{k,\alpha;2}(A_c)$ and satisfy
\begin{enumerate}
\item $\|m_{\epsilon, i}\|_{\mathscr C^{k,\alpha;2}(A_c)} \lesssim \epsilon$, 

\item $\|j_{\epsilon, i}\|_{\mathscr C^{k,\alpha;2}(A_c)} \lesssim \epsilon$.
\end{enumerate}

\end{prop}
\begin{proof}
Note that $\overline g = \rho^2g$ satisfies
\begin{equation}
\label{pull-bar-g}
\alpha_{\epsilon,i}^*\overline g = \epsilon^2 (g_\text{E} + m_{\epsilon, i}).
\end{equation}
Since $\overline g\in \mathscr C^{k,\alpha;2}(M)$ and $g_\text{E} \in \mathscr C^{k,\alpha;2}(B_{1/c})$, we immediately have $m_{\epsilon, i}\in \mathscr C^{k,\alpha;2}(B_{1/c})$.
This inclusion, however, does not come with an estimate uniform in $\epsilon$.

The preferred background coordinates are constructed so that $m_{\epsilon, i} =0$ at $(x,y) = (0,0)$.
Thus the mean value theorem implies that
\begin{equation*}
\|m_{\epsilon, i}\|_{C^0_2(B_{1/c})}
\lesssim
\|{}^\text{E} \nabla m_{\epsilon, i}\|_{C^0_2(B_{1/c})}.
\end{equation*}
As a consequence
\begin{equation*}
\|m_{\epsilon, i}\|_{C^{k,\alpha}_2(B_{1/c})} 
\lesssim 
\|{}^\text{E} \nabla m_{\epsilon, i}\|_{C^{k-1,\alpha}_2(B_{1/c})}.
\end{equation*}
Using  this, \eqref{pull-bar-g}, and \eqref{norm-scaling} we find
\begin{align*}
\|m_{\epsilon, i}&\|_{\mathscr C^{k,\alpha;2}(A_c)}
\\
&\lesssim 
\|^\text{E}\nabla m_{\epsilon, i}\|_{C^{k-1,\alpha}_2(B_{1/c})}
\\ &\qquad
+ \|^\text{E}\nabla m_{\epsilon, i}\|_{C^{k-1,\alpha}_3(B_{1/c})}
+ \|({}^\text{E}\nabla)^2 m_{\epsilon, i}\|_{C^{k-2,\alpha}_4(B_{1/c})}
\\
&\lesssim
 \|^\text{E}\nabla m_{\epsilon, i}\|_{C^{k-1,\alpha}_3(B_{1/c};y)}
+ \|({}^\text{E}\nabla)^2 m_{\epsilon, i}\|_{C^{k-2,\alpha}_4(B_{1/c};y)}
\\
&=
 \epsilon^{-2} \|^\text{E}\nabla \alpha_{\epsilon, i}^*\overline g\|_{C^{k-1,\alpha}_3(B_{1/c};y)}
+ \epsilon^{-2}\|({}^\text{E}\nabla)^2 \alpha_{\epsilon, i}^*\overline g\|_{C^{k-2,\alpha}_4(B_{1/c};y)}
\\
&\lesssim \epsilon \|\overline g\|_{\mathscr C^{k,\alpha;2}(M)},
\end{align*}
which establishes the first claim.
The second claim follows from similar reasoning, together with the fact that $\overline g\in C^{1,1}(\overline M)$ (see \cite[Lemma 2.2]{WAH}), from which we easily obtain a uniform invertibility bound.
\end{proof}

\begin{prop}
\label{prop:vanishing-j}
Along $y=0$ we have 
$$j_{\epsilon, i}(dy,dy) =0
\quad\text{ and }\quad
(\mathcal I^*j_{\epsilon, i})(dy, dy)=0.$$
\end{prop}
\begin{proof}
In view of \eqref{pull-bar-g} we have
\begin{equation*}
\alpha_{\epsilon, i}^*(|d\rho|^2_{\bar g})
= |\epsilon dy|^2_{\epsilon^2(g_\text{E} + m_{\epsilon, i})}
= |dy|^2_{g_\text{E} + m_{\epsilon,i}}
= (g_\text{E}^{-1} + j_{\epsilon,i})(dy, dy).
\end{equation*}
Thus the assumption that $|d\rho|_{\overline g} = 1$ along $\partial M$ implies that  $ j_{\epsilon,i}(dy, dy) =0$ where $y=0$.
The second claim follows from the first and the fact that $\mathcal I_*dy = r^{-2}dy + O(y)$.
\end{proof}

We now obtain uniform bounds on the metric $\lambda_\epsilon$; these give rise to uniform bounds for geometric differential operators on $M_\epsilon$.
We subsequently obtain stronger estimates in the neck region $\Psi_\epsilon(A_c)$.

\begin{prop} 
\label{prop:lambda-estimates}
\hfill
\begin{enumerate}

\item  $\|\lambda_\epsilon^{-1}\|_{C^{k,\alpha}(M_\epsilon)}\lesssim 1$.

\item  $\|\lambda_\epsilon \|_{C^{k,\alpha}(M_\epsilon)}\lesssim 1$.

\end{enumerate}
\end{prop}

\begin{proof}
Due to \eqref{exterior-metric} it suffices to establish the estimates on $\Psi_\epsilon(A_{1/2})$.
Since $A_{1/2}\subset A_c$ by \eqref{epsilon-estimate}, we use Lemma \ref{lemma:pullback-norms} and \eqref{def-g} to estimate 
\begin{align*}
\|\lambda_\epsilon^{-1}\|_{C^{k,\alpha}(\Psi_\epsilon(A_{1/2}))}
&\leq 
\|\Psi_\epsilon^*\lambda_\epsilon^{-1}\|_{C^{k,\alpha}(A_{c})}
\\
&\lesssim 
\|\alpha_{1,\epsilon}^*g^{-1}\|_{C^{k,\alpha}(A_{c})}
\\
&\quad+
\|(\alpha_{2,\epsilon}\circ \mathcal I)^*g^{-1}\|_{C^{k,\alpha}(A_{c})}
\\
&\lesssim \| g^{-1}\|_{C^{k,\alpha}(M)},
\end{align*}
which is finite due to the fact that $\overline g = \rho^2 g\in \mathscr C^{k,\alpha;2}(M)$, and therefore $\overline g{}^{-1} \in \mathscr C^{k,\alpha;2}(M)$.

Proposition \ref{prop:estimate-m} implies that $\Psi_\epsilon^*\lambda_\epsilon$ is uniformly invertible on $A_c$, and thus the desired bound on $\lambda_\epsilon$ follows from the corresponding bound on $\lambda_\epsilon^{-1}$.
\end{proof}

The bounds in the previous proposition imply estimates for differential operators arising from $\lambda_\epsilon$.
We say that a differential operator $ P =  P[g]$ is a \Defn{geometric operator of order $l$} determined by the metric $g$ if in any coordinate frame the components of $ Pu$ are linear functions of $u$ and its partial derivatives, whose coefficients are universal polynomials in the components of $g$, their partial derivatives, and $(\det g_{ij})^{-1/2}$.
We furthermore require that the coefficients of the $j^\text{th}$ derivatives of $u$ involve no more than $l-j$ derivatives of the metric.
Examples of geometric operators include the scalar Laplacian $\Delta_{g}$, the divergence operator, and the conformal Killing operator $\mathcal D_g$.
The mapping properties of geometric operators arising from asymptotically hyperbolic metrics are studied systematically in \cite{Lee-FredholmOperators} (see also \cite{Andersson-EllipticSystems}); the extension of that work to the weakly asymptotically hyperbolic setting appears in \cite{AHEM}.
From these works we deduce the following.

\begin{prop}
\label{prop:generic-operators}
Let $ P$ be a geometric operator of order $l$ and suppose that $l\leq j\leq k$ and $\delta\in \mathbb R$.
\begin{enumerate}
\item If $u\in C^{j,\alpha}_\delta(M_\epsilon;\tilde\rho_\epsilon)$ then 
\begin{equation*}
\|  P[\lambda_\epsilon]u \|_{C^{j-l,\alpha}_\delta(M_\epsilon;\tilde\rho_\epsilon)}
\lesssim \|u\|_{C^{j,\alpha}_\delta(M_\epsilon;\tilde\rho_\epsilon)}.
\end{equation*}

\item Furthermore, if $ P$ is an elliptic operator and $u\in C^{0}_\delta(M_\epsilon;\tilde\rho_\epsilon)$ with $ P[\lambda_\epsilon]u\in C^{j-l,\alpha}_\delta(M_\epsilon;\tilde\rho_\epsilon)$ then $u\in C^{j,\alpha}_\delta(M_\epsilon;\tilde\rho_\epsilon)$ with
\begin{equation*}
\|u\|_{C^{j,\alpha}_\delta(M_\epsilon;\tilde\rho_\epsilon)}
\lesssim 
\|  P[\lambda_\epsilon]u \|_{C^{j-l,\alpha}_\delta(M_\epsilon;\tilde\rho_\epsilon)}
+
\|u\|_{C^{0}_\delta(M_\epsilon;\tilde\rho_\epsilon)}.
\end{equation*}
\end{enumerate}

\end{prop}

\begin{proof}
The statements follow directly from \cite[Lemmas 4.6 and 4.8]{Lee-FredholmOperators}, making use of the uniform bounds established in Proposition \ref{prop:lambda-estimates}.
\end{proof}

The previous proposition immediately implies the following.
\begin{prop}
\label{prop:operator-estimates}
\hfill
\begin{enumerate}
\item The conformal Killing operator $\mathcal D_{\lambda_\epsilon}$ defined in \eqref{define-conformal-killing-operator} satisfies
$$
\|\mathcal{D}_{\lambda_\epsilon} W\|_{C^{k-1,\alpha}_2(M_\epsilon,\tilde\rho_\epsilon)}\lesssim \|W\|_{C^{k,\alpha}_2(M_\epsilon,\tilde\rho_\epsilon)}
$$
for any vector field $W\in C^{k,\alpha}_2(M_\epsilon,\tilde\rho_\epsilon)$.

\item The divergence operator satisfies
$$
\|\Div_{\lambda_\epsilon} T\|_{C^{k-2,\alpha}_2(M_\epsilon,\tilde\rho_\epsilon)}
\lesssim \|T\|_{C^{k-1,\alpha}_2(M_\epsilon,\tilde\rho_\epsilon)}
$$
for any covariant $2$-tensor field $T\in C^{k-1,\alpha}_2(M_\epsilon,\tilde\rho_\epsilon)$.

\item 
Let $L_{\lambda_\epsilon}$ be the  vector Laplace operator defined in \eqref{vl-eqn}.
For any vector field $W\in C^{0}_2(M_\epsilon,\tilde\rho_\epsilon)$ 
with $L_{\lambda_\epsilon} W\in C^{k-2,\alpha}_2(M_\epsilon,\tilde\rho_\epsilon)$
we have $W\in C^{k,\alpha}_2(M_\epsilon,\tilde\rho_\epsilon)$ with
$$
\|W\|_{C^{k,\alpha}_2(M_\epsilon,\tilde\rho_\epsilon)}
\lesssim \|L_\epsilon W\|_{C^{k-2,\alpha}_2(M_\epsilon,\tilde\rho_\epsilon)}+\|W\|_{C^{0}_2(M_\epsilon,\tilde\rho_\epsilon)}.
$$

\item Suppose the functions $f_\epsilon\in C^{k-2,\alpha}(M_\epsilon)$ satisfy $\| f_\epsilon \|_{C^{k-2,\alpha}(M_\epsilon)}\leq K$.
Let $\mathcal{P}_\epsilon=\Delta_{\lambda_\epsilon}+ f_\epsilon$.
For any function $u\in C^{0}_2(M_\epsilon,\tilde\rho_\epsilon)$ with $\mathcal P_\epsilon u \in C^{k-2,\alpha}_2(M_\epsilon,\tilde\rho_\epsilon)$ we have
$$
\|u\|_{C^{k,\alpha}_2(M_\epsilon,\tilde\rho_\epsilon)}
\lesssim \|\mathcal{P}_\epsilon u\|_{C^{k-2,\alpha}_2(M_\epsilon,\tilde\rho_\epsilon)}
+ \|u\|_{C^{0}_2(M_\epsilon,\tilde\rho_\epsilon)},
$$
where the implicit constant depends on $K$.
\end{enumerate}
\end{prop}

We now use Proposition \ref{prop:estimate-m} to obtain additional estimates for $\lambda_\epsilon$ in the neck region.
To accomplish this, we write
\begin{equation}
\label{decompose-lambda}
\Psi_\epsilon^*\lambda_\epsilon = y^{-2}(g_\text{E} + m_\epsilon)
\quad\text{ and }\quad
(\Psi_\epsilon^*\lambda_\epsilon)^{-1} = y^2(g_\text{E}^{-1} + j_\epsilon).
\end{equation}
Since $\mathcal I^*y = y/r^2$, the tensor $j_\epsilon$ takes the form 
\begin{equation}
\label{split-j}
j_\epsilon = \chi(r) j_{\epsilon,1} +  \frac{\chi(1/r)}{r^4} \mathcal I^*j_{\epsilon,2}.
\end{equation}

The following proposition uses \eqref{def-g} and \eqref{decompose-lambda} to show that $\lambda_\epsilon$ satisfies the regularity and boundary conditions necessary to be part of a CMCSF hyperboloidal data set of class $\mathscr C^{k,\alpha;2}$ on $M_\epsilon$.
In particular, $\lambda_\epsilon$ satisfies the hypotheses of the conformal method in \cite{AHEM}.

\begin{prop}
\label{prop:WAH}
For each $\epsilon$ satisfying \eqref{epsilon-estimate} the metric $\lambda_\epsilon$ satisfies 
\begin{enumerate}
\item $\overline\lambda_\epsilon = \rho_\epsilon^2 \lambda_\epsilon \in \mathscr C^{k,\alpha;2}(M_\epsilon)$ and
\item $|d\rho_\epsilon|^2_{\overline\lambda_\epsilon} = 1$ along $\partial M_\epsilon$.
\end{enumerate}
Furthermore,
\begin{enumerate}[resume]

\item in the exterior region $E_c$ we have $\iota_\epsilon^*\lambda_\epsilon = g$, while

\item
\label{me-small}
 in the neck region we have
$$
 \| y^2\left(\Psi_\epsilon^*\lambda_\epsilon - \breve g \right)\|_{\mathscr C^{k,\alpha;2}(A_c)} = \| m_\epsilon \|_{\mathscr C^{k,\alpha;2}(A_c)} \lesssim \epsilon
$$
and thus
\begin{equation*}
\| y^2 \Psi_\epsilon^*\lambda_\epsilon \|_{\mathscr C^{k,\alpha;2}(A_c)} \lesssim 1
\quad\text{ and }\quad
\| y^{-2}( \Psi_\epsilon^*\lambda_\epsilon)^{-1} \|_{\mathscr C^{k,\alpha;2}(A_c)} \lesssim 1.
\end{equation*}

\end{enumerate}
\end{prop}

\begin{proof}
Note that the cutoff functions in \eqref{def-g} extend smoothly to $\overline M$.
Thus Proposition \ref{prop:estimate-m} implies that $\overline\lambda_\epsilon \in \mathscr C^{k,\alpha;2}(M_\epsilon)$.

Using Proposition \ref{prop:vanishing-j} and \eqref{split-j} we have that
\begin{equation}
\label{neck-ah-identity}
|dy|^2_{g_\text{E}+m_\epsilon}= 1
\quad \text{ where }\quad
y=0.
\end{equation}
This, together with the computation
\begin{equation*}
\begin{aligned}
\Psi_\epsilon^*(|d\rho_\epsilon|^2_{\overline\lambda_\epsilon})
&=|d(yF)|^2_{F^2(g_\text{E}+m_\epsilon)}
\\
&= |dy|^2_{g_\text{E}+m_\epsilon}+2\frac{F'}{F} y\langle dy, dr\rangle_{g_\text{E}+m_\epsilon}+y^2\left(\frac{F'}{F}\right)^2|dr|^2_{g_\text{E}+m_\epsilon},
\end{aligned}
\end{equation*}
shows that $|d\rho_\epsilon|^2_{\overline\lambda_\epsilon} = 1$ where $\rho_\epsilon =0$. 

The identity in the exterior region is immediate from the construction.
The estimate in the neck region follows from Proposition \ref{prop:estimate-m} and the fact that the coordinate expression for $\mathcal I^*j_{\epsilon,2}$ is
\begin{equation*}
(\mathcal I^*j_{\epsilon,2})^{ab} = Q^{ab}_{cd}(j_{\epsilon,2})^{cd}
\end{equation*}
for some rational functions $Q^{ab}_{cd}$ that are uniformly bounded on $A_c$; see equations (36)--(37) of \cite{ILS-Gluing}.
\end{proof}

Finally, we obtain the following estimates in the neck region.
\begin{prop}
\label{prop:g-neck}
\hfill
\begin{enumerate}
\item $\| 1-|dy|^2_{g_\text{E} + m_\epsilon}\|_{\mathscr C^{k,\alpha;2}(A_c)}\lesssim \epsilon$,

\item $\| y\Delta_{g_\text{E} + m_\epsilon} y\|_{\mathscr C^{k-1,\alpha;2}(A_c)}\lesssim \epsilon$,

\item
\label{g-neck-scalar} $\|\R[g_\text{E} + m_\epsilon] \|_{C^{k-2,\alpha}(A_c)}\lesssim \epsilon$,

\item $\| 1-|dy|^2_{g_\text{E} + m_\epsilon}\|_{ C^{k,\alpha}_1(A_c)}\lesssim \epsilon$,

\item $\| y\Delta_{g_\text{E} + m_\epsilon} y\|_{C^{k-1,\alpha}_1(A_c)}\lesssim \epsilon$.
\end{enumerate}
\end{prop}
\begin{proof}
The first three claims follow directly from Proposition \ref{prop:WAH}\eqref{me-small}, while the latter two also make use of part \eqref{IvasMagicLemma} of Proposition \ref{prop:basic-inclusions}.
\end{proof}

\section{The second fundamental form}
\label{sec:2ff}
In this section we obtain, for each $\epsilon>0$ satisfying \eqref{epsilon-estimate}, a symmetric covariant $2$-tensor $\sigma_\epsilon$ such that the pair $(\lambda_\epsilon, \sigma_\epsilon)$ satisfies the shear-free condition \eqref{define-SF}, approximately solves the constraint equations \eqref{CMC-constraints}, and satisfies estimates compatible with the convergence statements in the main theorem.
Our construction differs significantly from that used in \cite{ILS-Gluing}, as the procedure there does not account for the shear-free condition.

\subsection{The spliced second fundamental form}
Recall that we express the second fundamental form of the original initial data as $K = \Sigma - g$.
We decompose the traceless part $\Sigma$ as
$$\Sigma=\rho^{-1}\mathcal H_{\overline g}(\rho)+\nu.$$ 
By hypothesis, $\overline\Sigma = \rho\Sigma \in \mathscr C^{k-1,\alpha;1}(M)$, and by Proposition \ref{H-properties} $\mathcal H_{\overline g}(\rho)\in \mathscr C^{k-1,\alpha;1}(M)$.
Thus $\rho\nu \in \mathscr C^{k-1,\alpha;1}(M)$.
The shear-free condition implies that $\rho\nu$ vanishes on $\partial M$; thus by part \eqref{IvasMagicLemma} of Proposition \ref{prop:basic-inclusions}  we have $\nu\in C^{k-1,\alpha}_2(M)$.

We splice the tensor $\nu$ using a cutoff function to construct a tensor $\nu_\epsilon^\text{ext}$ that agrees with $\nu$ in the exterior region and vanishes inside the neck so as to ensure that $\nu_\epsilon^\text{ext}$ is trace-free with respect to $\lambda_\epsilon$.
More precisely, we set $\nu_\epsilon^\text{ext} = (\pi_\epsilon)_*\nu$ outside of $\Psi_\epsilon(A_\epsilon)$, and on $\Psi_\epsilon(A_\epsilon)$  require
\begin{equation}\label{splicedeta}
\Psi_\epsilon^*\nu_\epsilon^\text{ext}=\chi\big(\tfrac{r^2}{8}\big)(\alpha_{\epsilon,1})^*\nu
+\chi\big(\tfrac{1}{8r^2}\big)(\alpha_{\epsilon,2}\circ \mathcal{I})^*\nu. 
\end{equation}
Note that
\begin{equation}
\label{exterior-nu-ext}
\Psi_\epsilon^*\nu_\epsilon^\text{ext}
=
\begin{cases}
(\alpha_{\epsilon,1})^*\nu &\text{ where } r>4,
\\
0 & \text{ where } \frac12 < r < 2,
\\
(\alpha_{\epsilon,2}\circ \mathcal{I})^*\nu & \text{ where } r<\frac{1}{4}.
\end{cases}
\end{equation}
Furthermore, we have the following.
\begin{lemma}
\label{lemma:nu-ext-reg}
The tensor $\nu_\epsilon^\text{ext}$ is trace-free with respect to $\lambda_\epsilon$ and is an element of $C^{k-1,\alpha}_2(M_\epsilon)$
with $\|\nu_\epsilon^\text{ext}\|_{C^{k-1,\alpha}_2(M_\epsilon)}\lesssim 1$.
\end{lemma}
\begin{proof}
We may view $\nu_\epsilon^\text{ext}$  as the pushforward under the projection $\pi_\epsilon$ of the tensor field $\chi_\epsilon\nu$ on $M,$ where the function $\chi_\epsilon$ is identically equal to $1$ except in the vicinity of the gluing points where, for $i=1,2$, we have
\begin{equation}\label{splice:eta}
\alpha_{\epsilon,i}^*\chi_\epsilon=
\chi({r^2}/{8}).
\end{equation}
In preferred background coordinates $\Theta_i$ on $M$ we have 
$$
\chi_\epsilon(\Theta_i) = \chi\left(\frac{|\Theta_i|^2}{8\epsilon^2}\right)
\quad\text{ where }\quad
\frac12 \leq \frac{|\Theta|^2}{8\epsilon^2} \leq 2;
$$ 
otherwise $\chi_\epsilon$ is constant.
From this it follows that
\begin{equation}
\label{chi-epsilon-est}
\|\chi_\epsilon\|_{C^k(M)}\lesssim 1.
\end{equation}
Thus $$
\| \chi_\epsilon\nu\|_{ C^{k-1,\alpha}_2(M)}
\leq 
\| \chi_\epsilon\|_{C^k(M)} \|\nu\|_{ C^{k-1,\alpha}_2(M)}
\lesssim 1.$$
As $\nu_\epsilon^\text{ext} = (\pi_\epsilon)_*(\chi_\epsilon\nu)$, this implies the desired estimate.

Note that the support of $\nu_\epsilon^\text{ext}$ is contained in the region where $\lambda_\epsilon = (\pi_\epsilon)_*g$.
Since $\nu$ is trace-free with respect to $g$, it follows that $\nu_\epsilon^\text{ext}$ is trace-free with respect to $\lambda_\epsilon$.
\end{proof}

Since $\nu_\epsilon^\text{ext}\in C^{k-1,\alpha}_2(M_\epsilon)$, one could for each sufficiently small $\epsilon>0$ apply the results of \cite{AHEM} to the pair $(\lambda_\epsilon, \nu_\epsilon^\text{ext})$ and obtain a shear-free initial data set via the conformal method.
However, the resulting solutions to the constraint equations might not satisfy the convergence statements of the main theorem.
This is because the term $\rho_\epsilon^{-1}\mathcal H_{\overline\lambda_\epsilon}(\rho_\epsilon)$ is generally of significant size in the neck, and thus corrections arising from the conformal method are generally not small.
For this reason we add a correction term, supported in $\Psi_\epsilon(A_\epsilon)$, to the tensor $\nu_\epsilon^\text{ext}$. 

It follows from \eqref{decompose-lambda} that $\Psi_\epsilon^*\overline\lambda_\epsilon = \epsilon^2 F^2(g_\text{E} + m_\epsilon)$.
Direct computation using parts \eqref{B-Scaling} and \eqref{B-ConformalScaling} of Proposition \ref{H-properties} shows 
\begin{equation}
\label{pullback-H-term}
\Psi_\epsilon^*(\rho_\epsilon^{-1}\mathcal H_{\overline\lambda_\epsilon}(\rho_\epsilon))
=
(yF)^{-1}\mathcal H_{F^2(g_\text{E}+m_\epsilon)}(yF).
\end{equation}
Thus our plan is to approximate $\Psi_\epsilon^*(\rho_\epsilon^{-1}\mathcal H_{\overline\lambda_\epsilon}(\rho_\epsilon))$ by
$
(yF)^{-1}\mathcal H_{F^2g_\text{E}}(yF).
$
We note the following properties of this approximating tensor.
\begin{lemma}
\label{lemma:Euclidean-H}
\hfill
\begin{enumerate}
\item $\mathcal H_{F^2g_\text{E}}(yF)$ is supported on $A_{1/2}$,
\item  $\mathcal H_{F^2g_\text{E}}(yF)$ vanishes at $y=0$, and
\item $(yF)^{-1}\mathcal H_{F^2 g_\text{E}}(yF)\in C^{k-1,\alpha}_2(A_{1/2})$.
\end{enumerate}
\end{lemma}
\begin{proof}
Note that $\mathcal H_{F^2g_\text{E}}(yF) 
= F^{-4} \mathcal H_{g_\text{E}}(yF)$.
Thus to establish the first property, it suffices to consider $\mathcal H_{g_\text{E}}(yF)$.
Where $r\geq 2$ we have $\mathcal H_{g_\text{E}}(yF) = \mathcal H_{g_\text{E}}(y) =0$, while for $r<1/2$ we have
\begin{equation*}
\mathcal H_{g_\text{E}}(yF)
=
r^{-8} \mathcal H_{r^{-4} g_\text{E}}({y}/{r^2})
=
r^{-8}\mathcal H_{\mathcal{I}^*g_\text{E}}(\mathcal{I}^*y)
=r^{-8}\mathcal{I}^*\mathcal H_{g_\text{E}}(y)
= 0.
\end{equation*}
This establishes the first claim.

For the second claim, we compute and estimate the various terms appearing in \eqref{H-formula}, restricting to the domain $A_{1/2}$, where $\mathcal H_{g_\text{E}}(yF)$ is supported.
First note that
\begin{equation*}
d(yF) = Fdy + ydF
\quad\text{ and }\quad
dF = \frac{F^\prime(r)}{r} \left(x^1 dx^2 + x^2 dx^2 + y dy\right).
\end{equation*}
Thus
\begin{equation*}
g_\text{E}(dy, dF) = O(y)
\quad\text{ and }\quad
|d(yF)|^2_{g_\text{E}} = F^2 + O(y^2).
\end{equation*}
Direct computation shows that
\begin{equation}
\label{first-H-yF}
{}^E\nabla\left(|d(yF)|_{g_\text{E}}^{-2} d(yF)\right) 
= F^{-2}\left( dy\otimes dF - dF\otimes dy \right) + O(y).
\end{equation}
Since $\mathcal D_{g_\text{E}}\left(|d(yF)|_{g_\text{E}}^{-2} \grad_{g_\text{E}}(yF)\right)$ is the symmetric and traceless part of   \eqref{first-H-yF}, we conclude that $\mathcal D_{g_\text{E}}\left(|d(yF)|_{g_\text{E}}^{-2} \grad_{g_\text{E}}(yF)\right) =O(y)$.
We furthermore compute
\begin{equation*}
|d(yF)|_{g_\text{E}} A_{g_\text{E}}(yF) = O(y).
\end{equation*}
Thus we establish the second claim.

Since $\mathcal H_{g_\text{E}}(yF)$ is a smooth covariant $2$ tensor on $\overline M_\epsilon$ and vanishes at the boundary we have  $\mathcal H_{g_\text{E}}(yF)\in C^{k-1,\alpha}_3(M_\epsilon)$ by 
Lemma \eqref{lemma:barM-M}.
This establishes the third claim.
\end{proof}

Lemma \ref{lemma:Euclidean-H} implies that we may obtain a well-defined $C^{k-1,\alpha}_2$ tensor field $\nu_\epsilon^\text{approx}$ on $M_\epsilon$ by requiring
$$
\Psi_\epsilon^*\nu_\epsilon^\text{approx}
= -(yF)^{-1} \mathcal H_{F^2g_\text{E}}(yF) ,
$$ 
and by setting $\nu_\epsilon^\text{approx} =0$ outside $\Psi_\epsilon(A_\epsilon)$.
However, the tensor $\nu_\epsilon^\text{approx}$ is not trace-free with respect to $\lambda_\epsilon$.
Thus we define the correction term $\nu_\epsilon^\text{neck}$ by
\begin{equation}
\label{define-nu-neck}
\nu_\epsilon^\text{neck} = \nu_\epsilon^\text{approx} - \frac13\left(\tr_{\lambda_\epsilon} \nu_\epsilon^\text{approx}\right) \lambda_\epsilon.
\end{equation}
We note the following basic properties of $\nu_\epsilon^\text{neck}$.
\begin{lemma}
\label{lemma:nu-neck-reg}
We have $\nu_\epsilon^\text{neck}\in C^{k-1,\alpha}_2(M_\epsilon)$
with $\|\nu_\epsilon^\text{neck}\|_{C^{k-1,\alpha}_2(M_\epsilon)} \lesssim 1$.
Furthermore, $\nu_\epsilon^\text{neck}$ is supported in $\Psi_\epsilon(A_{1/2})$, where we have
\begin{equation*}
\Psi_\epsilon^*\nu_\epsilon^\text{neck}
=
-(yF)^{-1} \left(\mathcal H_{F^2g_\text{E}}(yF) 
-
\frac13(j_\epsilon^{ab} \mathcal H_{F^2g_\text{E}}(yF)_{ab}) (g_\text{E} + m_\epsilon)\right).
\end{equation*}
\end{lemma}
\begin{proof}
Lemma \ref{lemma:Euclidean-H} yields an $\epsilon$-independent estimate for $\nu_\epsilon^\text{approx}$. 
Together with Proposition \ref{prop:estimate-m}, this implies the $C^{k-1,\alpha}_2$ estimate.
The remaining claims follow directly from the definition and from the fact that $\tr_{g_\text{E}}\mathcal H_{g_\text{E}}(yF) =0$.
\end{proof}

We now define, as discussed in the introduction, the spliced second fundamental form $\mu_\epsilon$ by
\begin{equation}
\label{defn:kappa}
\mu_\epsilon=\rho_\epsilon^{-1}\mathcal H_{\overline\lambda_\epsilon}(\rho_\epsilon) + \nu_\epsilon^\text{neck} + \nu_\epsilon^\text{ext}.
\end{equation}
It is important to note that $\mu_\epsilon = (\pi_\epsilon)_*\Sigma$ outside of  $\Psi_\epsilon(A_{1/4})$.

\subsection{Estimates for $\mu_\epsilon$}
Lemmas \ref{lemma:nu-ext-reg} and \ref{lemma:nu-neck-reg} indicate that  $\nu_\epsilon^\text{neck} + \nu_\epsilon^\text{ext}$ has the regularity required for using the results of \cite{AHEM} in order to obtain a shear-free initial data set from each pair $\lambda_\epsilon, \mu_\epsilon$ according to the procedure outlined in the introduction.

We now establish estimates on $\mu_\epsilon$ needed for the convergence results of the main theorem.
Recall that we are assuming \eqref{epsilon-estimate} throughout.
\begin{prop}
\label{prop:estimate-mu}
For all  $\epsilon$ satisfying \eqref{epsilon-estimate} we have
\begin{enumerate}
\item $\| \mu_\epsilon\|_{C^{k-1,\alpha}(M_\epsilon)}\lesssim 1.$
\end{enumerate}
Furthermore, in the neck and exterior regions we have
\begin{enumerate}[resume]
\item $\|y \Psi_\epsilon^*\mu_\epsilon \|_{\mathscr C^{k-1,\alpha;1}(A_c)}\lesssim \epsilon$, 
\item $\iota_\epsilon^*\mu_\epsilon = \Sigma$.
\end{enumerate}
\end{prop}

\begin{proof}
Due to Lemmas \ref{lemma:nu-ext-reg} and \ref{lemma:nu-neck-reg}, the first claim is established once we have estimated $\rho_\epsilon^{-1}\mathcal H_{\overline\lambda_\epsilon}(\rho_\epsilon)$.
From \eqref{rho-transition} and \eqref{exterior-metric}, we see that both $\rho_\epsilon$ and $\lambda_\epsilon$ agree with the projections of $\rho$ and $g$, respectively, outside of $\Psi_\epsilon (A_{1/2}) \subset \Psi_\epsilon(A_c)$.
Thus it suffices to obtain an estimate in the neck region.
Noting that each term in $\rho_\epsilon^{-1}\mathcal H_{\overline\lambda_\epsilon}(\rho_\epsilon)$ is a contraction of $\rho_\epsilon^{-1}(\otimes^3\overline\lambda_\epsilon{}^{-1})\otimes (\otimes^4 d\rho_\epsilon) \otimes \Hess_{\overline\lambda_\epsilon}\rho_\epsilon$,
we see that the last estimate in Proposition \ref{prop:WAH} implies the desired uniform bound.

We now establish the estimate in the neck region.
First we consider $y\Psi_\epsilon^*\nu_\epsilon^\text{ext}$.
Since $\chi(r^2/8)$, $\chi(1/8r^2)$ are smooth and uniformly bounded on $A_c$, it suffices to estimate $y\alpha_{\epsilon,1}^*\nu$ and $y(\alpha_{\epsilon,2}\circ\mathcal I)^*\nu$.
Note that 
\begin{equation*}
y\alpha_{\epsilon,1}^*\nu  = \frac{1}{\epsilon} \alpha_{\epsilon,1}^*(\rho \nu).
\end{equation*}
Since $\rho\nu \in \mathscr C^{k-1,\alpha;1}(M)$, it follows from \eqref{norm-scaling} that
\begin{equation*}
\| y\alpha_{\epsilon,1}^*\nu\|_{\mathscr C^{k-1,\alpha;1}(A_c)}\lesssim \epsilon.
\end{equation*}
A similar computation, relying on uniform estimates on $A_c$ for the inversion map, yields a corresponding estimate for $y(\alpha_{\epsilon,2}\circ\mathcal I)^*\nu$.
This gives the desired estimate for $y\Psi_\epsilon^*\nu_\epsilon^\text{ext}$.

In order to estimate $y\Psi_\epsilon^*(\rho_\epsilon^{-1}\mathcal H_{\overline\lambda_\epsilon}(\rho_\epsilon) + \nu_\epsilon^\text{neck})$ we make use of \cite[Lemma 7.5(b)]{AHEM}, which implies that $\overline\lambda \mapsto \mathcal H_{\overline\lambda}(\rho)$ is a locally Lipschitz mapping $\mathscr C^{k,\alpha;2} \to \mathscr C^{k-1,\alpha;1}$.
Thus the boundedness of $F$ and its derivatives on $A_c$ implies that
\begin{equation*}
\| \mathcal H_{F^2(g_\text{E} + m_\epsilon)}(yF)
- \mathcal H_{F^2g_\text{E}}(yF)\|_{\mathscr C^{k-1,\alpha;1}(A_c)} \lesssim \|m_\epsilon \|_{\mathscr C^{k,\alpha;2}(A_c)}.
\end{equation*}
Furthermore, we have
\begin{equation*}
\|(j_\epsilon^{ab} \mathcal H_{F^2g_\text{E}}(yF)_{ab}) (g_\text{E} + m_\epsilon) \|_{\mathscr C^{k-1,\alpha;1}(A_c)} 
\lesssim \|j_\epsilon \|_{\mathscr C^{k,\alpha;2}(A_c)}.
\end{equation*}
Thus Proposition \ref{prop:estimate-m} implies the desired estimate.
\end{proof}

We now show that $\mu_\epsilon$ is close to being divergence-free as measured by the weight function $\tilde\rho_\epsilon$; see \eqref{define-tilde-rho} and \eqref{rho-tilde-equivalence}.
\begin{prop}
\label{prop:div-mu}
$\| \Div_{\lambda_\epsilon}\mu_\epsilon\|_{C^{k-2,\alpha}_2(M_\epsilon, \tilde\rho_\epsilon)}
\lesssim \epsilon$
\end{prop}

\begin{proof}
If restricted to the complement of $\Psi_\epsilon(A_{1/4})$, we have $\pi_\epsilon^*\lambda_\epsilon = g$ and $\pi_\epsilon^*\mu_\epsilon = \Sigma$.
Thus, as $\Div_g\Sigma =0$ by hypothesis, $\Div_{\lambda_\epsilon}\mu_\epsilon$ is supported on $\Psi_\epsilon(A_{1/4})\subset \Psi_\epsilon(A_c)$.
From Lemma \ref{lemma:pullback-weighted-norms} we have
$$
\|\Div_{\lambda_\epsilon}\mu_\epsilon\|_{C^{k-2,\alpha}_2(\Psi_\epsilon(A_c);\tilde\rho_\epsilon)}
\approx
\|\Psi_\epsilon^*(\Div_{\lambda_\epsilon}\mu_\epsilon)\|_{C^{k-2,\alpha}_2(A_c)}.
$$
Using \eqref{decompose-lambda}, \eqref{pullback-H-term}, and \eqref{define-nu-neck}, and adding and subtracting a term, we write
\begin{align*}
\Psi_\epsilon^*(\Div_{\lambda_\epsilon}\mu_\epsilon) 
=& \Div_{y^{-2}(g_\text{E} + m_\epsilon)}\left( (yF)^{-1}\mathcal H_{F^2(g_\text{E} + m_\epsilon)}(yF)\right)
\\
&
- \Div_{y^{-2}g_\text{E}}\left( (yF)^{-1}\mathcal H_{F^2g_\text{E} }(yF)\right)
\\
&
+ \Div_{y^{-2}g_\text{E}}\left( (yF)^{-1}\mathcal H_{F^2g_\text{E} }(yF)\right)
\\
&-\Div_{y^{-2}(g_\text{E} + m_\epsilon)}\left( (yF)^{-1}\mathcal H_{F^2g_\text{E}}(yF)\right)
\\
&+\frac13\Div_{y^{-2}(g_\text{E} + m_\epsilon)}\left( (yF)^{-1}(j_\epsilon^{ab} \mathcal H_{F^2g_\text{E}}(yF)_{ab}) (g_\text{E} + m_\epsilon) \right)
\\
& + \Psi_\epsilon^*(\Div_{\lambda_\epsilon}\nu_\epsilon^\text{ext}).
\end{align*}
We now invoke \cite[Lemma 7.5(c)]{AHEM}, which implies that $\overline\lambda \mapsto \Div_{\rho^{-2}\overline\lambda} (\rho^{-1}\mathcal H_{\overline\lambda}(\rho))$ is a locally Lipschitz mapping $\mathscr C^{k,\alpha;2} \to C^{k-2,\alpha}_2$. 
Thus
\begin{multline*}
\|\Div_{y^{-2}(g_\text{E} + m_\epsilon)}\left( (yF)^{-1}\mathcal H_{F^2(g_\text{E} + m_\epsilon)}(yF)\right)
\\
-\Div_{y^{-2}g_\text{E}}\left( (yF)^{-1}\mathcal H_{F^2g_\text{E} }(yF)\right)\|_{C^{k-2,\alpha}_2(A_c)}
\\
\lesssim \|m_\epsilon \|_{\mathscr C^{k,\alpha;2}(A_c)}\lesssim \epsilon.
\end{multline*}
Next, we apply \cite[Proposition 7.9]{AHEM} to the divergence operator,  concluding that
\begin{multline*}
\| \Div_{y^{-2}g_\text{E}}\left( (yF)^{-1}\mathcal H_{F^2g_\text{E} }(yF)\right)
\\
-\Div_{y^{-2}(g_\text{E} + m_\epsilon)}\left( (yF)^{-1}\mathcal H_{F^2g_\text{E}}(yF)\right)\|_{C^{k-2,\alpha}_2(A_c)}
\\
\lesssim \|m_\epsilon \|_{\mathscr C^{k,\alpha;2}(A_c)}\lesssim \epsilon.
\end{multline*}
Since $\Div_{\lambda_\epsilon}(u\lambda_\epsilon) = du$ for any function $u$,
\begin{multline*}
\Div_{y^{-2}(g_\text{E} + m_\epsilon)}\left( (yF)^{-1}(j_\epsilon^{ab} \mathcal H_{F^2g_\text{E}}(yF)_{ab}) (g_\text{E} + m_\epsilon) \right)
\\
=
d\left( y^2(yF)^{-1}(j_\epsilon^{ab} \mathcal H_{F^2g_\text{E}}(yF)_{ab}) \right).
\end{multline*}
Using this, together with Lemma \ref{lemma:Euclidean-H}(3), we easily see that
\begin{multline*}
\|\Div_{y^{-2}(g_\text{E} + m_\epsilon)}\left( (yF)^{-1}(j_\epsilon^{ab} \mathcal H_{F^2g_\text{E}}(yF)_{ab}) (g_\text{E} + m_\epsilon) \right)\|_{C^{k-2,\alpha}_2(A_c)}
\\
\lesssim  \|j_\epsilon \|_{\mathscr C^{k,\alpha;2}(A_c)}\lesssim \epsilon.
\end{multline*}

Thus it remains only to estimate
\begin{equation*}
\| \Psi_\epsilon^*(\Div_{\lambda_\epsilon}\nu_\epsilon^\text{ext}) \|_{C^{k-2,\alpha}_2(A_{1/4})}.
\end{equation*}
By inversion symmetry, it suffices to consider the set where $r>2$.
On this set we have 
\begin{align*}
\Psi_\epsilon^*(\Div_{\lambda_\epsilon}\nu_\epsilon^\text{ext})
&=\Div_{\Psi_\epsilon^*\lambda_\epsilon} (\chi(r^2/8) \alpha_{\epsilon,1}^*\nu)
\\
&=
\chi(r^2/8) \alpha_{\epsilon,1}^*(\Div_{g} \nu)
+
\alpha_{\epsilon,1}^*\nu(\grad_{\Psi_\epsilon^*\lambda_\epsilon} \chi(r^2/8), \cdot) 
\end{align*}
due to \eqref{exterior-metric}.
As the cutoff function is smooth and uniformly bounded, we may use \eqref{norm-scaling} with the indices $m$ and $j$ set to zero, together with Proposition \ref{prop:operator-estimates}, to conclude that 
\begin{align*}
\| \Psi_\epsilon^*(\Div_{\lambda_\epsilon}\nu_\epsilon^\text{ext}) \|_{C^{k-2,\alpha}_2(A_{1/4})}
&\lesssim \| \alpha_{1,\epsilon}^*\nu\|_{C^{k-1,\alpha}_2(A_{1/4})} + \| \alpha_{2,\epsilon}^*\nu\|_{C^{k-1,\alpha}_2(A_{1/4})}
\\
&\lesssim 
\epsilon^2 \|\nu\|_{C^{k-1,\alpha}_2(M)}
\\
&\lesssim \epsilon^2.
\qedhere
\end{align*}
\end{proof}

\subsection{The tensor $\sigma_\epsilon$}
The metric $\lambda_\epsilon$ and tensor $\mu_\epsilon$ satisfy regularity and boundary conditions suitable to apply the conformal method in order to obtain CMCSF hyperboloidal solutions to the constraint equations
as outlined in the introduction.
The first step in that procedure is to solve 
\begin{equation}
\label{vl-new}
L_{\lambda_\epsilon}W_\epsilon = (\Div_{\lambda_\epsilon}\mu_\epsilon)^\sharp
\end{equation}
 for a vector field $W_\epsilon$ and subsequently define a tensor field $\sigma_\epsilon$ by \eqref{sigma-intro}.
We now establish that this process can be accomplished with appropriate uniform estimates in $\epsilon$.
\begin{lemma}
\label{lemma:W-estimates}
For each $\epsilon$ satisfying \eqref{epsilon-estimate} there exists a unique vector field $W_\epsilon\in C^{k,\alpha}_2(M_\epsilon;\tilde\rho_\epsilon)$ satisfying \eqref{vl-new}.
Furthermore we have the uniform estimates
\begin{gather*}
\| W_\epsilon\|_{C^{k,\alpha}_2(M_\epsilon;\tilde\rho_\epsilon)}
\lesssim \epsilon,
\\
\| \mathcal D_{\lambda_\epsilon}W_\epsilon\|_{C^{k-1,\alpha}_2(M_\epsilon;\tilde\rho_\epsilon)}
\lesssim \epsilon,
\\
\| \tilde\rho_\epsilon  \mathcal D_{\lambda_\epsilon}W_\epsilon\|_{\mathscr C^{k-1,\alpha;1}(M_\epsilon;\tilde\rho_\epsilon)}
\lesssim \epsilon.
\end{gather*}
\end{lemma}

\begin{proof}
Since $\overline\lambda_\epsilon\in \mathscr C^{k,\alpha;2}(M_\epsilon)$ and $\nu_\epsilon^\text{neck} + \nu_\epsilon^\text{ext}\in C^{k-1,\alpha}_2(M_\epsilon)$ we may apply Proposition 6.3 of \cite{AHEM} to conclude that there exists a unique vector field $W_\epsilon\in C^{k,\alpha}_2(M_\epsilon)$ satisfying \eqref{vl-new}.

To obtain the uniform estimates we first use Proposition \ref{invertibility:main} to obtain
\begin{align*}
\|W_\epsilon\|_{C^{k,\alpha}_2(M_\epsilon, \tilde\rho_\epsilon)} 
&\lesssim
\|L_{\lambda_\epsilon}W_\epsilon\|_{C^{k-2,\alpha}_2(M_\epsilon, \tilde\rho_\epsilon)}
\\
&=
\|(\Div_{\lambda_\epsilon}\mu_\epsilon)^\sharp\|_{C^{k-2,\alpha}_2(M_\epsilon, \tilde\rho_\epsilon)}.
\end{align*}
Using Propositions \ref{prop:lambda-estimates} and \ref{prop:div-mu} we have
\begin{equation*}
\|(\Div_{\lambda_\epsilon}\mu_\epsilon)^\sharp\|_{C^{k-2,\alpha}_2(M_\epsilon, \tilde\rho_\epsilon)}
\leq \| \lambda_\epsilon^{-1} \|_{C^{k,\alpha}(M_\epsilon)} \| \Div_{\lambda_\epsilon}\mu_\epsilon\|_{C^{k-2,\alpha}_2(M_\epsilon, \tilde\rho_\epsilon)}
\lesssim \epsilon,
\end{equation*}
which proves the first estimate.
The second estimate then follows from Proposition \ref{prop:operator-estimates}, and the third from Proposition \ref{prop:basic-inclusions}\eqref{curly-to-weighted} .
\end{proof}

Using $W_\epsilon$, we define the symmetric, trace-free tensor $\sigma_\epsilon$ by
\begin{align*}
\sigma_\epsilon &= \rho_\epsilon^{-1}\mathcal H_{\overline\lambda_\epsilon}(\rho_\epsilon) + \nu_\epsilon^\text{neck} + \nu_\epsilon^\text{ext} + \mathcal D_{\lambda_\epsilon}W_\epsilon
\\
&= \mu_\epsilon + \mathcal D_{\lambda_\epsilon}W_\epsilon.
\end{align*}
\begin{prop}
\label{estimate:sigma}
We have
\begin{enumerate}
\item $\overline\sigma_\epsilon = \rho_\epsilon\sigma_\epsilon \in \mathscr C^{k-1,\alpha;1}(M_\epsilon)$,
\item $\Div_{\lambda_\epsilon}\sigma_\epsilon =0$, and
\item $\overline\sigma_\epsilon = \mathcal H_{\overline\lambda_\epsilon}(\rho_\epsilon)$ along $\partial M_\epsilon$.
\end{enumerate}
Furthermore we have the global estimate
\begin{enumerate}[resume]
\item $\| \sigma_\epsilon\|_{C^{k-1,\alpha}(M_\epsilon)}\lesssim 1.$
\end{enumerate}
Finally, in the neck and exterior regions we have the uniform estimates
\begin{enumerate}[resume]
\item $\|y \Psi_\epsilon^*\sigma_\epsilon \|_{\mathscr C^{k-1,\alpha;1}(A_c)}\lesssim \epsilon$, and
\item $\|\rho ( \iota_\epsilon^*\sigma_\epsilon - \Sigma) \|_{\mathscr C^{k-1,\alpha;1}(E_c)}\lesssim \epsilon$
\end{enumerate}

\end{prop}
\begin{proof}
The first three claims are direct consequences of the construction and regularity involved; see also \cite[Theorem 8.2]{AHEM}.
The remaining claims follow from Proposition \ref{prop:estimate-mu}  and Lemma \ref{lemma:W-estimates}.
\end{proof}

The following additional estimates on $ |\sigma_\epsilon|^2_{\lambda_\epsilon}$ are required in our analysis of the Lichnerowicz equation.
\begin{prop}
\label{prop:sigma-aux-estimates}
We have
\begin{gather}
\label{norm-sigma}
\| |\sigma_\epsilon|^2_{\lambda_\epsilon}\|_{C^{k-1,\alpha}(M_\epsilon)}\lesssim 1,
\\
\label{norm-sigma-ext}
\| \iota_\epsilon^*|\sigma_\epsilon|^2_{\lambda_\epsilon} - |\Sigma|^2_g \|_{C^{k-1,\alpha}_2(E_c)} \lesssim \epsilon,
\\
\label{norm-sigma-neck}
\| \Psi_\epsilon^*  |\sigma_\epsilon|^2_{\lambda_\epsilon} \|_{C^{k-1,\alpha}_2(A_c)}
\lesssim \epsilon,
\\
\label{norm-sigma-difference}
\| |\sigma_\epsilon|^2 - |\mu_\epsilon|^2 \|_{C^{k-1,\alpha}_2(M_\epsilon;\tilde\rho_\epsilon)}\lesssim \epsilon.
\end{gather}
\end{prop}

\begin{proof}
Note that $|\sigma_\epsilon|^2_{\lambda_\epsilon}$ is a contraction of
\begin{equation*}
\lambda_\epsilon^{-1} \otimes \sigma_\epsilon \otimes\sigma_\epsilon.
\end{equation*}
This, together with Propositions \ref{prop:lambda-estimates} and \ref{estimate:sigma}, yields \eqref{norm-sigma}.
For \eqref{norm-sigma-neck}, we view $\Psi_\epsilon^*  |\sigma_\epsilon|^2_{\lambda_\epsilon}$ as a contraction of
\begin{equation*}
y^{-2}(\Psi_\epsilon^*\lambda_\epsilon)^{-1} \otimes y \Psi_\epsilon^*\sigma_\epsilon  \otimes y \Psi_\epsilon^*\sigma_\epsilon
\end{equation*}
and use Propositions \ref{prop:WAH} and \ref{estimate:sigma}.
Analogous reasoning yields \eqref{norm-sigma-ext}.
For the final estimate we note that $|\sigma_\epsilon|^2 - |\mu_\epsilon|^2$ is a contraction of
\begin{equation*}
\lambda_\epsilon^{-1}
\otimes
\lambda_\epsilon^{-1}
\otimes(\sigma_\epsilon + \mu_\epsilon)
\otimes (\sigma_\epsilon - \mu_\epsilon)
=
\lambda_\epsilon^{-1}
\otimes
\lambda_\epsilon^{-1}
\otimes(\sigma_\epsilon + \mu_\epsilon)
\otimes \mathcal D_{\lambda_\epsilon}W_\epsilon.
\end{equation*}
Using Propositions \ref{prop:lambda-estimates}, \ref{prop:estimate-mu}, and \ref{estimate:sigma}, together with Lemma \ref{lemma:W-estimates}, we have
\begin{align*}
\| |\sigma_\epsilon|^2 &- |\mu_\epsilon|^2 \|_{C^{k-1,\alpha}_2(M_\epsilon;\tilde\rho_\epsilon)}
\\
&\leq 
\| \lambda_\epsilon^{-1}\|_{C^{k,\alpha}(M_\epsilon)}^2
\| \sigma_\epsilon + \mu_\epsilon\|_{C^{k-1,\alpha}(M_\epsilon)}
\| \mathcal D_{\lambda_\epsilon}W_\epsilon \|_{C^{k-1,\alpha}_2(M_\epsilon;\tilde\rho_\epsilon)}
\\
&\lesssim \epsilon,
\end{align*}
which concludes the proof.
\end{proof}

\section{Analysis of the Lichnerowicz equation}
\label{sec:lich}
As discussed in \S\ref{subsec:overview}, the results of \cite{AHEM} (see also \cite{WAH}) imply that for each $\lambda_\epsilon$ and $\sigma_\epsilon$ there exists a positive solution $\phi_\epsilon$ to the Lichnerowicz equation \eqref{lichnero}; the solution $\phi_\epsilon$ is the unique such function satisfying $\phi_\epsilon - 1\in C^{k,\alpha}_1(M_\epsilon)$.
In this section we establish the following estimates for $\phi_\epsilon$ in the exterior region $E_c$ and gluing region $A_c$.

\begin{prop}
\label{prop:main-phi-estimate}
Suppose that \eqref{epsilon-estimate} holds.
There exists $\epsilon_*>0$, depending on $c$, such that for $0<\epsilon < \epsilon_*$ we have
\begin{equation}
\begin{gathered}
\label{main-phi-estimate}
\|\iota_\epsilon^*\phi_\epsilon - 1\|_{\mathscr C^{k,\alpha;2}(E_c)}\lesssim \epsilon,
\\
\|\Psi_\epsilon^*\phi_\epsilon - 1\|_{\mathscr C^{k,\alpha;2}(A_c)}\lesssim \epsilon.
\end{gathered}
\end{equation}
\end{prop}
\noindent
The proof of Proposition \ref{prop:main-phi-estimate} appears in \S\ref{sec:last-lich} below.

The proof of Theorem \ref{main} makes use of the following, which is an immediate consequence of $\iota_\epsilon^*\phi_\epsilon$ and $\Psi_\epsilon^*\phi_\epsilon$ being uniformly bounded close to $1$, and thus away from zero.
\begin{cor}
For any integer $p$ we have
\begin{equation}
\label{cor:phi}
\begin{gathered}
\|\iota_\epsilon^*(\phi_\epsilon^p - 1)\|_{\mathscr C^{k,\alpha;2}(E_c)}\lesssim \epsilon,
\\
\|\Psi_\epsilon^*(\phi_\epsilon^p - 1)\|_{\mathscr C^{k,\alpha;2}(A_c)}\lesssim \epsilon.
\end{gathered}
\end{equation}
\end{cor}

The proof of Proposition \ref{prop:main-phi-estimate}, which is completed at the end of \S\ref{sec:last-lich} below, follows the arguments in \cite{ILS-Gluing} and \cite{AllenStavrov-Density}, and makes use of a function $\theta_\epsilon$
that approximately solves the Lichnerowicz equation \eqref{lichnero}.
We then use the linearization of \eqref{lichnero} at $\theta_\epsilon$ to estimate the difference $\phi_\epsilon - \theta_\epsilon$ by means of a fixed-point argument, and thus establish \eqref{main-phi-estimate}.

\subsection{The approximate solution $\theta_\epsilon$}
\label{subsec:approx-soln}
For $\epsilon$ satisfying \eqref{epsilon-estimate} we define the Lichnerowicz operator $\mathcal L_\epsilon$ to act on a function $\theta$ by
\begin{equation*}
\mathcal L_\epsilon(\theta) 
= \Delta_{\lambda_\epsilon}\theta
-\frac18 \R[\lambda_\epsilon]\theta
+\frac18 |\sigma_\epsilon|^2_{\lambda_\epsilon}\theta^{-7}
-\frac34\theta^5,
\end{equation*}
so that the Lichnerowicz equation \eqref{lichnero} can be written $\mathcal L_\epsilon(\phi_\epsilon) =0$.
In this subsection we establish the following.
\begin{prop}
\label{prop:approx-solution}
For each $\epsilon$ there exists a positive function $\theta_\epsilon$ with $\theta_\epsilon -1\in C^{k,\alpha}_1(M_\epsilon)$  such that  
\begin{gather}
\label{approximately-one}
\|\theta_\epsilon -1 \|_{C^{k,\alpha}_1(M_\epsilon;\tilde\rho_\epsilon)}\lesssim \epsilon,
\\
\label{theta-approx-solution}
\|\mathcal L_\epsilon(\theta_\epsilon) \|_{C^{k-2,\alpha}_\delta(M_\epsilon;\tilde\rho_\epsilon)} 
\leq  \epsilon C_\delta,
\quad \delta \in \{0,2\},
\end{gather}
for constants $C_\delta$ independent of $\epsilon$ satisfying \eqref{epsilon-estimate}.
Furthermore we have improved regularity in the exterior and neck regions:
\begin{gather}
\label{improved-approximately-one-ext}
\|\iota_\epsilon^*\theta_\epsilon - 1\|_{\mathscr C^{k,\alpha;2}(E_c)}\lesssim \epsilon,
\\
\label{improved-approximately-one-neck}
\|\Psi_\epsilon^*\theta_\epsilon - 1\|_{\mathscr C^{k,\alpha;2}(A_c)}\lesssim \epsilon.
\end{gather}
\end{prop}
To prove Proposition \ref{prop:approx-solution} we need to establish a number of lemmas.
We first show that in the exterior region, it suffices to take the constant function $1$ as the approximate solution $\theta_\epsilon$.
\begin{lemma}
\label{lemma:one-outside}
We have
\begin{equation*}
\|\mathcal L_\epsilon(1)\|_{C^{k,\alpha}_2(M_\epsilon \setminus \Psi_\epsilon(A_{1/4});\tilde\rho_\epsilon)} \lesssim \epsilon.
\end{equation*}
\end{lemma}

\begin{proof}
On $M_\epsilon \setminus \Psi_\epsilon(A_{1/4})$ we have $\lambda_\epsilon = (\pi_\epsilon)_*g$ and $\mu_\epsilon = (\pi_\epsilon)_*\Sigma$.
Since $g$ and $\Sigma$ satisfy the constraint equations \eqref{CMC-constraints}, we have 
\begin{equation*}
\R[\lambda_\epsilon]  = -6 + |\mu_\epsilon|^2_{\lambda_\epsilon}.
\end{equation*}
Thus on $M_\epsilon \setminus \Psi_\epsilon(A_{1/4})$ we have
\begin{equation*}
\mathcal L_\epsilon(1) = \frac18\left( |\sigma_\epsilon|^2_{\lambda_\epsilon} - |\mu_\epsilon|^2_{\lambda_\epsilon}\right).
\end{equation*}
The desired estimate now follows from Proposition \ref{prop:sigma-aux-estimates}.
\end{proof}

In the neck region we cannot simply set $\theta_\epsilon$ equal to $1$ because the scalar curvature of the spliced metric $\lambda_\epsilon$ need not be close to $-6$ in $C^{k-2,\alpha}_2(M_\epsilon;\tilde\rho_\epsilon)$.
Rather, we seek an approximate solution $\theta_\epsilon$ that is a perturbation of the constant function $1$, with the perturbation supported in the neck region.
Before giving a careful definition of $\theta_\epsilon$, we establish some preliminary results.

Observe that for any function $v$ we have
\begin{equation}
\label{perturb-lich}
\mathcal L_\epsilon (1+v)
= \Delta_{\lambda_\epsilon}v - 3v - \frac18 (\R[\lambda_\epsilon] + 6)+ \mathscr R_\epsilon(v),
\end{equation}
where the remainder term $\mathscr R_\epsilon(v)$ is given by
\begin{equation*}
\mathscr R_\epsilon(v)
= -\frac18(\R[\lambda_\epsilon] + 6)v
+\frac18 |\sigma_\epsilon|^2_{\lambda_\epsilon} (1+v)^{-7}
-\frac34\left( (1+v)^5 - 1 - 5v \right).
\end{equation*}
Using $\Psi_\epsilon ^*\lambda_\epsilon = y^{-2}(g_\text{E} + m_\epsilon)$, the formula for how the scalar curvature changes under a conformal change of the metric yields
\begin{equation}
\label{conformal-R}
\Psi_\epsilon^* \R[\lambda_\epsilon]+6
= 6v_\epsilon 
+ y^2 \R[g_\text{E} + m_\epsilon],
\end{equation}
where
\begin{equation}
\label{define-v-epsilon}
v_\epsilon = 1-|dy|^2_{g_\text{E} + m_\epsilon} + \frac23 y \Delta_{g_\text{E} + m_\epsilon} y.
\end{equation}
Note that Proposition \ref{prop:lambda-estimates} implies that $\R[g_\text{E} + m_\epsilon]\in C^{k-2,\alpha}(A_c)$ with 
\begin{equation}
\label{estimate-scalar-bar}
\|y^2 \R[g_\text{E} + m_\epsilon]\|_{C^{k-2,\alpha}_2(A_c)} 
\approx \| \R[g_\text{E} + m_\epsilon]\|_{C^{k-2,\alpha}(A_c)} 
 \lesssim \epsilon.
\end{equation}

\begin{lemma}
\label{lemma:v-epsilon}
We have that
\begin{enumerate}
\item $v_\epsilon$  vanishes where $y=0$,
\item $v_\epsilon\in \mathscr C^{k-1,\alpha;2}(A_c)$ with $
\|v_\epsilon\|_{ \mathscr C^{k-1,\alpha;2}(A_c)}\lesssim\epsilon,
$

\item $v_\epsilon\in C^{k-1,\alpha}_1(A_c)$ with
$
\|v_\epsilon \|_{C^{k-1,\alpha}_1(A_c)} 
 \lesssim \epsilon,
$

\item $v_\epsilon \in C^{k-2,\alpha}_2(A_c\setminus A_{1/2})$ with 
$
\|v_\epsilon \|_{C^{k-2,\alpha}_2(A_c\setminus A_{1/2})} 
 \lesssim \epsilon.
$
\end{enumerate}
\end{lemma}
\begin{proof}
From Proposition \ref{prop:g-neck} we have $v_\epsilon \in C^{k-1,\alpha}_1(A_c)$, which implies that $v_\epsilon$ vanishes where $y=0$.
The second and third claims also follow directly from Proposition \ref{prop:g-neck}.

To establish the final claim we first note that by assumption we have $\rho\Sigma \in C^{k-1,\alpha}_2(M)$ and $g^{-1}\in C^{k,\alpha}(M)$.
Thus \eqref{CMC-constraints} implies that
\begin{equation*}
\R[g]+6 = |\Sigma|^2_g \in C^{k-1,\alpha}_2(M).
\end{equation*}
Since $\lambda_\epsilon = (\pi_\epsilon)_* g$ on $\Psi_\epsilon(A_c\setminus A_{1/2})$, this implies that 
\begin{equation*}
\Psi_\epsilon^*\R[\lambda_\epsilon]+6\in {C^{k-1,\alpha}_2(A_c\setminus A_{1/2})}.
\end{equation*}
Taken together with \eqref{estimate-scalar-bar} and \eqref{conformal-R}, this implies that $v_\epsilon \in C^{k-2,\alpha}_2(A_c\setminus A_{1/2})$.
Thus $v_\epsilon\to 0$ and $|{}^\text{E}\nabla v_\epsilon|_{g_\text{E}}\to 0$ as $y\to 0$, which, in view of 
Proposition \ref{prop:basic-inclusions}\eqref{IvasMagicLemma}, yields
\begin{equation*}
\|v_\epsilon \|_{C^{k-2,\alpha}_2(A_c\setminus A_{1/2})} 
\leq C \|v_\epsilon\|_{ \mathscr C^{k-1,\alpha;2}(A_c)}.
\end{equation*}
The final claim now follows from the second.
\end{proof}

We use Lemma \ref{lemma:v-epsilon} to estimate the remainder term $\mathscr R_\epsilon(v)$ in \eqref{perturb-lich}.
\begin{lemma}
\label{lemma:lich-remainder}
If $v\in C^{k,\alpha}_1(M_\epsilon)$ and if $v$ satisfies
$$
|v|<\frac12
\quad\text{ and }\quad
\| v\|_{C^{k,\alpha}_1(M_\epsilon;\tilde\rho_\epsilon)} \lesssim \epsilon
$$
then 
$$
\|\Psi_\epsilon^* \mathscr R_\epsilon(v) \|_{C^{k-2,\alpha}_2(A_c)}\lesssim \epsilon.
$$
\end{lemma}
\begin{proof}
It is easy to see that 
\begin{equation*}
\| (1+v)^5 - 1 - 5v\|_{C^{k-2,\alpha}_2(M_\epsilon;\tilde\rho_\epsilon)}
\leq 26 \| v\|_{C^{k,\alpha}_1(M_\epsilon;\tilde\rho_\epsilon)}^2 \lesssim \epsilon.
\end{equation*}
Since $|v|<\frac12$ we have
\begin{equation*}
\|(1+v)^{-7} \|_{C^{k-2,\alpha}(M_\epsilon;\tilde\rho_\epsilon)} \lesssim 1.
\end{equation*}
Thus \eqref{norm-sigma-neck} implies that
\begin{equation*}
\| |\sigma_\epsilon|^2_{\lambda_\epsilon} (1+v)^{-7} \|_{C^{k-2,\alpha}_2(M_\epsilon;\tilde\rho_\epsilon)} \lesssim \epsilon.
\end{equation*}
To estimate the scalar curvature term we use \eqref{conformal-R} to write
\begin{equation*}
(\R[\lambda_\epsilon] + 6)v = 6 v_\epsilon v + y^2 \R[g_\text{E} + m_\epsilon]v.
\end{equation*}
Thus Lemma \ref{lemma:v-epsilon} and \eqref{estimate-scalar-bar} imply that 
\begin{align*}
\|(\R[\lambda_\epsilon]& + 6)v\|_{C^{k-2,\alpha}_2(A_c)}
\\
&\lesssim  
\|v_\epsilon\|_{C^{k-2,\alpha}_1(A_c)}
\|v\|_{C^{k-2,\alpha}_1(A_c)} 
\\
&\qquad+ 
\| y^2 \R[g_\text{E} + m_\epsilon] \|_{C^{k-2,\alpha}_2(A_c)} \| v\|_{C^{k-2,\alpha}(A_c)}
\\
&\lesssim \epsilon.
\qedhere
\end{align*}
\end{proof}

We construct a regularization of $v_\epsilon$ that is supported in $A_c$.
Let $\eta$ be a smooth cutoff function on $A_c$ that is supported on $A_{1/6}$ and such that $\eta = 1$ on $A_{1/4}$.
We now apply Proposition \ref{prop:regularization} in order to obtain a function $\tilde\tau_\epsilon$ that approximates $\eta \Delta_{g_\text{E} + m_\epsilon}y$ in the following sense:  
\begin{lemma}
\label{lemma:tau-epsilon}
There exists $\tilde\tau_\epsilon \in \mathscr C^{k,\alpha;1}(\mathbb H)$ such that
\begin{enumerate}
\item $\tilde\tau_\epsilon$ is supported in $A_{1/8}$,
\item $\| \tilde\tau_\epsilon \|_{\mathscr C^{k,\alpha;1}(\mathbb H)}\lesssim\epsilon$,
\item $\|\tilde\tau_\epsilon - \eta \Delta_{g_\text{E} + m_\epsilon}y\|_{C^{k-1,\alpha}_1(\mathbb H)} \lesssim \epsilon$.
\end{enumerate}

\end{lemma}
\begin{proof}
It follows from the definition of $\eta$ and from Proposition \ref{prop:WAH} that $\eta\Delta_{g_\text{E} + m_\epsilon}y \in \mathscr C^{k-1,\alpha;1}(\mathbb H)$ is supported in $A_{1/6}$ and satisfies 
\begin{equation*}
\| \eta\Delta_{g_\text{E} + m_\epsilon}y\|_{\mathscr C^{k-1,\alpha;1}(\mathbb H)}\lesssim \epsilon.
\end{equation*}
We obtain $\tilde\tau_\epsilon$ by applying Proposition \ref{prop:regularization} to the function $\eta\Delta_{g_\text{E} + m_\epsilon}y$.
 Proposition \ref{prop:regularization} immediately implies the first and second claims, and also implies that $\tilde\tau_\epsilon - \eta \Delta_{g_\text{E} + m_\epsilon}y\in C^{k-1,\alpha}_1(\mathbb H;\breve\rho)$.
In view of part \eqref{IvasMagicLemma} of Proposition \ref{prop:basic-inclusions}, this latter fact, together with the second claim, implies the third claim.
\end{proof}

We now define $\tilde v_\epsilon$, the regularization of $\eta v_\epsilon$, by
\begin{equation*}
\tilde v_\epsilon = \eta (1-|dy|^2_{g_\text{E} + m_\epsilon}) +\frac23 y \tilde\tau_\epsilon.
\end{equation*}
\begin{lemma}
\label{lemma:tilde-v}
The function $\tilde v_\epsilon$ satisfies
\begin{enumerate}
\item $\tilde v_\epsilon$ vanishes where $y=0$,

\item $\tilde v_\epsilon \in \mathscr C^{k,\alpha;2}(\mathbb H)$
with $\| \tilde v_\epsilon \|_{ \mathscr C^{k,\alpha;2}(\mathbb H)}\lesssim \epsilon$,

\item $\tilde v_\epsilon \in  C^{k,\alpha}_1(\mathbb H)$
with $\| \tilde v_\epsilon \|_{  C^{k,\alpha}_1(\mathbb H)}\lesssim \epsilon$,

\item  $\eta v_\epsilon - \tilde v_\epsilon \in  C^{k-1,\alpha}_2(\mathbb H)$
with $\|\eta v_\epsilon -  \tilde v_\epsilon \|_{  C^{k-1,\alpha}_2(\mathbb H)}\lesssim \epsilon$,

\item $\tilde v_\epsilon \in  C^{k-1,\alpha}_2(A_c \setminus A_{1/2})$
with $\| \tilde v_\epsilon \|_{  C^{k-1,\alpha}_2(A_c \setminus A_{1/2})}\lesssim \epsilon$.
\end{enumerate}

\end{lemma}

\begin{proof}
The first claim follows from \eqref{neck-ah-identity} and the boundedness of $\tilde\tau_\epsilon$.
The second and third claims follow from analogous estimates in Proposition \ref{prop:g-neck} and Lemma \ref{lemma:tau-epsilon}.
For the fourth claim we note that 
\begin{equation*}
\eta v_\epsilon - \tilde v_\epsilon = \frac23 y (\eta\Delta_{g_\text{E} + m_\epsilon} y -\tilde\tau_\epsilon);
\end{equation*}
thus the desired estimate follows from Lemma \ref{lemma:tau-epsilon}.
For the last claim we write $\tilde v_\epsilon = (\tilde v_\epsilon - \eta v_\epsilon) + \eta v_\epsilon$ and apply the final statements of Lemmas \ref{lemma:tau-epsilon} and \ref{lemma:v-epsilon}.
\end{proof}

We now define the approximate solution $\theta_\epsilon$ by requiring that $\theta_\epsilon = 1$ outside $\Psi_\epsilon(A_{1/8})$ and that $\Psi_\epsilon^*\theta_\epsilon = 1-\frac{3}{16}\tilde v_\epsilon$; note that $\theta_\epsilon$ is well-defined since $\tilde v_\epsilon$ is supported on $A_{1/8}$.
With this definition in hand, we may use Lemmas \ref{lemma:one-outside} through \ref{lemma:tilde-v} to prove Proposition \ref{prop:approx-solution}, showing that $\theta_\epsilon$ is in fact an approximate solution to the Lichnerowicz equation.

\begin{proof}[Proof of Proposition \ref{prop:approx-solution}]
The estimate \eqref{approximately-one} follows from Lemmas \ref{lemma:tilde-v} and \ref{lemma:pullback-weighted-norms}, and the fact that $\tilde v_\epsilon$ is supported in $A_{1/8}$.

We next establish \eqref{theta-approx-solution}.
In view of Lemma \ref{lemma:one-outside}, it suffices to estimate $\| \Psi_\epsilon^* \mathcal L_\epsilon(\theta_\epsilon) \|_{C^{k-2,\alpha}_2(A_{c})}$.
Using \eqref{perturb-lich} and \eqref{conformal-R} we have
\begin{multline}
\label{L-theta-epsilon}
\Psi_\epsilon^* \mathcal L_\epsilon(\theta_\epsilon)
=
-\frac{3}{16}\left(\Delta_{y^{-2}(g_\text{E} + m_\epsilon)} \tilde v_\epsilon + \tilde v_\epsilon\right)
+\frac34\left( \tilde v_\epsilon - v_\epsilon \right)
\\
-\frac18 y^2 \R[g_\text{E} + m_\epsilon]
+ \Psi_\epsilon^* \mathscr R_\epsilon \left(-\frac{3}{16}\tilde v_\epsilon\right).
\end{multline}
The final two terms in \eqref{L-theta-epsilon} are easily estimated in $C^{k-2,\alpha}_2(A_{c})$ using \eqref{estimate-scalar-bar} and Lemmas \ref{lemma:lich-remainder} and \ref{lemma:tilde-v}.
Since $\eta =1$ on $A_{1/4}$ we can estimate the second term in \eqref{L-theta-epsilon} by
\begin{multline*}
\| \tilde v_\epsilon -  v_\epsilon \|_{C^{k-2,\alpha}_2(A_{c})}
\leq 
\| \tilde v_\epsilon \|_{C^{k-2,\alpha}_2(A_{c}\setminus A_{1/2})}
\\
+ \|  v_\epsilon \|_{C^{k-2,\alpha}_2(A_{c}\setminus A_{1/2})}
+ \|\eta v_\epsilon -  \tilde v_\epsilon  \|_{C^{k-2,\alpha}_2(A_{c})},
\end{multline*}
which in turn is controlled by Lemmas \ref{lemma:v-epsilon} and \ref{lemma:tilde-v}.

In order to estimate $\|\Delta_{y^{-2}(g_\text{E} + m_\epsilon)} \tilde v_\epsilon + \tilde v_\epsilon \|_{C^{k-2,\alpha}_2(A_{c})}$ we use the identity
\begin{equation}
\label{gustav}
\Delta_{y^{-2}(g_\text{E} + m_\epsilon)} \tilde v_\epsilon + \tilde v_\epsilon
=
y^2 \Delta_{g_\text{E} + m_\epsilon} \tilde v_\epsilon
+\tilde v_\epsilon
-y \langle d \tilde v_\epsilon, dy\rangle_{g_\text{E} + m_\epsilon}.
\end{equation}
Using Lemma \ref{lemma:tilde-v} and Proposition \ref{prop:WAH} we estimate the first term in \eqref{gustav} as follows
\begin{align*} 
\| y^2 \Delta_{g_\text{E} + m_\epsilon} \tilde v_\epsilon \|_{C^{k-2,\alpha}_2(A_{c})}
&\approx
\| \Delta_{g_\text{E} + m_\epsilon} \tilde v_\epsilon \|_{C^{k-2,\alpha}(A_{c})}
\\
&\lesssim
\|  \tilde v_\epsilon \|_{\mathscr C^{k,\alpha;2}(A_c)}
\\
&\lesssim \epsilon.
\end{align*}
Note that Lemma \ref{lemma:tilde-v} implies that the expression
\begin{equation*}
\tilde v_\epsilon - y \langle d \tilde v_\epsilon, dy\rangle_{g_\text{E} + m_\epsilon}
\end{equation*}
satisfies the hypotheses of Proposition \ref{prop:bdy-taylor} and thus
\begin{equation*}
\|\tilde v_\epsilon - y \langle d \tilde v_\epsilon, dy\rangle_{g_\text{E} + m_\epsilon}\|_{C^{k-2,\alpha}_2(A_{c})}
\lesssim 
\| \tilde v_\epsilon \|_{\mathscr C^{k,\alpha;2}(A_c)}
\lesssim \epsilon.
\end{equation*}
Thus \eqref{theta-approx-solution} is established for $\delta =2$.

Note that for any function $u$ we have
\begin{equation*}
\|u\|_{C^{k-2,\alpha}(A_c)} 
=\|\breve\rho^2 u\|_{C^{k-2,\alpha}_2(A_c)} 
\lesssim  \|\breve\rho^2 \|_{C^{k-2,\alpha}(A_c)} \|u\|_{C^{k-2,\alpha}_2(A_c)}
\end{equation*}
and thus the $\delta =0$ estimate in \eqref{theta-approx-solution} follows from the estimate with $\delta=2$.

Finally,  \eqref{improved-approximately-one-neck} follows from the second claim of Lemma \ref{lemma:tilde-v}, while  \eqref{improved-approximately-one-ext} holds trivially due to our definition that $\theta_\epsilon =1$ outside $\Psi_\epsilon(A_{1/8})$.
\end{proof}

\subsection{Linearization of the Lichnerowicz equation}
Let $\mathcal P_\epsilon[\theta]$ denote the linearization of the Lichnerowicz operator $\mathcal L_\epsilon$ at a function $\theta$.
We have
\begin{equation}
\label{linear-lich}
\mathcal P_\epsilon[\theta] u = \Delta_{\lambda_\epsilon}u 
- \frac18\left(\R[\lambda_\epsilon] +7|\sigma_\epsilon|^2_{\lambda_\epsilon}\theta^{-8} +30\theta^4 \right)u.
\end{equation}

\begin{prop}
\label{prop:uniformly-invert-Pe}
Suppose $-1< \delta < 3$ and let $\theta_\epsilon$ be the function given by Proposition \ref{prop:approx-solution}.
Then there exists $\epsilon_*>0$ such that if $0<\epsilon < \epsilon_*$, then the operator
\begin{equation*}
\mathcal P_\epsilon[\theta_\epsilon]\colon C^{k,\alpha}_\delta(M_\epsilon;\tilde\rho_\epsilon) \to C^{k-2,\alpha}_\delta(M_\epsilon;\tilde\rho_\epsilon)
\end{equation*}
is invertible, and there exists a constant $K_\delta$
independent of $\epsilon$ such that the  operator norm of $\mathcal P_\epsilon[\theta_\epsilon]^{-1}\colon   C^{k-2,\alpha}_\delta(M_\epsilon;\tilde\rho_\epsilon)\to C^{k,\alpha}_\delta(M_\epsilon;\tilde\rho_\epsilon)$ satisfies
\begin{equation}
    \label{op-norm-est}
\|\mathcal P_\epsilon[\theta_\epsilon]^{-1}\|_\delta \le K_\delta.
\end{equation}
\end{prop}

\begin{proof}
From Proposition \ref{invertibility:main} we know that $\mathcal P_\epsilon[1]$ is uniformly invertible.
Thus it remains to show that
\begin{equation}
\label{invert-suffices}
\|\mathcal P_\epsilon[1]u - \mathcal P_\epsilon[\theta_\epsilon]u\|_{C^{k-2,\alpha}_\delta(M_\epsilon;\tilde\rho_\epsilon)}
\lesssim 
\epsilon \|u\|_{C^{k,\alpha}_\delta(M_\epsilon;\tilde\rho_\epsilon)}
\end{equation}
for all $u\in C^{k,\alpha}_\delta(M_\epsilon;\tilde\rho_\epsilon)$.
We have 
\begin{equation*}
\mathcal P_\epsilon[1]u - \mathcal P_\epsilon[\theta_\epsilon]u
= \left(\frac78|\sigma_\epsilon|_{\lambda_\epsilon}^2(1-\theta_\epsilon^{-8})  + \frac{15}{4}(1-\theta_\epsilon^4)\right)u.
\end{equation*}
Recall that from  \eqref{norm-sigma} we have
$
\| |\sigma_\epsilon|_{\lambda_\epsilon}^2 \|_{C^{k-2,\alpha}(M_\epsilon)} 
\lesssim 1.
$
From \eqref{approximately-one} we can choose $\epsilon_*$ small enough that
\begin{equation*}
\|1-\theta_\epsilon^{-8}\|_{C^{k-2,\alpha}(M_\epsilon)}
\lesssim\epsilon
\quad\text{ and }\quad
\|1-\theta_\epsilon^{4}\|_{C^{k-2,\alpha}(M_\epsilon)}
\lesssim\epsilon,
\end{equation*}
from which \eqref{invert-suffices} follows.
\end{proof}

We define the error term $\mathcal Q_\epsilon(u)$ by
\begin{equation}
\label{define-Q}
\mathcal L_\epsilon(\theta_\epsilon + u) =\mathcal L_\epsilon(\theta_\epsilon)+\mathcal P_\epsilon[\theta_\epsilon](u) + \mathcal Q_\epsilon(u).
\end{equation}
In order to describe the mapping properties of $\mathcal Q_\epsilon$, we use the following.
\begin{prop}
\label{prop:Q-mapping}
There exist $r_*>0$, $\epsilon_*>0$, and $D_*$ such that for each $\delta\geq 0$ and  $0<\epsilon < \epsilon_*$ we have 
\begin{multline}
\label{Q-diff}
\|  \mathcal Q_\epsilon(u_1) - \mathcal Q_\epsilon(u_2)\|_{C^{k-1,\alpha}_\delta(M_\epsilon ;\tilde \rho_\epsilon)}
\\
\leq  D_*\| u_1 - u_2\|_{C^{k,\alpha}_\delta(M_\epsilon ;\tilde \rho_\epsilon)}
\left(\| u_1\|_{C^{k,\alpha}(M_\epsilon)}+\| u_2\|_{C^{k,\alpha}(M_\epsilon)}\right)
\end{multline}
for all functions $u_1,u_2\in C^{k,\alpha}_\delta(M_\epsilon;\tilde\rho_\epsilon)$ with
$\|u_i\|_{C^{k,\alpha}(M_\epsilon)}\leq r_*$.
\end{prop}

\begin{proof}
Note that 
\begin{equation}
\label{Q-f}
\mathcal Q_\epsilon(u) = f(\theta_\epsilon + u) - f(\theta_\epsilon) - f^\prime(\theta_\epsilon) u,
\end{equation}
where 
$$
f(x) = \frac18|\sigma_\epsilon|^2_{\lambda_\epsilon} x^{-7} - \frac34 x^5.
$$
We now make use of the integral form of Taylor's remainder formula
\begin{equation}
\label{Taylor}
f(b) - f(a) =(b-a) f^\prime(a) + (b-a)^2 \int_0^1 (1-t) f^{\prime\prime}\big(a + t(b-a)\big)dt.
\end{equation}
First consider \eqref{Taylor} with $a = \theta_\epsilon + u_1$ and $b = \theta_\epsilon + u_2$, and then consider \eqref{Taylor} with $a = \theta_\epsilon + u_2$ and $b=\theta_\epsilon + u_1$.
Taking the difference of \eqref{Taylor} with these two choices of $a$ and $b$ and then using \eqref{Q-f} we find that
\begin{multline*}
Q_\epsilon(u_1) - Q_\epsilon(u_2)
= \frac12(u_1 - u_2)\big( f^\prime(\theta_\epsilon + u_1) -2 f^\prime(\theta_\epsilon) +   f^\prime(\theta_\epsilon + u_2)\big)
\\
+\frac12 (u_1 - u_2)^2 \int_0^1(1-t)\Big( f^{\prime\prime}(\theta_\epsilon + u_2 - t(u_2-u_1)) 
\\- f^{\prime\prime}(\theta_\epsilon + u_2 - (1-t)(u_2-u_1))\Big)dt.
\end{multline*}
Using the fundamental theorem of calculus we write this expression as
\begin{equation*}
\begin{aligned}
Q_\epsilon(u_1) &- Q_\epsilon(u_2)
\\
=&  (u_1 - u_2) u_1 \frac12 \int_0^1 f^{\prime\prime}(\theta_\epsilon + tu_1)dt
\\
&-(u_1 - u_2) u_2 \frac12 \int_0^1  f^{\prime\prime}(\theta_\epsilon + tu_2)dt 
\\
&+  (u_1 - u_2)^2\frac12 \int_0^1(1-t) f^{\prime\prime}(\theta_\epsilon + u_2 - t(u_2-u_1))\,dt
\\
&- (u_1 - u_2)^2\frac12 \int_0^1 (1-t)f^{\prime\prime}(\theta_\epsilon + u_2 - (1-t)(u_2-u_1))dt.
\end{aligned}
\end{equation*}

From Proposition \ref{prop:sigma-aux-estimates} we have that $|\sigma_\epsilon|^2_{\lambda_\epsilon}$ is bounded in $C^{k-1,\alpha}(M_\epsilon)$.
Using \eqref{approximately-one} we can choose $\epsilon_*$ to ensure that $|\theta_\epsilon|$ is uniformly bounded away from zero.
Thus we can choose $r_*$ sufficiently small that each of the four integrals above is bounded in $C^{k-1,\alpha}(M_\epsilon)$, which concludes the proof.
\end{proof}

\subsection{Proof of Proposition \ref{prop:main-phi-estimate}}
\label{sec:last-lich}
In order to prove Proposition \ref{prop:main-phi-estimate} we first establish an estimate for the difference between the solution $\phi_\epsilon$ to the Lichnerowicz equation and the approximate solution $\theta_\epsilon$ defined in \S\ref{subsec:approx-soln}.
Our strategy is to use a contraction-mapping argument.
For each $r>0$ let
\begin{equation*}
\overline B^{k,\alpha}_\delta(r) = \{ u \in C^{k,\alpha}_\delta(M_\epsilon) \colon \|u\|_{C^{k,\alpha}_\delta(M_\epsilon;\tilde\rho_\epsilon)} \leq  r\}.
\end{equation*}
For $\epsilon >0$ we define
\begin{equation*}
X_\epsilon = \overline B^{k,\alpha}_2(2K_2 C_2\epsilon)
 \cap \overline B^{k,\alpha}_0(2K_0  C_0\epsilon),
\end{equation*}
where $C_2$, $C_0$ are the constants appearing in \eqref{theta-approx-solution}, 
and $K_2$, $K_0$ are those appearing in 
\eqref{op-norm-est}.
Choosing the metric 
\begin{equation*}
d(u,v) = \| u-v\|_{C^{k,\alpha}_2(M_\epsilon;\tilde\rho_\epsilon)} 
+\| u-v\|_{C^{k,\alpha}(M_\epsilon)}
\end{equation*}
we find that $X_\epsilon$ is a complete metric space.

From \eqref{define-Q} we have that $\theta_\epsilon + u$ is a solution to the Lichnerowicz equation if 
\begin{equation*}
\mathcal L_\epsilon[\theta_\epsilon] u = - \left(\mathcal L_\epsilon(\theta_\epsilon) + \mathcal Q_\epsilon(u)\right).
\end{equation*}
This holds provided $u$ is a fixed point of the map
\begin{equation*}
\mathcal G_\epsilon \colon u\mapsto -\mathcal P_\epsilon[\theta_\epsilon]^{-1} \left(\mathcal L_\epsilon(\theta_\epsilon) + \mathcal Q_\epsilon(u)\right).
\end{equation*}

\begin{lemma}
\label{lemma:contraction}
We may choose $\epsilon_*$ such that 
for $0<\epsilon < \epsilon_*$ the map $\mathcal G_\epsilon$
is a contraction mapping on $X_\epsilon$.
\end{lemma}

\begin{proof}
Start by taking $\epsilon_*$ to be smaller than the choices made for this constant in Propositions \ref{prop:uniformly-invert-Pe} and \ref{prop:Q-mapping}. 
We first show that $\mathcal G_\epsilon$ maps $X_\epsilon$ to itself.
If $u\in X_\epsilon$, then taking $u_2=0$ in \eqref{Q-diff} implies that for $\delta \in \{0,2\}$ we have
\begin{equation*}
\|\mathcal Q_\epsilon(u)\|_{C^{k-2,\alpha}_\delta(M_\epsilon;\tilde\rho_\epsilon)}
\leq 4D_* K_\delta K_0C_\delta C_0\epsilon^2.
\end{equation*}
Thus from \eqref{op-norm-est} we have
\begin{equation*}
\|\mathcal P_\epsilon[\theta_\epsilon]^{-1}\mathcal Q_\epsilon(u)\|_{C^{k,\alpha}_\delta(M_\epsilon;\tilde\rho_\epsilon)}
\leq 4D_* K_\delta^2  K_0C_\delta C_0\epsilon^2.
\end{equation*}
From \eqref{op-norm-est} and \eqref{theta-approx-solution} we have
\begin{equation*}
\|\mathcal P_\epsilon[\theta_\epsilon]^{-1}\mathcal L_\epsilon(\theta_\epsilon)\|_{C^{k,\alpha}_\delta(M_\epsilon;\tilde\rho_\epsilon)}
\leq K_\delta C_\delta \epsilon.
\end{equation*}
Thus by choosing $\epsilon_*$ small enough, we can guarantee that $\mathcal G_\epsilon(u)\in X_\epsilon$.

To see that $\mathcal G_\epsilon$ is a contraction, suppose that $u_1,u_2\in X_\epsilon$.
Using \eqref{op-norm-est} and \eqref{Q-diff} we have
\begin{align*}
\|  \mathcal G_\epsilon(u_1) - \mathcal G_\epsilon(u_2)\|_{C^{k,\alpha}_\delta(M_\epsilon ;\tilde \rho_\epsilon)}
&\leq
K_\delta \|  \mathcal Q_\epsilon(u_1) - \mathcal Q_\epsilon(u_2)\|_{C^{k-2,\alpha}_\delta(M_\epsilon ;\tilde \rho_\epsilon)}
\\
&\leq K_\delta D_* 4 K_0 C_0 \epsilon \| u_1 - u_2\|_{C^{k,\alpha}_\delta(M_\epsilon ;\tilde \rho_\epsilon)}.
\end{align*}
Thus $\mathcal G_\epsilon$ is a contraction so long as $\epsilon$ is sufficiently small.
\end{proof}

\begin{prop}
\label{prop:phi-theta}
We may choose $\epsilon_*$ sufficiently small that if $0<\epsilon<\epsilon_*$ then $\phi_\epsilon - \theta_\epsilon \in X_\epsilon$.
In particular, we have 
\begin{equation}
\label{phi-theta-estimate}
\|\phi_\epsilon - \theta_\epsilon\|_{C^{k,\alpha}_\delta (M_\epsilon;\tilde\rho_\epsilon)}\lesssim \epsilon,
\qquad \delta \in\{0,2\}.
\end{equation}
\end{prop}

\begin{proof}
Lemma \ref{lemma:contraction} shows that if $\epsilon_*$ is sufficiently small then for $0<\epsilon< \epsilon_*$ the map $\mathcal G_\epsilon$ has a unique fixed point $u_\epsilon\in X_\epsilon \subset C^{k,\alpha}_2(M_\epsilon)$.
Since $\mathcal G_\epsilon(u_\epsilon) = u_\epsilon$ we see from \eqref{define-Q} that $\mathcal L_\epsilon(\theta_\epsilon + u_\epsilon)=0$ and thus $\theta_\epsilon + u_\epsilon$ is a solution to the Lichnerowicz equation \eqref{lichnero}.
By Proposition \ref{prop:approx-solution} we have $\theta_\epsilon - 1\in C^{k,\alpha}_1(M_\epsilon)$.
Thus, since $u_\epsilon \in X_\epsilon \subset C^{k,\alpha}_1(M_\epsilon)$, we have $(\theta_\epsilon + u_\epsilon) - 1\in C^{k,\alpha}_1(M_\epsilon)$.
Furthermore, for sufficiently small $\epsilon$ we have $\theta_\epsilon + u_\epsilon>0$.
But from \cite[Proposition 6.4]{WAH} we have that $\phi_\epsilon$ is the unique positive solution to \eqref{lichnero} such that $\phi_\epsilon - 1\in C^{k,\alpha}_1(M_\epsilon)$.
Thus we have $\phi_\epsilon = \theta_\epsilon + u_\epsilon$.
In particular, $\phi_\epsilon - \theta_\epsilon= u_\epsilon \in X_\epsilon$, which immediately implies \eqref{phi-theta-estimate}.
\end{proof}

\begin{proof}[Proof of Proposition \ref{prop:main-phi-estimate}]
We establish the estimates \eqref{main-phi-estimate} by writing
\begin{equation*}
\phi_\epsilon - 1 = (\theta_\epsilon - 1) + (\phi_\epsilon - \theta_\epsilon).
\end{equation*}
We estimate $\theta_\epsilon - 1$ using \eqref{improved-approximately-one-ext} and \eqref{improved-approximately-one-neck}.
We estimate $\phi_\epsilon - \theta_\epsilon$ using \eqref{phi-theta-estimate}, together with Proposition \ref{prop:basic-inclusions}\eqref{curly-to-weighted}.
\end{proof}

\section{Proof of Theorem \ref{main}}
\label{sec:proof}

We now complete the proof of our main theorem.
We assume \eqref{epsilon-estimate}, and that  $\epsilon<\epsilon_*$ as in Proposition \ref{prop:main-phi-estimate}.
From Proposition \ref{prop:WAH} we have $\overline\lambda_\epsilon = \rho_\epsilon^2 \lambda_\epsilon \in \mathscr C^{k,\alpha;2}(M_\epsilon)$ with $|d\rho_\epsilon|_{\overline\lambda_\epsilon}^2 =1$ along $\partial M_\epsilon$,  and from Lemmas \ref{lemma:nu-ext-reg} and \ref{lemma:nu-neck-reg} we have $$\nu_\epsilon^\text{ext} + \nu_\epsilon^\text{neck}\in C^{k-1,\alpha}_2(M_\epsilon).$$
Thus the results of \cite{AHEM} imply that 
\begin{equation*}
g_\epsilon = \phi_\epsilon^4 \lambda_\epsilon
\quad\text{ and }\quad
\Sigma_\epsilon = \phi_\epsilon^{-2} \sigma_\epsilon
\end{equation*}
constitutes an appropriate seed data set so that the conformal method produces a CMCSF hyperboloidal data set on $M_\epsilon$.
It remains to verify the convergence statements \eqref{ext-convergence} and \eqref{neck-convergence}.

We first consider the metric $g_\epsilon$.
In the exterior region $E_c$ we have 
\begin{align*}
\rho^2 (\iota_\epsilon^*g_\epsilon - g)
&= \iota_\epsilon^*\left( \rho_\epsilon^2 (g_\epsilon - \lambda_\epsilon)\right)
\\
&= \iota_\epsilon^* \left( (\phi_\epsilon^4-1) \rho_\epsilon^2 \lambda_\epsilon\right)
\\
&= \iota_\epsilon^* (\phi_\epsilon^4-1) \rho^2 g.
\end{align*}
By hypothesis we have $\overline g = \rho^2 g \in \mathscr C^{k,\alpha;2}(M)$. Thus \eqref{cor:phi} implies that
\begin{align*}
\| \rho^2 (\iota_\epsilon^*g_\epsilon - g)\|_{\mathscr C^{k,\alpha;2}(E_c)}
&\lesssim \|\iota_\epsilon^* (\phi_\epsilon^4-1)\|_{\mathscr C^{k,\alpha;2}(E_c)} 
\|\rho^2 g \|_{\mathscr C^{k,\alpha;2}(M)}
\\
&\lesssim \epsilon.
\end{align*}
In the neck region $A_c$ we write
\begin{equation*}
\breve \rho^2 (\Psi_\epsilon^*g_\epsilon - \breve g)
=
\Psi_\epsilon^*(\phi_\epsilon^4 - 1) \breve\rho^2 \Psi_\epsilon^*\lambda_\epsilon 
+
\breve\rho^2\left(\Psi_\epsilon^*\lambda_\epsilon - \breve g \right).
\end{equation*}
Applying \eqref{cor:phi} and Proposition \ref{prop:WAH}, and using \eqref{rho-tilde-equivalence}, we have
\begin{align*}
\| \breve \rho^2 (\Psi_\epsilon^*g_\epsilon - \breve g)\|_{\mathscr C^{k,\alpha;2}(A_c)}
&\lesssim 
\|\Psi_\epsilon^*(\phi_\epsilon^4 - 1)\|_{\mathscr C^{k,\alpha;2}(A_c)}
\|y^2\Psi_\epsilon^* \lambda_\epsilon \|_{\mathscr C^{k,\alpha;2}(A_c)}
\\
&\qquad+
\| y^2( \Psi_\epsilon^*\lambda_\epsilon - \breve g)\|_{\mathscr C^{k,\alpha;2}(A_c)}
\\
&\lesssim \epsilon.
\end{align*}

We now turn attention to the tensor $\Sigma_\epsilon$.
In the exterior region $E_c$ we have
\begin{equation*}
\rho (\iota_\epsilon^*\Sigma_\epsilon - \Sigma)
=
\iota_\epsilon^*(\phi_\epsilon^{-2}-1) \rho\iota_\epsilon^*\sigma_\epsilon 
+
\rho \left( \iota_\epsilon^*\sigma_\epsilon - \Sigma\right).
\end{equation*}
Using Proposition \ref{estimate:sigma} we have
\begin{equation*}
\| \rho\iota_\epsilon^*\sigma_\epsilon\|_{\mathscr C^{k-1,\alpha;1}(E_c)}
\leq
\| \rho(\iota_\epsilon^*\sigma_\epsilon - \Sigma)\|_{\mathscr C^{k-1,\alpha;1}(E_c)}
+
\| \rho\Sigma\|_{\mathscr C^{k-1,\alpha;1}(E_c)}
\lesssim 1.
\end{equation*}
Thus applying \eqref{cor:phi} and Proposition \ref{estimate:sigma} we have
\begin{align*}
\| \rho (\iota_\epsilon^*\Sigma_\epsilon - \Sigma)\|_{\mathscr C^{k-1,\alpha;1}(E_c)}
&\lesssim 
\|\iota_\epsilon^*(\phi_\epsilon^{-2}-1)\|_{\mathscr C^{k,\alpha;2}(E_c)}
\| \rho\iota_\epsilon^*\sigma_\epsilon\|_{\mathscr C^{k-1,\alpha;1}(E_c)}
\\ &\qquad+ 
\| \rho(\iota_\epsilon^*\sigma_\epsilon - \Sigma)\|_{\mathscr C^{k-1,\alpha;1}(E_c)}
\\
& \lesssim \epsilon.
\end{align*}
In the neck region $A_c$, the analogous decomposition yields
\begin{align*}
\| \breve \rho (\Psi_\epsilon^*\Sigma_\epsilon - 0)\|_{\mathscr C^{k-1,\alpha;1}(A_c)}
&\lesssim
\|\Psi_\epsilon^*(\phi_\epsilon^{-2}-1)\|_{\mathscr C^{k,\alpha;2}(A_c)}
\| \rho\Psi_\epsilon^*\sigma_\epsilon\|_{\mathscr C^{k-1,\alpha;1}(A_c)}
\\ &\qquad+ 
\| \rho\Psi_\epsilon^*\sigma_\epsilon\|_{\mathscr C^{k-1,\alpha;1}(A_c)}
\\
&\lesssim \epsilon.
\end{align*}
This concludes the proof of the main theorem.\hfill \qed

%%%%%%%%%%%%%%%%%%%%
\appendix

\section{Uniform invertibility for elliptic operators}
\label{uniforminvertibility}
In this appendix we study the vector Laplace operator $L_{\lambda_\epsilon}$ defined in \eqref{vl-eqn} and  the linearized Lichnerowicz operator 
\begin{equation}
\label{LinLich-1}
\mathcal P_\epsilon[1] = \Delta_{\lambda_\epsilon} - \frac18\left(\R[\lambda_\epsilon] + 7 |\sigma_\epsilon|^2_{\lambda_\epsilon} + 30 \right)
\end{equation}
 given by \eqref{linear-lich} in the special case $\theta=1$.
We obtain uniform invertibility of these operators in the following sense.

\begin{prop}
\label{invertibility:main}
Let $\lambda_\epsilon$ be the metrics constructed in \eqref{def-g}. 
For each $\delta\in [0,3)$ there exists a constant $C_\delta$, independent of $\epsilon$, such that:

\begin{enumerate}
\item $L_{\lambda_\epsilon}\colon C^{k,\alpha}_{\delta}(M_\epsilon)\to C^{k-2,\alpha}_{\delta}(M_\epsilon)$ is invertible with
\begin{equation*}
\| X\|_{C^{k,\alpha}_{\delta}(M_\epsilon;\tilde\rho_\epsilon)} \leq C_\delta \|L_{\lambda_\epsilon} X\|_{C^{k-2,\alpha}_{\delta}(M_\epsilon;\tilde\rho_\epsilon)}
\end{equation*}
for all vector fields $X\in C^{k,\alpha}_{\delta}(M_\epsilon)$,

\item $\mathcal{P}_\epsilon[1]\colon C^{k,\alpha}_{\delta}(M_\epsilon)\to C^{k-2,\alpha}_{\delta}(M_\epsilon)$ is invertible with
\begin{equation*}
\| u \|_{C^{k,\alpha}_{\delta}(M_\epsilon;\tilde\rho_\epsilon)} \leq C_\delta \| \mathcal{P}_\epsilon[1] u \|_{C^{k-2,\alpha}_{\delta}(M_\epsilon;\tilde\rho_\epsilon)}
\end{equation*}
for all functions $u\in C^{k,\alpha}_{\delta}(M_\epsilon)$.
\end{enumerate}

\end{prop}
Theorem 1.6 of \cite{WAH} implies that $L_{\lambda_\epsilon}$ and $\mathcal P_\epsilon[1]$ are Fredholm of index zero; see also \cite{Lee-FredholmOperators}.
Thus Proposition \ref{invertibility:main} is an immediate consequence of the elliptic regularity estimates in Proposition \ref{prop:operator-estimates} and of the following lemma, which controls the kernels of the operators, and is proven in section \ref{subsec:proof-of-kernel-lemma} below.

\begin{lemma}
\label{lemma:trivial-kernel}
For each  $\delta\in [0,3)$ there exists $C_\delta$, independent of $\epsilon$, such that:
\begin{enumerate}
\item 
$
\| X\|_{C^0_\delta(M_\epsilon;\tilde\rho_\epsilon)}
\leq C_\delta
\| L_{\lambda_\epsilon} X\|_{C^0_\delta(M_\epsilon;\tilde\rho_\epsilon)}
$
for all $X\in C^{2,\alpha}_\delta(M_\epsilon)$,

\item  
$
\| u \|_{C^0_\delta(M_\epsilon;\tilde\rho_\epsilon)}
\leq C_\delta
\| \mathcal P_\epsilon[1] u \|_{C^0_\delta(M_\epsilon;\tilde\rho_\epsilon)}
$
for all $u\in C^{2,\alpha}_\delta(M_\epsilon)$.
\end{enumerate}
\end{lemma}

Prior to proving Lemma \ref{lemma:trivial-kernel}, we introduce a general framework for blowup analysis and establish some results concerning the kernels of model operators.

\subsection{Exhaustions of weighted Riemannian manifolds}
\label{subsec:exhaustion}
Let $(M _*, g _*)$ be a Riemannian manifold. We say that a sequence of Riemannian manifolds \Defn{$(M_j, g_j)$ forms an exhaustion of $(M _*, g _*)$} if 
\begin{itemize}
\item $M_j$ are non-empty precompact open subsets of $M _*$,

\item $M_1\subseteq \overline{M}_1\subseteq M_2\subseteq \overline{M}_2 \subseteq M_3 \subseteq \cdots$ ,

\item $\bigcup_{j=1}^\infty M_j=M _*$, and

\item $\|g_j-g _*\|_{C^2(K, g_*)}\to 0$ on each precompact set $K\subset M _*$.
\end{itemize}
If in addition we have continuous functions $w_j\colon M_j\to (0,\infty)$ and $w _*\colon M _*\to (0,\infty)$ such that $\|w_j-w_*\|_{C^0(K)}\to 0$ on each precompact set $K\subset M _*$, then we say that \Defn{$(M_j, g_j, w_j)$ forms an exhaustion of $(M _*, g _*, w _*)$}.

We now give a definition of convergence for linear differential operators.
Let $(M_j, g_j)$ be an exhaustion of $(M _*, g _*)$. Consider a second order linear differential operator $P_*$ acting on sections of some tensor bundle over $M_*$, and operators $P_j$ acting on the restriction of that bundle to $M_j$.
We write $P_j=A_j\nabla^2+B_j\nabla+C_j$ where $A_j, B_j, C_j$ are appropriate bundle homomorphisms and where $\nabla$ is the connection associated to $g_j$.
Similarly, we write $P_*=A _*\nabla^2+B _*\nabla+C _*$. 
We say that \Defn{$P_j$ converges to $P_*$}, and write $P_j\to P_*$, if  
$$\|A_j-A _* \|_{C^2(K, g_*)}+ \|B_j-B _* \|_{C^1(K, g_*)}+ \|C_j-C _* \|_{C^0(K, g_*)}\to 0$$
on each precompact $K\subset M _*$.

Clearly, if $P_j\to P _*$ then for each precompact $K$ and each smooth tensor field $\eta$ we have 
$$\|P_j \eta-P _* \eta\|_{C^0(K, g_*)}\to 0 
\quad\text{ and }\quad
 \|P_j^\dagger \eta-P _*^\dagger \eta\|_{C^0(K, g_*)}\to 0,$$
where $P_j^\dagger$ and $P_*^\dagger$ denote the formal adjoints of $P_j$ and $P_*$, respectively. 
If in addition the operators $P_j$ and $P _*$ are elliptic then the  constants in the interior elliptic regularity estimates can be chosen independently of (sufficiently large) $j$. Finally, the reader should notice that if $(M_j, g_j)$ is an exhaustion of $(M _*, g _*)$ then any family of second order geometric operators $P_j=P[g_j]$ and $P _*=P[g _*]$ satisfies $P_j\to P _*$.

\begin{prop}\label{blowup:BigMamma}
Let $(M_j, g_j, w_j)$ be an exhaustion of $(M _*, g _*, w _*)$, and let $P_j$ and $P _*$ be second order elliptic linear differential operators on $(M_j, g_j)$ and $(M _*, g _*)$ with $P_j\to P _*$. 
Suppose also that there exists points $q_j\in M_1$ converging to $q_*\in M_1$ with respect to $g_*$, a sequence of tensor fields $u_j\in C^2(M_j)$, and constants $c,C>0$ such that: 
\begin{enumerate}

\item 
\label{B-assume-below}
for all $j$ we have $\big(w_j^{-1}|u_j|_{g_j}\big)\Big|_{q_j}\ge c$; 

\item 
\label{B-assume-above}
for all $j$ we have $\displaystyle{\sup_{M_j}}\, w_j^{-1}|u_j|_{g_j}\le C$;

\item 
\label{B-assume-operator}
we have $\displaystyle{\sup_{M_j}}\, w_j^{-1}|P_ju_j|_{g_j}\to 0$ as $j\to \infty$.
\end{enumerate}
Then there is a non-zero tensor field $u_*\in C^{0}(M_*, g_*)$ and a subsequence $\{u_{j_n}\}$ such that 
\begin{itemize}
\item $u_{j_n}\to u_*$ uniformly on compact sets;

\item $\displaystyle{\sup_{M}}\, w _*^{-1}|u_*|_{g _*}<\infty$;

\item $P_* u_*=0$ in the weak sense.
\end{itemize}
\end{prop}

\begin{proof}
Fix $p>\dim(M_*)$ so that the Sobolev space $H^{1,p}(M_j, g_*)$ embeds continuously into $C^0(M_j, g_*)$ for each $j$.

We now describe a process for extracting a subsequence of $\{u_j\}$ that we use iteratively in order to produce the desired subsequence via a diagonal argument.
Given the sequence $\{u_j\}$ and the sets $M_1\subset\overline M_1\subset M_2$, we extract a subsequence $u_{j_n,1}$ as follows.
Our assumptions imply that for sufficiently large $j$ we have
$$
|u_j|_{g _j}\le 2Cw_*,
\quad
 |P_j u_j|_{g_j}\le Cw_*
 \quad \text{ on }M_2.
$$
As the volumes $\vol_{g_j}(M_2)$ are uniformly bounded  for $j>2$, we have that the Sobolev norms $\|u_j\|_{H^{0,p}(M_2,g_j)}$ and $\|P_ju_j\|_{H^{0,p}(M_2,g_j)}$ are bounded uniformly. 
Since the assumption $P_j\to P_*$ implies 
$$\|u_j\|_{H^{2,p}(M_1,g_j)}\le C'\left(\|P_ju_j\|_{H^{0,p}(M_2,g_j)}+\|u_j\|_{H^{0,p}(M_2,g_j)}\right)$$
for some constant $C'$ independent of $j$, we have that $\|u_j\|_{H^{2,p}(M_1,g_j)}$ are bounded, and thus so are $\|u_j\|_{H^{2,p}(M_1,g_*)}$.
Applying Rellich's lemma yields a subsequence $\{u_{j_n,1}\}$ that converges in $H^{1,p}(M_1, g_*)$ to some function $u_1$. 
Since $p$ has been chosen such that $H^{1,p}(M_1,g_*)\subset C^0(M_1)$, it follows that we have uniform pointwise convergence 
$$u_{j_n, 1}\to u_{1}\quad\text{ in } C^0(M_1,g_*).$$
Furthermore,  assumptions \eqref{B-assume-below} and \eqref{B-assume-above} imply that  
$$|u_{1}(q_*)|_{g_*}\geq \frac{c}{2},
\quad
|u_{1}|_{g_*}\le 2C w_* 
\quad\text{ on } M_1.
$$

The process that produces the subsequence $\{u_{j_n,1}\}$ from the sequence $\{u_j\}$ and the sets $M_1 \subset \overline M_1 \subset M_2$ is now applied iteratively.
For example, applying this process to the sequence $\{u_{j_n, 1}\}$ and the sets $M_2 \subset \overline M_2 \subset M_3$ gives rise to the subsequence $\{ u_{j_n, 2}\}$ of $\{u_{j_n,1}\}$  that converges  in $C^0(M_2, g_*)$ to some limit $u_2$.
Since $u_{j_n,1} \to u_1$ in $C^0(M_1, g_*)$, we see that the function $u_2$ is a continuous extension of $u_1$ to the domain $M_2$.
Furthermore, we have that
$$|u_{2}(q_*)|_{g_*}\geq \frac{c}{2},
\quad
|u_{2}|_{g_*}\le 2C w_* 
\quad\text{ on } M_2.
$$

Repeating this process inductively we obtain subsequences $\{u_{j_n,l}\}$ of $\{u_j\}$ and limiting functions $u_l\in C^0(M_l)$ such that 
\begin{equation*}
u_{j_n, l} \to u_l \quad \text{ in }C^0(M_l, g_*)
\end{equation*}
as $n\to\infty$.
Consequently, the diagonal sequence $\{u_{j_n,n}\}$ is uniformly convergent on every compact subset of $M_*$ to a limit $u_*\in C^0(M_*, g_*)$.
Furthermore, we have
\begin{equation*}
|u_{*}(q_*)|_{g_*}\geq \frac{c}{2}, 
\quad\text{ and }\quad|u_*|_{g_*}\le 2C w_* \quad\text{ on }M_*.
\end{equation*}
For the remainder of the proof we denote the subsequence $\{u_{j_n, n}\}$ by $\{u_{j_n}\}$.

We now show that $P_*u_*=0$ weakly. 
Consider a smooth tensor field $\eta$ supported on some $\Omega\subseteq \overline{\Omega}\subseteq M_*$, where $\overline\Omega$ is compact. Since $P_{j_n}\to P_*$ and $g_{j_n}\to g_*$ we have
\begin{align*}
\left|\int_{M_*} \langle P_*^\dagger \eta, u_*\rangle_{g_*} dV_{g_*}\right|
&=\lim_{n\to \infty} \left|\int_{\Omega} \langle P_{j_n}^\dagger \eta, u_{j_n}\rangle_{g_{j_n}} dV_{g_{j_n}}\right|
\\
&\le \lim_{n\to \infty} \int_{\Omega} \left|\langle \eta, P_{j_n} u_{j_n}\rangle_{g_{j_n}}\right| dV_{g_{j_n}}
\\
&\le \|\eta\|_{C^0(\Omega, g_*)} \mathrm{Vol}_{g_*}(\Omega) \cdot \lim_{n\to \infty} \|P_{j_n}u_{j_n}\|_{C^0(\Omega, g_{j_n})}.
\end{align*}
It follows from our assumptions that $\sup_\Omega w_{j_n}^{-1}|P_{j_n}u_{j_n}|_{g_{j_n}}\to 0$. 
As the functions $w_{j_n}$ converge uniformly to the positive function $w_*$ on the precompact set $\Omega$, they are uniformly bounded from above and below on $\Omega$.
Thus $\|P_{j_n}u_{j_n}\|_{C^0(\Omega, g_{j_n})}\to 0$ and hence 
$$\int_{M_*} \langle P_*^\dagger \eta, u_*\rangle_{g_*} dV_{g_*}=0.$$ 
Therefore $P_* u_*=0$ weakly. 
\end{proof}

\subsection{Invertibility of model operators}
Our blowup analysis uses the mapping properties of elliptic geometric operators defined using one of two model CMCSF hyperboloidal initial data sets: the data assumed in the main theorem, given by $(g,\Sigma)$ on $M$, which serves as a model away from the gluing region, and the data given by $(\breve g,0)$ on $\mathbb H^3$, which serves as a model in the gluing region.
In the first case, our aim is to establish the injectivity of the vector Laplace operator $L_g$ and of the operator $\mathcal P_0$ given by
\begin{align*}
\mathcal P_0u &= \Delta_gu -\frac18\left(\R[g] + 7|\Sigma|^2_g + 30\right) u\\
&= \Delta_g u - (3+|\Sigma|^2_g),
\end{align*}
where we have used \eqref{CMC-constraints}.
The operator $\mathcal P_0$ serves as a model for the linearization $\mathcal P_\epsilon[1]$ of the Lichnerowicz operator about the function $1$.
In the second case, we establish injectivity of the analogous operators defined by $(\breve g, 0)$.

First we consider the case of the data assumed in the main theorem.
\begin{prop}
\label{prop:inject-on-M}
Let $(g,\Sigma)$ be initial data on $M$ as in Theorem \ref{main} and suppose $|1-\delta|<2$.
\begin{enumerate}
\item If a continuous vector field $X$ on $M$ satisfies $|X|_g\leq C\rho^\delta$ for some constant $C$ and if $L_gX =0$, then $X=0$.

\item If a continuous function $u$ on $M$ satisfies $|u|\leq C\rho^\delta$ for some constant $C$ and if $\mathcal P_0u=0$, then $u=0$.
\end{enumerate}
\end{prop}

\begin{proof}
For the first claim, we note that $L_g$ is an elliptic geometric operator. 
Thus from the elliptic regularity results in \cite[Lemma 5.1]{WAH} we have $X\in C^{k,\alpha}_\delta(M)$.
From Proposition 6.3 of \cite{AHEM} we have that
\begin{equation*}
L_g \colon C^{k,\alpha}_\delta(M) \to C^{k-2,\alpha}_\delta(M)
\end{equation*}
is invertible, and thus $X=0$.

For the second claim, we note that $\Delta_g -3$ is an elliptic geometric operator.
Since $|\Sigma|^2_g\in C^{k-1,\alpha}_2(M)$, 
adding $- |\Sigma|^2_g u$ to the lower order term does not affect the arguments leading to elliptic regularity results for $\mathcal P_0$; see \cite[Lemma 4.8]{Lee-FredholmOperators} and \cite[Lemma 5.1]{WAH}.
(Note that the sign convention for the Laplacian $\Delta_g$ in \cite{Lee-FredholmOperators} is the opposite of the one used here.)
Thus $u\in C^{k,\alpha}_\delta(M)$.
Since Proposition 6.5 of \cite{AHEM} implies that
\begin{equation*}
\mathcal P_0 \colon C^{k,\alpha}_\delta(M) \to C^{k-2,\alpha}_\delta(M)
\end{equation*}
is invertible, we conclude that $u=0$.
\end{proof}

We now turn to the model of hyperbolic space.
As in section \ref{sec:hyperbolic-space} we use the coordinates $(x,y)$ on the half-space model of hyperbolic space and write $r^2 = |x|^2 + y^2$.
Recall also the function $\breve\rho$ defined in \eqref{breve-rho} and the function $F$ described in Proposition \ref{lemma:F}.
It is established in \cite[Theorem 5.9]{Lee-FredholmOperators} that any self-adjoint elliptic geometric operator $\breve P$ on hyperbolic space is an isomorphism
\begin{equation*}
\breve P\colon C^{k,\alpha}_\delta(\mathbb H^3) \to C^{k-2,\alpha}_\delta(\mathbb H^3)
\end{equation*}
provided $|\delta -1|<R$, where $R$ is the indicial radius of the operator $\breve P$.
In particular, this applies to the vector Laplace operator $L_{\breve g}$ and to the  operator $\Delta_{\breve g} - 3$ for $|\delta -1|<2$.

The isomorphism property of $\breve P$, together with interior elliptic regularity, implies that any continuous tensor field $v\in \ker\breve P$ with $|v|_{\breve g}\leq C \breve\rho^\delta$ must in fact vanish.
In our blowup analysis we require a slight strengthening of this statement that makes use of the functions $y$ and $yF$ on the half-space model of hyperbolic space.
The argument we present is a generalization of the proof of Proposition 13 and Corollary 14 in \cite{ILS-Gluing}.

\begin{prop}
\label{prop:half-space-kernel}
Let $\breve P$ be a self-adjoint elliptic geometric operator on $\mathbb H$ with indicial radius $R>0$.
Suppose that for some $\delta$ satisfying $|1-\delta|<R$ there exists $v\in \ker \breve P$ satisfying either $|v|_{\breve g} \leq C(yF)^\delta$ or $|v|_{\breve g} \leq Cy^\delta$ .
Then $v$ is identically zero.
\end{prop} 

\begin{proof}
We argue by contradiction and consider first the case that there exists a nonzero tensor field $v\in \ker \breve P$ and constants $C,\delta\in \mathbb R$ such that $|\delta -1|<R$ and $|v|_{\breve g}\leq C (yF)^\delta$.
Let $r_0>0$ be such that $v$ does not vanish identically on the set where $r<r_0$.
From Proposition \ref{lemma:F} we have
\begin{gather}
\label{v-est-1}
|v|_{\breve g} \leq C y^\delta \quad \text{ where }r\geq r_0,
\\
\label{v-est-2}
|v|_{\breve g} \leq C \frac{y^\delta}{r^{2\delta}} \quad \text{ where }r\leq 3r_0,
\end{gather}
where here, and in the following, the value of $C$ may vary from line to line.

Let $\varphi_0\colon \mathbb R^2 \to [0,1]$ be a smooth cutoff function supported on $|x|\leq 2r_0$, where $x = (x^1,x^2)$ are Cartesian coordinates on $\mathbb R^2$, and define $\tilde v$ on $\mathbb H$ by
$$
\tilde{v}(x,y)=\int_{\mathbb{R}^2} v(x-\xi,y)\varphi_0(\xi)\,d\xi.
$$
Since  $v$ does not vanish on $r<2r_0$ by assumption, one can always choose $\varphi_0$ so that $\tilde{v}$ is not identically zero.
Differentiation under the integral sign shows that $\breve P \tilde v=0$. 

We claim that
\begin{equation}
\label{tv:est}
|\tilde{v}|_{\breve{g}}\leq C\left( y^{2-\delta}+y^{\delta}\right).
\end{equation}
On the region where $r\geq 3r_0$, the estimate \eqref{v-est-1} implies \eqref{tv:est} and thus we focus attention on the region where $r\leq 3r_0$.
There, \eqref{v-est-2} implies that
$$
|\tilde{v}|_{\breve{g}}\leq C \int_{|\xi|\le 2r_0} \frac{y^{\delta}}{(y^2+|x-\xi|^2)^{\delta}} d\xi.
$$
We now use the change of variables $\xi=x-y\zeta$ and observe that $r\le 3r_0$ and $|\xi|\le 2r_0$ implies $|\zeta|\leq \tfrac{5r_0}{y}$. 
Thus using polar coordinates yields 
$$
\begin{aligned}
|\tilde{v}|_{\breve g}
&\leq C y^{2-\delta}\int_{|x-y\zeta|\le 2r_0}\frac{d\zeta}{(1+|\zeta|^2)^\delta}
\\
&\leq C y^{2-\delta}\int_0^{5r_0/y} \frac{t}{(1+t^2)^\delta}\,dt.
\end{aligned}
$$
It follows from 
$$
\frac{1}{(1+t^2)^\delta}\le
\begin{cases} 
C & \text{\ \ for\ \ }t\le 1,\\
C t^{-2\delta} &  \text{\ \ for\ \ }t\ge 1  
\end{cases}
$$
that 
$$
\int_0^{5r_0/y} \frac{t}{(1+t^2)^\delta}\,dt\leq C( 1+y^{2\delta -2}).
$$
This completes the proof of \eqref{tv:est}.

We now define $u=\mathcal{I}^*\tilde{v}$, where $\mathcal I$ is the inversion operator defined in \eqref{define-inversion}.
Note that $\breve Pu=0$ due to $\mathcal{I}$-invariance of $\breve g$, and consequently the $\mathcal I$-invariance of $\breve P$. 
Choose $r_1>0$ so that $u\neq 0$ on $r\le r_1$. 
Choose also a smooth function $\varphi_1 \colon \mathbb{R}^2\to[0,1]$ supported on $|x|<2r_1$ such that the tensor field
$$
\tilde{u}(x,y)=\int_{\mathbb{R}^2} u(x-\xi,y)\varphi_1(\xi)\,d\xi
$$ 
is non-zero. 
Differentiation under the integral sign shows that $\breve P\tilde{u}=0$. 

We claim that 
\begin{equation}
\label{tu:est}
|\widetilde{u}|_{\breve{g}}\lesssim \breve{\rho}^{2-\delta}+\breve{\rho}^\delta.
\end{equation}
To this end, observe that \eqref{tv:est} implies 
\begin{equation*}
|u|_{\breve{g}}\lesssim \left(\frac{y}{r^2}\right)^{2-\delta}+\left(\frac{y}{r^2}\right)^{\delta}.
\end{equation*}
The same change of variables argument involved in the proof of  \eqref{tv:est} shows that 
$$
|\widetilde{u}|_{\breve{g}} \leq C y^{2-\delta}+y^{\delta} 
\quad\text{ on }\quad r\le 3r_1.
$$
In the region where $r\leq 3r_1$, we have $C^{-1} y\leq  \breve{\rho}\leq C y$, which implies the estimate \eqref{tu:est} in that region.

In the region where $r\ge 3r_1$ and $|\xi|\le 2r_1$ we have  
$$
C^{-1}r^2
\leq y^2+|x-\xi|^2
Cr^2.
$$ 
The estimate \eqref{tu:est} now follows from the fact that 
$C^{-1}\breve{\rho}\leq  \frac{y}{r^2}\leq C \breve{\rho}$ where $r\ge 3r_1$.

Since $\breve\rho \leq C$, the estimate \eqref{tu:est} implies that $|\tilde u|_{\breve g} \leq C \breve \rho^\nu$ where $\nu = \min{(2-\delta, \delta)}$.
Thus $\tilde u\in \ker \breve P$ and $\tilde u\in C^0_\nu(\mathbb H^3)$, where $|1-\nu|<R$.
The isomorphism property of $\breve P$ implies that $\tilde u =0$, which is the desired contradiction.

Suppose now that $\breve Pv=0$ and that $|v|_{\breve g} \leq C y^\delta$. If $\delta<0$, then the fact that $y\geq C\breve\rho$ implies $|v|_{\breve g} \leq C \breve\rho^\delta$ and hence the isomorphism property of $\breve P$ implies $v=0$.
If $\delta\geq 0$, then the fact that $F^\delta \geq C$ implies that $|v|_{\breve g} \leq C (yF)^\delta$ and thus $v=0$ by the previous argument.
\end{proof}

\subsection{Proof of Lemma \ref{lemma:trivial-kernel}}
\label{subsec:proof-of-kernel-lemma}
We now establish Lemma \ref{lemma:trivial-kernel}.
We present the argument for the estimate 
\begin{equation}
\label{Pe-kernel}
\| u\|_{C^0_\delta(M_\epsilon;\tilde\rho_\epsilon)}\lesssim \| \mathcal P_\epsilon[1]u\|_{C^0_\delta(M_\epsilon;\tilde\rho_\epsilon)}; 
\end{equation}
the estimate for the vector Laplace operator follows from analogous reasoning.

We argue by contradiction and assume that \eqref{Pe-kernel} does not hold.
Thus there exists $\delta\in [0,3)$ and a sequence $\epsilon_j\to 0$, together with functions $u_j\in C^{2,\alpha}_\delta(M_{\epsilon_j})$, such that 
\begin{equation}
\label{assume-norm-1}
\|u_j\|_{C^0_\delta(M_{\epsilon_j};\tilde\rho_{\epsilon_j})} =1
\end{equation}
 and
\begin{equation}
\label{assume-almost-kernel}
\| \mathcal P_{\epsilon_j}[1]u_j\|_{C^0_\delta(M_{\epsilon_j};\tilde\rho_{\epsilon_j})} \to 0.
\end{equation}
Hence there exist points $q_j\in M\setminus (U_{1,\epsilon_j} \cup U_{2,\epsilon_j})$, where we recall \eqref{preferred-half-ball}, such that at the point $\pi_{\epsilon_j}(q_j)\in M_{\epsilon_j}$ we have
\begin{equation}
\label{u-point-est}
\big(\tilde\rho_{\epsilon_j}^{-\delta} |u_j|\big)\Big|_{\pi_{\epsilon_j}(q_j)} \geq \frac12 .
\end{equation}
Passing to a subsequence if necessary, we may assume that $q_j \to q\in \overline M$.
We now consider several cases, depending on the location of $q\in \overline M$, obtaining a contradiction in each case.

\subsubsection*{Case 1: $q\in  M$}
In this case we define $(M_j, g_j, w_j)$ by setting
\begin{equation*}
M_j = \{p\in M\setminus (U_{1,\epsilon_j} \cup U_{2,\epsilon_j})\colon \rho_{\epsilon_j}(\pi_{\epsilon_j}(p))> \epsilon_j\},
\end{equation*}
$g_j = \pi_{\epsilon_j}^*\lambda_{\epsilon_j}$, and 
$w_j = \pi_{\epsilon_j}^*\tilde\rho_{\epsilon_j}^\delta = \pi_{\epsilon_j}^*({\rho_{\epsilon_j}}/{\omega_{\epsilon_j}} )^\delta;$
see \eqref{define-tilde-rho}.

Let $\psi\colon(0,\infty) \to (0,1]$ be the smooth cutoff function used in \eqref{define-w-epsilon}, and define the function $\omega_*\colon \overline M\to (0,\infty)$ by setting $\omega=1$ outside the domain of the preferred background coordinates $\Theta_i = (\theta, \rho)$ and by requiring that $\omega_* = \psi(|(\theta,\rho)|)$ in each background coordinate chart.
Set $w_* = (\rho/\omega_*)^\delta$ on $M$.

We claim that  $(M_j, g_j, w_j)$ forms an exhaustion of $(M, g, w_*)$.
The convergence of the metrics is immediate from the fact that $\iota_\epsilon^*\lambda_\epsilon = g$ on $E_c$; see Proposition \ref{prop:WAH}.
To see the convergence of the weight functions, we recall from \S\ref{subsec:functions-on-Me} that in preferred background coordinates $\Theta_i = (\theta, \rho)$ we have $\pi_\epsilon^*\omega_\epsilon = \psi(|(\theta,\rho)| + \epsilon^2|(\theta,\rho)|^{-1})$.
Thus $\pi_{\epsilon_j}^*\omega_{\epsilon_j} \to \omega$ uniformly on every precompact subset of $M$.
As $\pi_\epsilon ^*\rho_\epsilon = \rho$ on $E_c$ we see that $w_j\to w_*$ uniformly on precompact sets as well.

Let $v_j = \pi_{\epsilon_j}^* u_j$.
As $q_*\in M_j$ for sufficiently large $j$, we may pass to a subsequence to ensure that $q_j\in M_1$ for all $j$, and hence that $q_*\in M_1$.
From \eqref{u-point-est} and \eqref{assume-norm-1} we have
\begin{equation*}
\big(w_j^{-1} |v_j|\big)\Big|_{q_j}\geq \frac12
\quad\text{ and }\quad
\sup_{M_j} w_j^{-1}|u_j|\leq 1.
\end{equation*}
Furthermore, setting $P_j = \pi_{\epsilon
_j}^*\mathcal P_{\epsilon_j}[1]$, we have $\sup_{M_j} w_j^{-1}|P_j v_j| \to 0$.

The convergence of $\iota_\epsilon^*|\sigma_\epsilon|^2_{\lambda_\epsilon}$ to $|\Sigma|^2_g$ in the exterior region given by \eqref{norm-sigma-ext} implies that $\mathcal P_{\epsilon_j}[1]\to \mathcal P_0$ as described in \S\ref{subsec:exhaustion}.
Thus from Proposition \ref{blowup:BigMamma} there exists a nonzero function $v_*\in C^0(M)$ such that $|v_*|\leq C (\rho/\omega_*)^\delta$ and $\mathcal P_0 v_* = 0$.
Note that $\rho\leq C \omega_* \leq C.$
Thus, since $\delta \geq 0$, we have $|v_*|\leq C$.
But Proposition \ref{prop:inject-on-M} implies that the only continuous and bounded function in the kernel of $\mathcal P_0$ is the zero function, contradicting that $v_*$ is nonzero.

\subsubsection*{Case 2: $q\in \partial\overline M \setminus \{p_1, p_2\}$}
Let $\Theta = (\theta,\rho)$ be background coordinates on $M$ centered at $q$ as introduced in \S\ref{subsec:background-coord}.
After an affine change of coordinates we can arrange that at $q$ we have $\bar g_{ij}d\Theta^id\Theta^j = \delta_{ij}d\Theta^id\Theta^j$.
For $j$ sufficiently large, $q_j$ is contained in the domain $Z(q)$ of $\Theta$; denote $\Theta(q_j)$ by $(\hat\theta_j, \hat\rho_j)$.
Let $r_*>0$ be such that neither $p_1$ nor $p_2$ is contained in that part of $Z(q)$ where $|(\theta,\rho)|\leq r_*$.
Without loss of generality, we may assume that $|\hat\theta_j|< r_*/2$.

Set $M_j = \{ (x,y)\in \mathbb H \colon |(x,y)| < r_*/4\hat\rho_j, y>\hat\rho_j/2\}$ and use the background coordinates $(\theta,\rho)$ about $q$ to define $\Phi_j\colon M_j \to M$ by
\begin{equation*}
\Phi_j\colon (x,y) \mapsto (\theta,\rho) = (\hat\theta_j + \hat\rho_j x, \hat\rho_j y).
\end{equation*}
Note that $\Phi_j(0,1) = q_j$ and that for sufficiently small $c$ (and hence, in view of \eqref{epsilon-estimate}, sufficiently small $\epsilon_j$) the image of $\Phi_j$ is contained in the exterior region $E_c$.
Thus we may define $T_j \colon M_j \to M_{\epsilon_j}$ by $T_j = \iota_{\epsilon_j}\circ\Phi_j$.

Set $g_j = T_j^*\lambda_{\epsilon_j}$.
Let $\hat\omega_j = \omega_{\epsilon_j}(\iota_{\epsilon_j}(q_j))$ and, recalling from \eqref{define-tilde-rho} that $\tilde\rho_\epsilon = \rho_\epsilon/\omega_\epsilon$, 
define
\begin{equation}
\label{case2:weight}
w_j = \tilde\rho_{\epsilon_j}(\iota_{\epsilon_j}(q_j))^{-\delta} T_j^*\tilde\rho_{\epsilon_j}^\delta
= \left( \frac{\hat\rho_j}{\hat\omega_j}\right)^{-\delta}T_j^*\tilde\rho_{\epsilon_j}^\delta.
\end{equation}

We claim that $(M_j, g_j, w_j)$ forms an exhaustion of $(\mathbb H, \breve g, y^\delta)$.
Since $\hat\rho_j\to 0$ we have $\bigcup_{j} M_j = \mathbb H$.
To see that $g_j\to \breve g$, recall from Proposition \ref{prop:WAH} that  $\iota_\epsilon^*\lambda_\epsilon = g$ on $E_c$, and thus $g_j$ is simply the pullback of $g$ by $\Phi_j$.
In coordinates $\Theta = (\theta,\rho)$, we write $g = \rho^{-2} \overline g_{ij}(\theta,\rho) d\Theta^i d\Theta^j$.
Thus in coordinates $X = (x,y)$ we have $g_j = \Phi_j^* g= y^{-2} \overline g_{ij} (\hat\theta_j + \hat\rho_jx, \hat\rho_j y) dX^i dX^j$.
Thus on any precompact set $K\subset\mathbb H$ we have $g_j  \to  y^{-2} \overline g_{ij}(q)dX^idX^j = y^{-2} g_\text{E} = \breve g$ uniformly.
Finally, to see the convergence of the weight function, note that
$\hat\rho_j^{-1}T_j^*\rho_{\epsilon_j} = y$
and that on any precompact $K\subset \mathbb H$ we have 
$\hat\omega_j^{-1} T_j^*\omega_{\epsilon_j} \to 1$ uniformly.
Thus the claim is verified.

Define the functions $v_j$ on $M_j$ by $v_j = (\hat\rho_j/\hat\omega_j)^{-\delta}T_j^*u_j$.
Using \eqref{case2:weight} we see that $w_j^{-1} |v_j| = T_j^* (\tilde\rho_{\epsilon_j}^{-\delta} |u_j|)$.
Thus from \eqref{u-point-est} and \eqref{assume-norm-1}  we have 
$$
\big(w_j^{-1}|v_j| \big)\Big|_{(0,1)}
\geq \frac12
\quad\text{ and }\quad
\sup_{M_j} w_j^{-1} |v_j| \leq 1;
$$
hence assumptions \eqref{B-assume-below} and \eqref{B-assume-above} of Proposition \ref{blowup:BigMamma} hold.

Define the differential operator $P_j$ on $M_j$ by $P_j = T_j^*\mathcal P_{\epsilon_j}[1]$.
We claim that $P_j\to \Delta_{\breve g} - 3$.
To see this, note first that since $\iota_\epsilon^*\lambda_\epsilon = g$ in the exterior region $E_c$, applying the constraint equations \eqref{CMC-constraints} to \eqref{LinLich-1} yields
\begin{equation}
\begin{aligned}
\iota_\epsilon^*\mathcal P_{\epsilon_j}[1]
&= \Delta_g - 3 - \frac18|\Sigma|^2_g - \frac78\iota_\epsilon^*|\sigma_\epsilon|^2_{\lambda_\epsilon}
\\
&=
\Delta_{g} - 3 
- |\Sigma|^2_{g}
-\frac78\left( \iota_{\epsilon_j}^*|\sigma_{\epsilon_j}|^2_{\lambda_{\epsilon_j}} - |\Sigma|^2_g\right).
\end{aligned}
\end{equation}
By the hypotheses of Theorem \ref{main} we have that $|\overline\Sigma|^2_{\overline g}$  is bounded, and thus  for some constant $C$ we have $\Phi_j^*|\Sigma|^2_{g} =\Phi_j^*(\rho^2 |\overline\Sigma|^2_{\overline g})\leq C(\hat\rho_j y)^2$, which tends to zero uniformly on any precompact set.
Furthermore, from  \eqref{norm-sigma-ext} we have 
\begin{equation*}
\Phi_j^*\left| \iota_{\epsilon_j}^*|\sigma_{\epsilon_j}|^2_{\lambda_{\epsilon_j}} - |\Sigma|^2_g \right| \leq C (\hat\rho_j y)^2 \epsilon_j,
\end{equation*}
which also tends to zero.
Finally, since $g_j = \Phi_j^*g\to \breve g$ we have $\Delta_g\to \Delta_{\breve g}$, which establishes the claim.

From \eqref{assume-almost-kernel} we have $w_j^{-1}|P_j v_j| \to 0$.
Applying Proposition \ref{blowup:BigMamma} we obtain a continuous, nonzero function $v_*$ on $\mathbb H$ such that $|v_*| \leq C y^\delta$ and $\Delta_{\breve g} v_* - 3v_*=0$.
This, however, contradicts Proposition \ref{prop:half-space-kernel}.

\subsubsection*{Case 3: $q\in \{p_1, p_2\}$}
In this case we may assume, without loss of generality, that $q = p_1$ and, by passing to a subsequence if necessary, that the points $q_j$ are contained in the domain $Z(p_1)$ of the preferred background coordinates $(\theta,\rho)$ centered about $p_1$.
Let $(\hat\theta_j,\hat\rho_j)$ be the background coordinates of $q_j$.
Since $q_j\in M\setminus (U_{1,\epsilon_j} \cup U_{2,\epsilon_j})$ we have $|(\hat\theta_j,\hat\rho_j)| \geq \epsilon_j$. 
Since $|(\hat\theta_j,\hat\rho_j)| \to 0$ we may assume that $|\hat\theta_j|+\hat\rho_j<1/8$.

Below, we consider three sub-cases, depending on the nature of the convergence $(\hat\theta_j,\hat\rho_j)\to (0,0)$.
In each case we define nested precompact subsets $M_j \subset \mathbb H$ and maps $T_j\colon M_j \to M_{\epsilon_j}$.
We arrange $T_j$ so that the preferred background coordinate expression for $T_j(x,y)$ satisfies
\begin{equation}
\label{3-domain-ok}
8\epsilon_j^2 < |T_j(x,y)| < \frac18,
\end{equation}
which ensures that  $g_j = T_j^*\lambda_{\epsilon_j}$ is well defined.
We then show that $(M_j, g_j)$ forms an exhaustion of $(\mathbb H, \breve g)$, and that $P_j = T_j^*P_{\epsilon_j}[1]\to \Delta_{\breve g} - 3$.
Finally, in each case we construct a sequence of functions $v_j$ and weights $w_j$ satisfying the hypotheses of Proposition \ref{blowup:BigMamma}.
We thus obtain a nonzero limiting function $v_*$, from which we obtain a contradiction via Proposition \ref{prop:half-space-kernel}.

\subsubsection*{Case 3(a): Both $|\hat\theta_j|/\hat\rho_j$ and $|(\hat\theta_j,\hat\rho_j)|/\epsilon_j$ are bounded above.}
Thus there exists $C>1$ such that $|\hat\theta_j|\leq C\hat\rho_j$ and $|(\hat\theta_j,\hat\rho_j)|\leq C\epsilon_j$ for all $j$.
Thus $\hat\rho_j \leq |(\hat\theta_j,\hat\rho_j)| \leq |\hat\theta_j| + \hat\rho_j \leq 2C\hat\rho_j$.
Furthermore, since $\epsilon_j\leq |(\hat\theta_j,\hat\rho_j)|$ we have $\epsilon_j \leq 2C \hat \rho_j$ and $\hat\rho_j \leq |(\hat\theta_j, \hat\rho_j)| \leq C \epsilon_j$.
Combining these yields
\begin{equation}
\label{case3a-equivalences}
\frac{1}{2C}\hat\rho_j \leq |(\hat\theta_j,\hat\rho_j)| \leq 2C \hat\rho_j
\quad\text{ and }\quad
\frac{1}{2C}\epsilon_j\leq \hat\rho_j \leq 2C\epsilon_j.
\end{equation}
Let
\begin{equation*}
M_j = \left\{ (x,y)\in \mathbb H \colon |(x,y)|< \frac{1}{8\epsilon_j}, y>8\epsilon_j  \right\}
\end{equation*}
and define $T_j\colon M_j \to M_{\epsilon_j}$ by setting $T_j(x,y) = \Psi_{\epsilon_j}(x,y) = (\epsilon_j x, \epsilon_j y)$.
For each $(x,y)\in M_j$  we have, for sufficiently large $j$, that
\begin{equation*}
8\epsilon_j^2
<\epsilon_j y
\leq|(\epsilon_j x, \epsilon_j y)|
< \frac{1}{8}
\end{equation*}
and thus \eqref{3-domain-ok} holds.
Let 
\begin{equation*}
w = \left(\frac{yF}{2(r+ 1/r)}\right)^\delta.
\end{equation*}
Thus from \eqref{defn-of-rho-epsilon} and \eqref{omega-epsilon-cases} we have $w = T_j^*\tilde\rho_{\epsilon_j}^\delta$.
From Proposition \ref{prop:WAH} we have $g_j = \Psi_{\epsilon_j}^*\lambda_{\epsilon_j} \to \breve g$ on precompact sets of $\mathbb H$, and thus $(M_j, g_j, w)$ is an exhaustion of $(\mathbb H, \breve g, w)$.
Furthermore, it follows from \eqref{norm-sigma-neck} that $T_j^* |\sigma_{\epsilon_j}|^2_{\lambda_{\epsilon_j}} \to 0$ uniformly on precompact subsets of $\mathbb H$; hence $P_j = T_j^* \mathcal P_{\epsilon_j}[1] \to \Delta_{\breve g} - 3$.

Let $(\hat x_j,\hat y_j) = (\hat \theta_j/\epsilon_j, \hat \rho_j/\epsilon_j) \in M_j$ so that $T_j (\hat x_j, \hat y_j) = (\hat \theta_j, \hat \rho_j)$.
From \eqref{case3a-equivalences} the sequence $(\hat x_j, \hat y_j)$ is bounded, and $\hat y_j$ is bounded away from zero.
Thus by passing to a subsequence we have $(\hat x_j, \hat y_j) \to (\hat x_*, \hat y_*)$ with $\hat y_*>0$.

Set $v_j = T_j^* u_j$. By assumption we have 
\begin{equation*}
\big(w^{-1}|v_j|\big)\Big|_{(\hat x_j, \hat y_j)} 
= \big(\tilde\rho_{\epsilon_j}^{-\delta} |u_j|\big)\Big|_{\pi_{\epsilon_j}(q_j)} 
\geq \frac12
\end{equation*}
and
\begin{equation*}
\sup_{M_j} w^{-1}|v_j| 
=\sup_{M_j} T_j^*\big(\tilde\rho_{\epsilon_j}^{-\delta} |u_j|\big) \leq 1.
\end{equation*}
Furthermore
\begin{equation*}
\sup_{M_j}w^{-1}|P_j v_j|
= \sup_{M_j}T_j^*\big(\tilde\rho_{\epsilon_j}^{-\delta} |P_{\epsilon_j}[1]u_j|\big)
\leq \| \mathcal P_{\epsilon_j}[1]u_j\|_{C^0_\delta(M_{\epsilon_j};\tilde\rho_{\epsilon_j})}\to 0.
\end{equation*}
Thus the hypotheses of Proposition \ref{blowup:BigMamma} are satisfied and there exists a nonzero function $v_*$ on $\mathbb H$ with $\Delta_{\breve g}v_* - 3v_* =0$ and
\begin{equation*}
|v_*| \leq w = \left(\frac{yF}{2(r+ 1/r)}\right)^\delta.
\end{equation*}
As we are assuming $\delta \geq 0$ this implies $|v_*| \leq (yF)^\delta$ and the desired contradiction is obtained from Proposition \ref{prop:half-space-kernel}.

\subsubsection*{Case 3(b): $|\hat\theta_j|/\hat\rho_j$ is bounded above, but $|(\hat\theta_j,\hat\rho_j)|/\epsilon_j$ is not bounded above.}
In this case we may, after passing to a subsequence if necessary, suppose that $ |\hat\theta_j|\leq C\hat\rho_j$ for some constant $C>1$ and that $\hat r_j = |(\hat\theta_j,\hat\rho_j)|/\epsilon_j \to \infty$.
Thus
\begin{equation}
\label{3b-constraints}
\hat\rho_j
 \leq \epsilon_j \hat r_j
=|(\hat\theta_j,\hat\rho_j)| \leq 2C \hat\rho_j
\quad\text{ and }\quad
\frac{\hat\rho_j}{\epsilon_j}\to \infty.
\end{equation}

Let 
\begin{equation*}
M_j = \left\{ (x,y) \in \mathbb H \colon |(x,y)|<\frac{1}{16C\hat\rho_j}, y> \frac{8\epsilon_j}{\hat \rho_j} \right\}.
\end{equation*}
For sufficiently large $j$ we have $\epsilon_j / \hat\rho_j < 8C$ and thus $M_j\subset A_{\epsilon_j}$.
Hence we may define $T_j \colon M_j \to M_{\epsilon_j}$ by $T_j(x,y) = \Psi_{\epsilon_j}(\hat r_j x, \hat r_j y)$.
In preferred background coordinates $(\theta,\rho)$ about $p_1$ we have $T_j(x,y) = \left(\epsilon_j\hat r_jx, \epsilon_j\hat r_jy\right)$ and thus from \eqref{3b-constraints} we see that 
\begin{equation*}
8\epsilon_j^2
< 8 \epsilon_j
< \hat\rho_j y
\leq \epsilon_j\hat r_j y
\leq |T_j(x,y)|
\leq \epsilon_j\hat r_j |(x,y)| < \frac18;
\end{equation*}
thus \eqref{3-domain-ok} holds.

Let $(\hat x_j, \hat y_j) = (\epsilon_j \hat r_j)^{-1} (\hat\theta_j, \hat \rho_j)$ so that $T_j(\hat x_j, \hat y_j) = \pi_{\epsilon_j}(q_j)$.
By construction we have $|(\hat x_j, \hat y_j)| = 1$ and it follows from \eqref{3b-constraints} that $\hat y_j \geq 1/2C$.
Thus we may pass to a subsequence such that $(\hat x_j, \hat y_j) \to (\hat x_*, \hat y_*)$ with $|(\hat x_*, \hat y_*)| = 1$ and $y_*>0$.

Setting $g_j = T_j^*\lambda_{\epsilon_j}$ and $w_j = T_j^*\tilde\rho_{\epsilon_j}^{\delta}$, we claim that $(M_j, g_j, w_j)$ forms an exhaustion of $(\mathbb H, \breve g, (y/2r)^\delta)$, where as usual we write $r = |(x,y)|$.
To see this, first note that Proposition \ref{prop:WAH} implies that $\Psi_{\epsilon_j}^*\lambda_{\epsilon_j} \to \breve g$ uniformly on precompact sets.
Dilation by $\hat r_j$ is an isometry of hyperbolic space that preserves unweighted norms; see \eqref{scaling-of-norms}.
Thus $g_j\to \breve g$ uniformly on precompact sets.
Next observe from \eqref{define-tilde-rho}, while using \eqref{defn-of-rho-epsilon} and \eqref{omega-epsilon-cases}, that
\begin{equation*}
T_j^*\tilde\rho_{\epsilon_j} = \frac{y F(\hat r_j r)}{2\left( r + {1}/{\hat r_j^2 r}\right)}.
\end{equation*}
The hypotheses that define this case include $\hat r_j \to \infty$, while Proposition \ref{lemma:F} states that $F(\hat r_j r) = 1$ if $\hat r_j r >2$.
Thus we see that $T_j^*\tilde\rho_{\epsilon_j} \to y/2r$ uniformly on precompact subsets of $\mathbb H$ and the claim is established.

Set $v_j = T_j^* u_j$ and $P_j = T_j^*\mathcal P_{\epsilon_j}[1]$.
We now verify that the hypotheses of Proposition \ref{blowup:BigMamma} are satisfied.
The assumption \eqref{u-point-est} implies that
\begin{equation*}
\big(w_j^{-1} |v_j|\big)\Big|_{(\hat x_j, \hat y_j)} \geq \frac12
\end{equation*}
and thus hypothesis \eqref{B-assume-below} holds.
As the assumptions \eqref{assume-norm-1} and \eqref{assume-almost-kernel} imply that hypotheses \eqref{B-assume-above} and \eqref{B-assume-operator} hold, it remains to establish the convergence of the operators $P_j$.
Proposition \ref{prop:sigma-aux-estimates} implies that $T_j^* |\sigma_{\epsilon_j}|^2_{\lambda_{\epsilon_j}}\to 0$ uniformly on precompact subsets of $\mathbb H$.
Thus the convergence $g_j \to \breve g$ implies that $P_j \to \Delta_{\breve g} - 3$.

We now invoke Proposition \ref{blowup:BigMamma} to conclude that there exists a nonzero continuous function $v_*$ on $\mathbb H$ such that $|v_*|\leq C(y/r)^\delta \leq Cy^0$ and $\Delta_{\breve g} v_* - 3 v_* =0$.
Consequently, Proposition \ref{prop:half-space-kernel} yields a contradiction.

\subsubsection*{Case 3(c): $|\hat\theta_j|/\hat\rho_j$ is not bounded above.}
Passing to a subsequence we may assume that $ |\hat\theta_j|/\hat\rho_j  \to \infty$.
We may further assume that $\hat\rho_j / |\hat\theta_j| < 1/2$; when combined with the fact that $\epsilon_j \leq |(\hat\theta_j, \hat\rho_j)|$ we find that $|\hat\theta_j|\geq \epsilon_j/2$.

Let $(\hat x_j, \hat y_j) = \epsilon_j^{-1}(\hat\theta_j, \hat\rho_j)$ so that $\Psi_{\epsilon_j}(\hat x_j, \hat y_j) = \pi_{\epsilon_j}(q_j)$.
Set
\begin{equation*}
M_j = \left\{(x,y)\in \mathbb H \colon |(x,y)|< \frac{|\hat\theta_j|}{2\hat\rho_j},
y > \epsilon_j \right\}.
\end{equation*}
For $(x,y)\in M_j$ we may use $|\hat\theta_j| < 1/8$ to conclude that 
\begin{equation}
\label{3c-upper}
|(\hat x_j + \hat y_j x, \hat y_j y)|
\leq \frac{|\hat\theta_j|}{\epsilon_j} + \frac{\hat\rho_j}{\epsilon_j}|(x,y)|
< 2\frac{|\hat\theta_j|}{\epsilon_j}
<\frac{1}{8\epsilon_j}
\end{equation}
and use $\hat\rho_j |x| < |\hat\theta_j|/2$ to obtain
\begin{equation}
\label{3c-lower}
|(\hat x_j + \hat y_j x, \hat y_j y)|
\geq |\hat x_j + \hat y_j x|
= \frac{1}{\epsilon_j} |\hat\theta_j + \hat\rho_j x|
\geq \frac{|\hat\theta_j|}{2\epsilon_j}
> \frac14
> 8 \epsilon_j.
\end{equation}
Thus the map $\Phi_j \colon M_j \to A_{8\epsilon_j}$ given by $\Phi_j (x,y) = (\hat x_j + \hat y_j x, \hat y_j y)$ is well defined, and the map $T_j= \Psi_{\epsilon_j}\circ \Phi_j\colon M_j \to M_{\epsilon_j}$ satisfies \eqref{3-domain-ok}.
In preferred background coordinates $(\theta,\rho)$ about $p_1$ we have $T_j (x,y) = (\epsilon_j\hat x_j + \epsilon_j\hat y_j x, \epsilon_j\hat y_j y) = (\hat\theta_j + \hat\rho_j x, \hat\rho_j y)$ and thus $T_j(0,1) = \pi_{\epsilon_j}(q_j)$.

We now estimate $T_j^*\tilde\rho_{\epsilon_j}$ using \eqref{rho-tilde-equivalence}, which implies that
\begin{equation*}
T_j^*\tilde\rho_{\epsilon_j}
= \frac{\hat y_j y F(|\Phi_j(x,y)|)}{2\left(|\Phi_j(x,y)| + \frac{1}{|\Phi_j(x,y)|} \right)}.
\end{equation*}
The estimate \eqref{3c-lower} implies that $|\Phi_j(x,y)| > 1/4$ and thus from Proposition \ref{lemma:F} we have $F(|\Phi_j(x,y)|)$ uniformly bounded above and below.
The lower bound $|\Phi_j(x,y)| > 1/4$ furthermore implies that 
\begin{equation*}
|\Phi_j(x,y)| \leq |\Phi_j(x,y)| + \frac{1}{|\Phi_j(x,y)|} \leq 17 |\Phi_j(x,y)|.
\end{equation*}
As  \eqref{3c-upper} and \eqref{3c-lower} imply that 
\begin{equation*}
\frac{|\hat\theta_j|}{2\epsilon_j} 
\leq |\Phi_j(x,y)|
\leq 2\frac{|\hat\theta_j|}{\epsilon_j},
\end{equation*}
and as $\hat y_j = \hat \rho_j / \epsilon_j$, we find that 
\begin{equation}
\label{3c-weight}
\frac{1}{C }\frac{\hat\rho_j}{|\hat\theta_j|} y
\leq T_j^* \tilde\rho_{\epsilon_j} 
\leq C \frac{\hat\rho_j}{|\hat\theta_j|} y
\end{equation}
for some constant $C$.

Let $g_j = T_j^*\lambda_{\epsilon_j}$ and set $w_j = y^\delta$.
From Proposition \ref{prop:WAH} we see that $g_j \to \breve g$ on precompact sets and thus $(M_j, g_j, w_j)$ forms an exhaustion of $(\mathbb H, \breve g, y^\delta)$.
We seek to apply Proposition \ref{blowup:BigMamma} to the functions 
\begin{equation*}
v_j = \left(\frac{\hat\rho_j}{|\hat\theta_j|}\right)^{-\delta} T_j^* u_j
\end{equation*}
and operators $P_j = T_j^*\mathcal P_{\epsilon_j}[1]$.
Proposition \ref{prop:sigma-aux-estimates} implies that $T_j^* |\sigma_{\epsilon_j}|^2_{\lambda_{\epsilon_j}}\to 0$ uniformly on precompact subsets of $\mathbb H$.
Thus the convergence $g_j \to \breve g$ implies that $P_j \to \Delta_{\breve g} - 3$.
Thus by applying \eqref{3c-weight} to \eqref{assume-norm-1}, \eqref{assume-almost-kernel}, and \eqref{u-point-est} we have that the hypotheses of Proposition \ref{blowup:BigMamma} are satisfied.
The result is a nonzero function $v_*$ satisfying both $|v_*| \leq Cy^\delta$ and $\Delta_{\breve g} v - 3v=0$.
This, however, is in contradiction to Proposition \ref{prop:half-space-kernel}.

With all cases exhausted, the proof of Lemma \ref{lemma:trivial-kernel} is complete.

%--------------------------------
\bibliographystyle{plain}
\bibliography{SFGluing}

\end{document}